\documentclass[10pt]{article}
\usepackage{amsthm}
\usepackage{amsfonts, amssymb, amsmath, amscd, latexsym}
\usepackage{mathrsfs,epsfig}
\usepackage[latin1]{inputenc}
\usepackage[all]{xy}
\usepackage{cite}
\usepackage{subcaption}
\usepackage[dvipsnames]{xcolor}

\def\today{\ifcase\month\or
January\or February\or March\or April\or May\or June\or July \or
August\or September\or October\or November\or December\fi
\space\NNumber\day, \NNumber\year}
\numberwithin{equation}{section}

\def\eu{ {\, \textrm{\rm e} }}

\def\sbwd{\sigma^{\textrm{bwd}}}

\def\sfwd{\sigma^{\textrm{fwd}}}
\def\sbwd{\sigma^{\textrm{bwd}}}
\def\sTfwd{\Sigma^{\textrm{fwd}}}
\def\sTbwd{\Sigma^{\textrm{bwd}}}
\def\dist{{\mathscr{D}}}
\def\al{\alpha}

\def\ga{\vec{\gamma}}
\def\de{\delta}
\def\ep{\varepsilon}
\def\la{\lambda}

\def\Om{\Omega}
\def\na{\vec{\nabla}}

\def\N{{\mathbb N}}

\def\R{{\mathbb R}}

\def\I{{\mathbb I}}

\def\FF{{\mathcal F}}

\def\MM{{\mathcal M}}

\def\DDD{\mathscr{D}}
\def\TTT{\mathscr{T}}
\def\PPP{\mathscr{P}}

\def\l1{\Lambda^{1}}

 \def\PPs{\vec{\pi}_s}
 \def\PPu{\vec{\pi}_u}

\def\Q{\vec{Q}}
\def\P{\vec{P}}

\def\zz{\vec{\zeta}}
\def\xxi{\vec{\xi}}

\def\x{\vec{x}}
\def\y{\vec{y}}
\def\f{\vec{f}}
\def\g{\vec{g}}
\def\h{\vec{h}}

\def\gx{\boldsymbol{g_x}}
\newcommand{\lit}{{\lim_{t \to +\infty}}}
\newcommand{\lito}{{\lim_{t \to -\infty}}}

\def\Ae(t){\boldsymbol{A^\ep}}

\def\Xe{\boldsymbol{X^{\pm,\ep}}}

\def\Pe{\boldsymbol{P^{\pm,\ep}}}
\def\({\left(}
\def\){\right)}

\DeclareMathOperator{\tr}{tr}
\newcommand\bs[1]{{\boldsymbol{#1}}}
\newcommand\ev[1]{\textcolor{black}{#1}}

\def\und{\underline}
\def\ov{\overline}

\newcommand{\mat}[1]{\textcolor{black}{#1}} 

\newcommand\assump[1]{{\rm\textbf{#1}}}
\def\inn{\textrm{in}}
\def\outt{\textrm{out}}
\def\Di{\textrm{D}}
\newtheorem{theorem}{Theorem}[section]
\newtheorem{proposition}[theorem]{Proposition}
\newtheorem{lemma}[theorem]{Lemma}

\theoremstyle{definition}

\theoremstyle{remark}
\newtheorem{remark}[theorem]{Remark}

\newcommand\michal[1]{\textcolor{black}{#1}}

\begin{document}
\title{On the dynamics of non-autonomous systems in a~neighborhood of a~homoclinic trajectory}
\pagestyle{myheadings}
\date{}
\markboth{}{Non-autonomous homoclinic trajectories}

\author{	
A. Calamai\thanks{Dipartimento di Ingegneria Civile, Edile e Architettura,
Universit\`a Politecnica delle Marche, Via Brecce Bianche 1, 60131 Ancona -
Italy. Partially supported by G.N.A.M.P.A. - INdAM (Italy) and
PRIN 2022 - Progetti di Ricerca di rilevante Interesse Nazionale, \emph{Nonlinear differential
problems with applications to real phenomena} (Grant Number: 2022ZXZTN2).
},
M. Franca\thanks{Dipartimento di Scienze Matematiche, Universit\`a di Bologna, 40126 Bologna, Italy.
 Partially supported by G.N.A.M.P.A. - INdAM (Italy) and PRIN 2022 - Progetti di Ricerca di rilevante Interesse Nazionale, \emph{Nonlinear
 differential
problems with applications to real phenomena} (Grant Number: 2022ZXZTN2). },
M. Posp\' i\v sil\thanks{Department of Mathematical Analysis and Numerical Mathematics, Faculty of Mathematics, Physics and Informatics,
Comenius University in Bratislava,
Mlynsk\'a dolina, 842 48 Bratislava, Slovakia; Mathematical Institute, Slovak Academy of Sciences, \v Ste\-f\'a\-ni\-ko\-va 49, 814 73
Bratislava, Slovakia.
Partially supported by the Slovak Research and Development Agency under the Contract no. APVV-23-0039, and by the Grants VEGA 1/0084/23 and
VEGA-SAV 2/0062/24.}
}

\maketitle
\begin{abstract}
This article is devoted to the study of a $2$-dimensional piecewise smooth (but possibly) discontinuous dynamical system, subject
to a non-autonomous perturbation;
we assume that the unperturbed system admits a homoclinic trajectory $\ga(t)$. Our aim is to analyze the dynamics in a neighborhood of $\ga(t)$
as the perturbation is turned on, by defining a Poincar\'e map and evaluating fly time and space displacement  of trajectories
 performing a loop close to $\ga(t)$.

Besides their intrinsic mathematical interest, these results can be thought of as a first step in the analysis of several interesting problems, such as the stability of
a homoclinic trajectory of a non-autonomous ODE and a possible extension of Melnikov chaos to a discontinuous setting.
\end{abstract}

{\bf Keywords: piecewise smooth systems,   homoclinic trajectories, non-autonomous dynamical systems, exponential dichotomy,
non-autonomous perturbation}

{\bf 2020 AMS Subject Classification:  Primary   34C37, 34A36, 34D10, 37G20; Secondary   37C60. }

\vskip 24bp


 \section{Introduction}
This article is devoted to the study of a $2$-dimensional piecewise smooth (but possibly) discontinuous dynamical system, subject
to a non-autonomous perturbation,
namely
\begin{equation}\label{eq-smooth}
	\dot{\x}=\f(\x)+\ep\g(t,\x,\ep),
\end{equation}
where $\f$ and $\g$ are  bounded together with their derivatives up to the $r$-th order, $r>1$, and $\ep \ge 0$ is a small parameter.

In particular we assume that the origin $\vec 0= (0,0)$ is a critical point for \eqref{eq-smooth} for any $\ep$ and that there is a trajectory $\ga(t)$ homoclinic to $\vec 0$ when $\ep=0$.

 In this context, classical Melnikov theory allows us to define explicitly a function $\MM(\tau)$, the \textit{Melnikov function}, see \eqref{melni-disc}, such that the existence of non-degenerate zeros of $\MM(\tau)$  is a sufficient condition  for the persistence of the homoclinic trajectory, say $\x_b(t,\ep)$,
 and the insurgence of chaotic phenomena if we also assume that $\g$ is periodic in $t$.
 In fact, after the pioneering work of Melnikov \cite{Me}, in the smooth case, the problem of detecting the occurrence of chaotic solutions for non-autonomous dynamical systems is now well studied, see e.g.\ \cite{Pa84, GH, WigBook, SSTC}.
 In this paper we want to deepen our knowledge of the dynamics close to $\ga(t)$ when the perturbation is turned on.
  Namely, let $L^0$ be a curve transversal  to $  \{ \ga(t) \mid t \in \R \}$, our purpose is
 to define a Poincar\'e map from  a compact connected subset $A^{\textrm{fwd},+}(\tau)$ of $L^0$ back to $L^0$ and to evaluate space displacement
 and fly time when $\ep \ne 0$.

 Our work takes motivation and inspiration from the study of the asymptotic stability of periodic and homoclinic trajectories of an ordinary    differential equation (ODE).
 Borg, Hartman and Leonov (among others), see e.g.\ \cite{Bo, Ha, Leo}, have given necessary and sufficient conditions for the existence, the uniqueness and the asymptotic stability
 of a periodic orbit of an autonomous $n \ge 2$ dimensional ODE, and criteria to determine their basin of attraction. More recently this kind of study has been addressed in a non-autonomous context, extending Borg's criterion to almost periodic trajectories of an  almost periodic differential equation, see \cite{Gi1,Gi2, GiRa}:
 in all these problems the study of the Poincar\'e map in a neighborhood of   the periodic solutions plays a key role.

  The case of homoclinic trajectories is more tricky. Consider \eqref{eq-smooth} when $\ep=0$: the problem of the stability of the homoclinic trajectory $\ga(t)$
 has been completely solved by the so called Dulac sequence, see e.g.\ \cite[\S 13]{SSTC} and references therein. That is, if
 $\operatorname{div} [\bs{f_x}(\vec{0})]<0$, or  $\operatorname{div} [\bs{f_x}(\vec{0})]=0$ and $\int_{\ga}\operatorname{div} [\bs{f_x}]ds<0$, then $\ga(t)$ is stable from inside,
 while if $\operatorname{div} [\bs{f_x}(\vec{0})]>0$, or $\operatorname{div}[\bs{f_x}(\vec{0})]=0$ and $\int_{\ga}\operatorname{div} [\bs{f_x}]ds>0$, then $\ga(t)$ is unstable from inside,
 see \cite[\S 13]{SSTC} for more details.
 Once again the study of the Poincar\'e map ``inside'' $\bs{\Gamma}=\{\ga(t) \mid t \in \R \} \cup \{\vec{0}\}$ is crucial for this analysis; as far as the authors are aware there have been no attempts
 to extend the theory in a non-autonomous context: with  this paper we try to fill this gap.

 The first difficulty one faces in doing so is that we do not have anymore a proper Poincar\'e map, but we are forced to consider a map starting from
 a time dependent transversal,  $A^{\textrm{fwd},+}(\tau)$, which has as endpoint the intersection $\vec{P}_s(\tau)$ between the stable leaf $\tilde{W}^s(\tau)$
 and the transversal $L^0$, to end up close to $\vec{P}_u(\tau)$, the intersection of the unstable leaf $\tilde{W}^u(\tau)$ and $L^0$.
Further we have to split our argument in the analysis of four different subpaths: firstly from $L^0$ to a transversal   to $\tilde{W}^s(\tau)$ close to the origin, denoted by $\tilde{S}^+$, secondly  from $\tilde{S}^+$ to a curve $\Om^0$ passing through the origin, thirdly  from $\Om^0$ to a transversal   to $\tilde{W}^u(\tau)$ close to the origin,
$\tilde{S}^-$,
finally from $\tilde{S}^-$ back to $L^0$. In the first path we use a (nontrivial) fixed point argument which reminds the one used in \cite{Ha} for periodic trajectories,
in the second one we decompose the trajectory as a stable part, a linear one and a remainder and we use a (tricky) fixed point argument which combines exponential dichotomy
with some ideas borrowed from
\cite[\S 13]{SSTC}, and we estimate fly time  space displacement adapting in a nontrivial way \cite[\S 13]{SSTC}. Then we conclude with an inversion of time argument.

 This contribution   will be essential in proving some results concerning the analogues of the \emph{stability from inside} of $\bs{\Gamma}$
in a non-autonomous context which will be addressed in a forthcoming paper.
\\
Further we plan to use these ideas to study the possibility to establish a sub-harmonic theory for Melnikov, i.e., the possibility that
we may get a large number of homoclinic trajectories for $\ep \ne 0$ even if we have just one nondegenerate zero of $\MM(\tau)$.
\\
Moreover we believe this result could be used in a system characterized  by the presence of a chaotic pattern to find \emph{safe regions}, i.e., subsets
of initial conditions in a neighborhood of  $\bs{\Gamma}$, which cannot exhibit chaotic behavior.  See   \S \ref{s.open} for
more details concerning these three possible developments.

In fact all the results of this paper are obtained in a discontinuous piecewise smooth setting, i.e., we consider
\begin{equation}\label{eq-disc}\tag{PS}
	\dot{\x}=\f^\pm(\x)+\ep\g(t,\x,\ep),\quad \x\in\Om^\pm ,
\end{equation}
where $\Om^{\pm} = \{ \x\in\Om \mid \pm
G(\x)>0\}$, $\Om^{0}= \{ \x\in\Om \mid G(\x) = 0 \}$, $\Omega \subset \R^2$ is an open set, $G$ is a $C^{r}$-function on $\Om$ with $r> 1$
such that $0$ is a regular value of
$G$. Next, $\ep\in\R$ is a small parameter, and $\f^\pm \in C^{r}_b(\Om^\pm \cup \Omega^0, \R^2)$, $\g\in C_b^r(\R\times\Om\times\R,\R^2)$ and
$G\in C_b^r(\Om,\R)$, i.e., the
derivatives of $\f^\pm$, $\g$ and $G$ are uniformly continuous and bounded up to the $r$-th order, respectively,
 if $r\in \N$, and up to $r_0$ if $r=r_0+r_1$ with $r_0 \in \N$ and $0< r_1 <1$ and the $r_0$-th derivatives are $r_1$
H\"older continuous.
Further, we assume that both the systems
$\dot{\vec{x}}=\f^{\pm}(\x)$
admit the origin $\vec{0} \in \R^2$ as a fixed point,
and that $\vec{0}$ lies
on the discontinuity level $\Omega^0$.

Even in this discontinuous framework, the existence of a non-degenerate zero of the Melnikov function $\MM(\tau)$ guarantees the persistence
to perturbation of the homoclinic trajectory, see \cite{CaFr}.
Nevertheless,  a geometrical obstruction
forbids chaotic phenomena whenever we have sliding close to the origin, see \cite{FrPo}: this is indeed quite unexpected since this condition,
together with periodicity in $t$ of $\g$, are enough to guarantee the existence of a chaotic pattern in  a smooth context, see e.g.\ \cite{Pa84}, and in
a piecewise smooth context if $\vec{0} \not\in \Om^0$, see e.g.~\cite{BF10, BF11, BF12}.
 We plan to use   Theorems \ref{key} and \ref{keymissed} in a forthcoming paper to show that the usual Melnikov conditions guarantees chaos as in the smooth setting
  if the  geometrical obstruction is removed.

 Discontinuous problems are motivated by several physical applications, for instance mechanical systems with impacts, see e.g.~\cite{B99},
power electronics when we
have state dependent switches \cite{BV01}, walking machines \cite{GCRC98}, relay feedback systems \cite{BGV02}, biological systems \cite{PK};
see also \cite{DE13, LSZH16} and
the references therein.
Further they are also a  good source of examples since it is somehow easy to produce piecewise linear systems exhibiting an explicitly known
homoclinic
trajectory, so giving rise to chaos if subject to perturbations, whereas this is not an easy task in general for the smooth case (especially
if
the system
is not Hamiltonian).

But why shall one insist on studying the case where the critical point lies on the discontinuity surface $\Om^0$?
 In many real applications this is what really happens, e.g.\ in systems with dry friction.

 The paper is divided as follows: in \S \ref{S.WuWs} we collect the main assumptions used in the paper and we define
 the stable and unstable leaves; in \S \ref{S.prel} we construct the Poincar\'e map made up by trajectories of the perturbed problem performing a
 loop close to $\ga(t)$; in \S \ref{s.lemmakey} we state the main results of the paper, i.e., Theorems~\ref{key} and \ref{keymissed}, which are proved in \S \ref{S.key}.
 In \S \ref{s.open} we  give some hints concerning future development of this work which are in preparation.
 In the Appendix we sketch the proof of two auxiliary lemmas which are used to construct the Poincar\'e map.

\section{Preliminary constructions and notation}\label{S.WuWs}
First we give a notion of solution and we collect the basic assumptions which we assume through the whole paper.
By a \emph{solution} of
\eqref{eq-disc} we mean a continuous, piecewise $C^r$ function
$\x(t)$ that satisfies
\begin{align}
\dot{\x}(t)=\f^+(\x(t)) + \ep \g(t,\x(t),\ep), \quad \textrm{whenever } \x(t) \in \Om^+,  \label{eq-disc+} \tag{PS$+$}\\
\dot{\x}(t)=\f^-(\x(t)) + \ep \g(t,\x(t),\ep), \quad \textrm{whenever } \x(t) \in \Om^-. \label{eq-disc-}\tag{PS$-$}
\end{align}
 Moreover, if  $\x(t_{0})$ belongs to
$\Om^{0}$ for some $t_0$,   then  we assume
either $\x(t)\in\Om^{-}$ or $\x(t)\in\Om^{+}$  for $t$ in some left neighborhood of
$t_{0}$, say $]t_{0}-\tau,t_{0}[$ with $\tau>0$.
In the first case, the
left derivative of $\x(t)$ at $t=t_0$ has to satisfy $\dot
\x(t_{0}^{-}) = \f^{-}(\x(t_0))+ \ep \g(t_0,\x(t_0),\ep)$; while in the
second case, $\dot \x(t_{0}^{-}) = \f^{+}(\x(t_0))+ \ep
\g(t_0,\x(t_0),\ep)$. A~similar condition is required for the right
derivative $\dot \x(t_{0}^{+})$.
We stress that, in this paper, we do not consider solutions of equation
\eqref{eq-disc} that belong to $\Om^{0}$ for $t$ in some
nontrivial interval, i.e., sliding solutions.

\subsection*{Notation}
Throughout the paper we will use the following notation. We denote scalars by small letters,
e.g.\ $a$, vectors in $\R^2$ with an arrow, e.g.\ $\vec{a}$, and $n \times n$ matrices by bold letters, e.g.\ $\bs{A}$.
By $\vec{a}^*$ and $\bs{A}^*$ we mean the transpose of the vector $\vec{a}$ and of the matrix $\bs{A}$, resp.,
so that $\vec{a}^*\vec{b}$ \michal{or $\langle\vec{a},\vec{b}\rangle$} denotes the scalar product of the vectors $\vec{a}$, $\vec{b}$.
We denote by $\|\cdot\|$ the Euclidean norm in $\R^2$, while for matrices
we use the functional norm $\|\bs{A}\|= \sup_{\|\vec{w}\| \le 1} \|\bs{A} \vec{w}\|$.
We will use the shorthand notation $\bs{f_x}=\bs{\frac{\partial f}{\partial x}}$ unless this may cause confusion.

In the whole paper we denote by
$\x(t,\tau;\P)$ the trajectory of \eqref{eq-disc} passing through $\P$ at $t=\tau$ (evaluated at $t$).
If $D$ is a set and $\delta>0$, we
define
\begin{equation*}
  \begin{split}
     B(D, \delta):= &  \{ \vec{Q} \in \R^2 \mid   \exists \vec{P} \in D : \, \|\vec{Q}-\vec{P} \| <  \delta \}= \cup \{ B(\vec{P}, \delta)
 \mid \vec{P} \in D \}.
  \end{split}
\end{equation*}

We list here some hypotheses which we assume through the whole paper.

\begin{description}
	\item[\assump{F0}] We have $\vec{0}\in\Omega^0$,  $\f^\pm(\vec{0})=\vec{0}$, and the eigenvalues $\la_s^\pm$, $\la_u^\pm$ of
$\bs{f_x^\pm}(\vec{0})$ are such that
$\la_s^\pm<0<\la_u^\pm$.
\end{description}

Denote by $\vec{v}_s^{\pm}$, $\vec{v}_u^{\pm}$ the normalized eigenvectors of $\bs{f_x^\pm}(\vec{0})$ corresponding to $\la_s^\pm$,
$\la_u^\pm$.
Let us set
 \begin{equation}\label{cucs}
   c_u^{\perp,\pm}= [\na G(\vec{0})]^*\vec{v}^{\pm}_u \, , \quad  c_s^{\perp,\pm}= [\na G(\vec{0})]^*\vec{v}^{\pm}_s.
 \end{equation}
 We assume that the eigenvectors $\vec{v}_s^{\pm}$, $\vec{v}_u^{\pm}$ are not orthogonal to $\na G(\vec{0})$, i.e., the constants
 $c_u^{\perp,\pm}$ and $c_s^{\perp,\pm}$ are nonzero. To fix the ideas we require
\begin{description}
	\item[\assump{F1}]  $c_u^{\perp,-}<0<c_u^{\perp,+} \, , \quad  c_s^{\perp,-}<0<c_s^{\perp,+}$.
\end{description}

 Moreover we require a further condition on the mutual positions
 of the
directions spanned by
$\vec{v}_s^{\pm}$, $\vec{v}_u^{\pm}$.

More precisely, set $\mathcal{T}_u^{\pm}:=\{ c\vec{v}_u^{\pm} \mid c \ge 0 \}$, and
denote by $\Pi_u^1$ and $\Pi_u^2$ the disjoint open sets
in which $\R^2$ is divided by the polyline $\mathcal{T}^u:=\mathcal{T}_u^+ \cup \mathcal{T}_u^-$.
We require that $\vec{v}_s^+$ and $\vec{v}_s^-$ lie on ``opposite sides'' with respect to
	$\mathcal{T}^u$.
Hence, to fix the ideas, we assume: \begin{description}
	\item[\assump{F2}]   $\vec{v}_s^+ \in \Pi_u^1$ and $\vec{v}_s^- \in \Pi_u^2$.
\end{description}

We emphasize that if \assump{F2} holds, there is no sliding on $\Om^0$ close to $\vec{0}$.
On the other hand, sliding \michal{might occur} when both  $\vec{v}_s^{\pm}$ lie in
$\Pi_u^1$, or they both lie in  $ \Pi_u^2$, see \cite[\S 3]{FrPo}.

\begin{remark} \label{rem:settings}
We point out that it is the \emph{mutual position of the eigenvectors} that plays a role in the argument. In the paper we fix a particular
situation
for definiteness; however, by
reversing all the directions, one may obtain equivalent results.
Moreover, in the continuous case, i.e., for equation \eqref{eq-smooth},
then $\mathcal{T}^u$ is a line and $\Pi_u^1$, $\Pi_u^2$ are halfplanes.
In fact, all
smooth systems satisfy \assump{F2}.
On the other hand, if assumption \textbf{F2} is replaced with the opposite condition, that is $\vec{v}_s^+$ and $\vec{v}_s^-$ lie ``on the
same side'' with respect to
$\mathcal{T}^u$, then it was shown in \cite{FrPo} that \michal{generically} chaos cannot occur,  while new bifurcation
phenomena, involving continua of sliding homoclinic trajectories, may arise.
\end{remark}
\begin{description}
	\item[\assump{K}] For $\ep=0$ there is  a unique solution $\ga(t)$ of \eqref{eq-disc} homoclinic to the origin such that
	$$\ga(t)\in\begin{cases}\Om^-,& t<0,\\\Om^0,&t=0,\\\Om^+,&t>0.\end{cases}$$
	Furthermore, $(\na G(\ga(0)))^*\f^\pm(\ga(0))>0$.
\end{description}
Recalling  the orientation of $\vec{v}_s^{\pm}$, $\vec{v}_u^{\pm}$ chosen in \assump{F1}, we assume w.l.o.g.\  that
\begin{equation}\label{ass.scenario}
\lito\frac{\dot{\ga}(t)}{\|\dot{\ga}(t)\|}=\vec{v}_u^-
\quad \mbox{ and } \quad
\lit\frac{\dot{\ga}(t)}{\|\dot{\ga}(t)\|}=-\vec{v}_s^+.
\end{equation}

Concerning the perturbation term $\g$, we assume the following:
\begin{description}
	\item[\assump{G}] $\g(t,\vec{0},\ep)=\vec{0}$ for any $t,\ep\in\R$.
\end{description}
Hence, the origin is a critical point for the perturbed problem too.

\begin{figure}[t]
\centerline{\epsfig{file=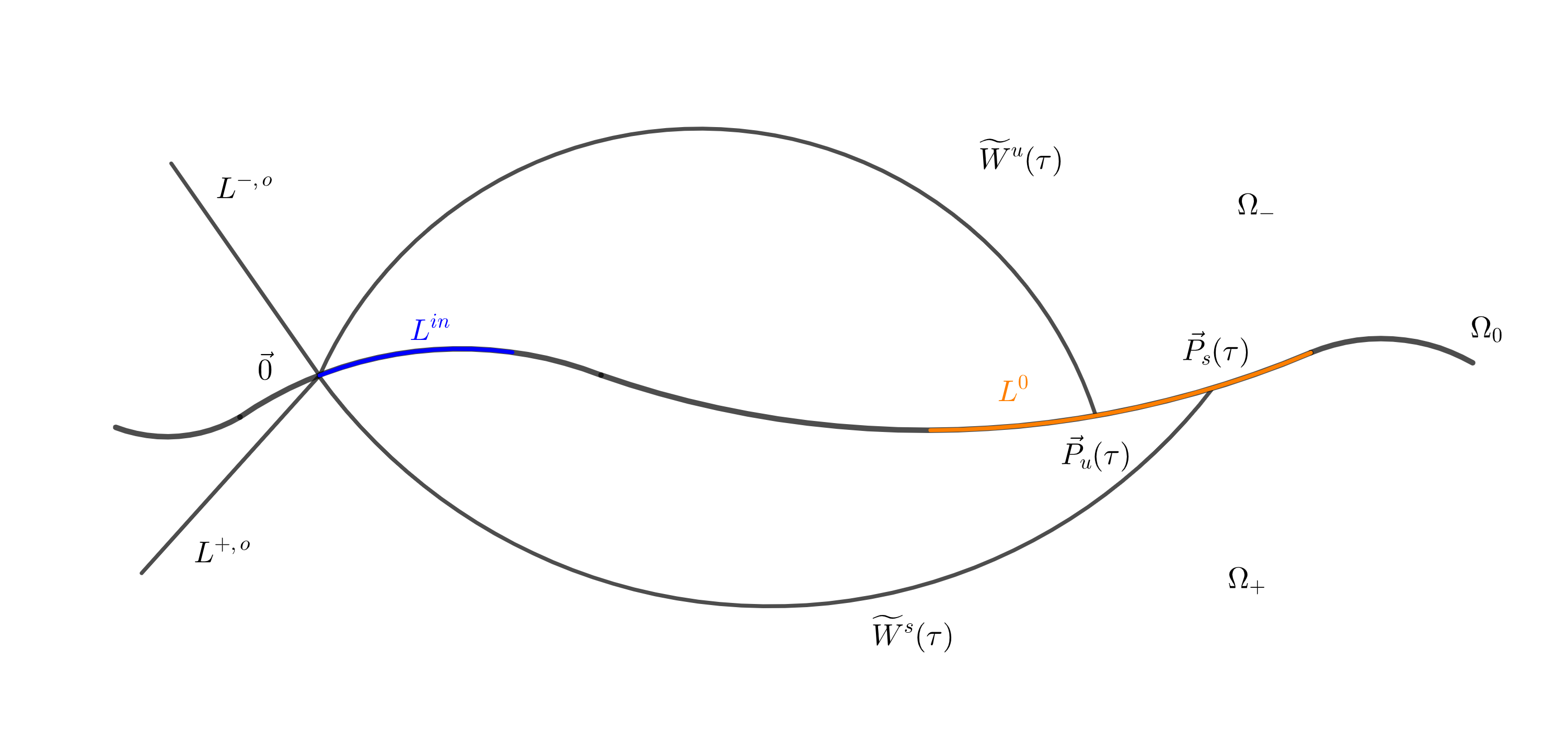, width = 10 cm} }
\caption{Stable and unstable leaves (the superscript ``$^{\textrm{out}}$"
	is denoted as ``$^{\textrm{o}}$" for short).}
\label{LinL0}
\end{figure}
We recall that
$\bs{\Gamma}:= \{\ga(t) \mid t \in \R \} \cup \{\vec{0}\}$; let us denote  by $E^{\textrm{in}}$ the open set enclosed by $\bs{\Gamma}$,
and by $E^{\textrm{out}}$ the open set complementary to $E^{\textrm{in}}\cup\bs{\Gamma}$.

Further, for any fixed $\delta>0$, we set (see Figure \ref{LinL0}):
 \begin{equation}\label{L0}
 \begin{split}
   L^0 =  L^0(\delta) & := \{ \vec{Q} \in \Omega^0 \mid \|\vec{Q}- \ga(0) \| <\delta \}, \\
   L^{\textrm{in}}=  L^{\textrm{in}}(\delta) &:= \{ \vec{Q} \in (\Omega^0 \cap E^{\textrm{in}}) \mid \|\vec{Q} \| <\delta\}, \\
   L^{-,\textrm{out}}=L^{-,\textrm{out}}(\delta) &:= \{ \vec{Q}= d (\vec{v}_u^-+  \vec{v}_s^-) \mid 0 \le d \le \delta \},\\
   L^{+,\textrm{out}}= L^{+,\textrm{out}}(\delta) &:= \{ \vec{Q}= d (\vec{v}_u^++  \vec{v}_s^+) \mid 0 \le d \le \delta \}.
  \end{split}
\end{equation}

Now, we define  the stable and the unstable leaves $W^s(\tau)$ and $W^u(\tau)$ of \eqref{eq-disc}.

Assume first for simplicity that the system is
\textit{smooth}; that is consider \eqref{eq-smooth} and suppose that \assump{F0} holds true.
 Then, following \cite[Appendix]{mFaS}, which is based on \cite[Theorem 2.16]{Jsell},  we can define local and global stable and unstable
 leaves as follows:
\begin{equation}\label{mm}
\begin{array}{c}
W^u_{loc}(\tau) := \{ \P \in B(\vec{0},\delta) \mid \x(t,\tau;\P) \in  B(\vec{0},\delta) \text{ for }t \le \tau, \, \lito
\x(t,\tau;\P)=\vec{0} \}, \\
W^s_{loc}(\tau) := \{ \P \in  B(\vec{0},\delta) \mid \x(t,\tau;\P) \in  B(\vec{0},\delta) \text{ for }t \ge \tau, \, \lit
\x(t,\tau;\P)=\vec{0} \}, \\
W^u(\tau) := \{ \P \in \R^2 \mid    \lito \x(t,\tau;\P)=\vec{0} \}, \\
W^s(\tau) := \{ \P \in \R^2 \mid   \lit  \x(t,\tau;\P)=\vec{0} \}.
\end{array}
\end{equation}
From  \cite[Theorem 2.16]{Jsell} it follows that
 $W^u_{loc}(\tau)$ and $W^s_{loc}(\tau)$ are $C^r$ embedded $1$-dimensional manifolds
if $\delta>0$ is small enough, while $W^u(\tau)$ and $W^s(\tau)$ are $C^r$ immersed $1$-dimensional manifolds,
i.e., they are the image of $C^r$ curves. Note that they also depend on $\ep$ but we leave this dependence unsaid.
The manifolds $W^u(\tau)$ and $W^s(\tau)$ are the sets of all the initial conditions of the trajectories
 converging to the origin in the past and in the future, respectively, and
they are not invariant for the flow of \eqref{eq-smooth}. However, if $\P \in W^u(\tau)$ then $\x(t,\tau;\P)\in W^u(t)$
for any $t, \tau \in \R$. Analogously for $W^s(\tau)$.

We emphasize that, choosing $\delta>0$ small enough, we can assume that $W^u_{loc}(\tau)$ and $W^s_{loc}(\tau)$
are graphs on the respective tangent spaces; further
$W^u(\tau)$ and $W^s(\tau)$ are constructed from $W^u_{loc}(\tau)$ and
 $W^s_{loc}(\tau)$ using the flow of \eqref{eq-smooth} as follows
\begin{equation}\label{ifneeded}
\displaystyle  \begin{split}
  W^u(\tau)= & \cup_{T \le \tau}\{ x(\tau, T; \Q) \mid \Q \in W^u_{loc}(T) \},  \\
   W^s(\tau)= & \cup_{T \ge \tau}\{ x(\tau, T; \Q) \mid \Q \in W^s_{loc}(T) \}.
   \end{split}
\end{equation}
Denote by
$W^{u,\pm}_{loc}(\tau)=W^u_{loc}(\tau) \cap (\Omega^{\pm} \cup \Omega^0)$ and by
$W^{s,\pm}_{loc}(\tau)=W^s_{loc}(\tau) \cap (\Omega^{\pm} \cup \Omega^0)$.
Assume further \assump{F1}, \assump{K} and follow $W^u(\tau)$ (respectively $W^s(\tau)$) from the origin towards
$L^0(\sqrt{\ep})$: then it intersects $L^0(\sqrt{\ep})$ transversely in a point denoted by $\P_u(\tau)$
(respectively by $\P_s(\tau)$).
In fact, $\P_u(\tau)$ and $\P_s(\tau)$ are $C^r$ functions of $\ep$ and $\tau$; hence
$\P_u(\tau)=\P_s(\tau)= \ga(0)$ if $\ep=0$ for any $\tau \in \R$.

We denote by $\tilde{W}^u(\tau)$ the branch of $W^u(\tau)$ between the origin and $\P_u(\tau)$
(a path), and by $\tilde{W}^s(\tau)$ the branch of $W^s(\tau)$ between the origin and $\P_s(\tau)$, in both the cases including the endpoints.
Since $\tilde{W}^u(\tau)$ and $\tilde{W}^s(\tau)$ coincide with $\bs{\Gamma} \cap (\Om^- \cup \Om^0)$
and $\bs{\Gamma} \cap (\Om^+ \cup \Om^0)$ if $\ep=0$, respectively, and vary in a $C^r$ way, see \cite[Theorem 2.16]{Jsell}
or  \cite[Appendix]{mFaS},
 we find $\tilde{W}^u(\tau)\subset  (\Om^- \cup \Om^0)$ and
$\tilde{W}^s(\tau)\subset  (\Om^+ \cup \Om^0)$, for any $\tau \in \R$ and any $0\le \ep \le \ep_0$.
Further
\begin{equation}\label{proprieta}
\begin{gathered}
 \Q_u \in \tilde{W}^u(\tau)   \Rightarrow  \x(t,\tau; \Q_u) \in \tilde{W}^u(t) \subset (\Om^- \cup \Om^0) \quad \textrm{for any $t\le
 \tau$},\\
\Q_s \in \tilde{W}^s(\tau) \Rightarrow  \x(t,\tau; \Q_s) \in \tilde{W}^s(t) \subset (\Om^+ \cup \Om^0) \quad \textrm{for any $t\ge \tau$}.
\end{gathered}
\end{equation}

Now, we go back to the general case where \eqref{eq-disc} is piecewise smooth but discontinuous.
Using  the previous construction, we can define also in this case $W^{u,\pm}_{loc}(\tau)$, $W^{s, \pm}_{loc}(\tau)$,
$\tilde{W}^{u}(\tau)$ and  $\tilde{W}^{s}(\tau)$,
and they are all $C^r$.
\begin{remark}\label{Wloc}
Consider \eqref{eq-disc} and assume
 \assump{F0}, \assump{F1},  \assump{F2}.
   Then the manifolds $W^u_{loc}(\tau)$ and
 $W^s_{loc}(\tau)$ defined as in \eqref{mm} are such that $W^u_{loc}(\tau)=
 W^{u,+}_{loc}(\tau) \cup W^{u,-}_{loc}(\tau)$,
 $W^s_{loc}(\tau)=
 W^{s,+}_{loc}(\tau) \cup W^{s,-}_{loc}(\tau)$. Hence, $W^u_{loc}(\tau)$ and
 $W^s_{loc}(\tau)$ are piecewise $C^r$ manifolds (they might lose smoothness in the origin).
 Further $\tilde{W}^u(\tau)$ and $\tilde{W}^s(\tau)$ are again smooth and
 have the property \eqref{proprieta}.

 Moreover, if \textbf{K} holds then
  $\P_u(\tau)$ and $\P_s(\tau)$ are again $C^r$ in $\ep$ and $\tau$,
 and $\P_u(\tau)=\P_s(\tau)= \ga(0)$ if $\ep=0$ for any $\tau \in \R$.
\end{remark}
We set
\begin{equation}\label{tildeW}
\tilde{W}(\tau):= \tilde{W}^u(\tau) \cup \tilde{W}^s(\tau).
\end{equation}
 We will see that $\tilde{W}(\tau) \subset B(\bs{\Gamma}, \bar{c}^* \ep)$ for a suitable $\bar{c}^*>0$,  see Remark \ref{rMelnibis}
below.

\begin{figure}[t]
\begin{center}
	\epsfig{file=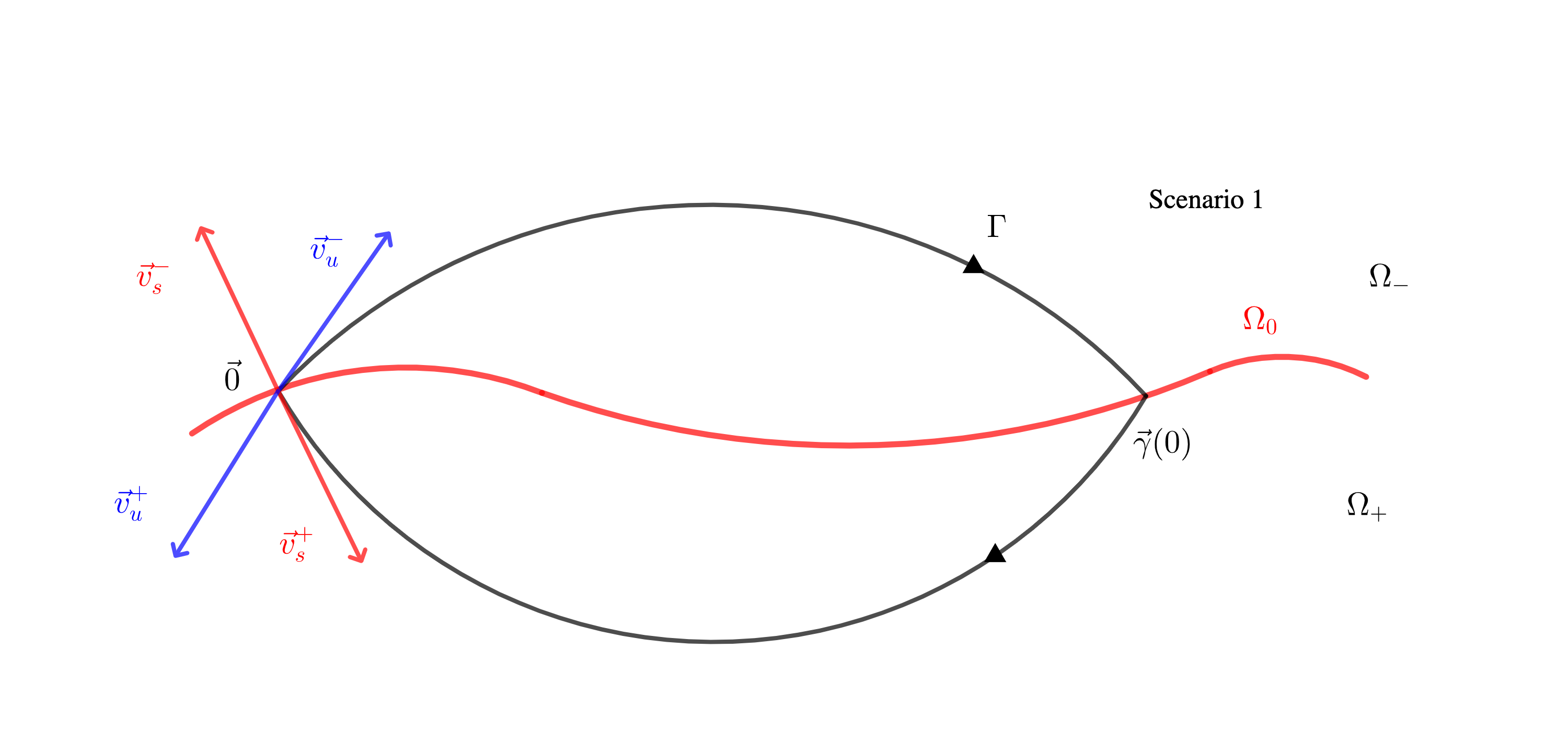, width = 8 cm}\\
	\epsfig{file=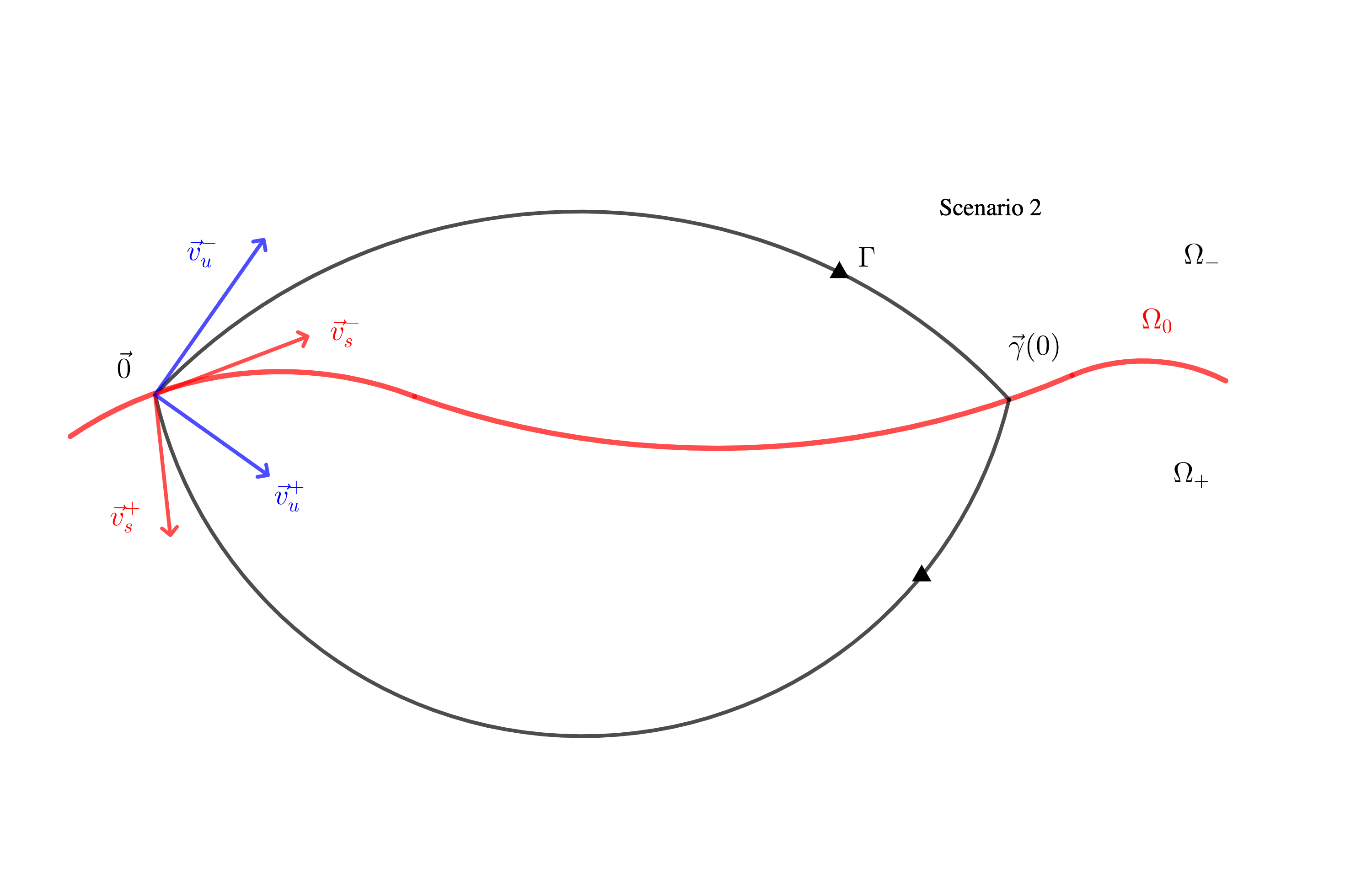, width = 8 cm}
\end{center}
\caption{Scenarios 1 and 2. In   these two settings there is no sliding close to the origin. Further Melnikov theory guarantees persistence of the homoclinic trajectories
  \cite{CaFr}, and we conjecture that we may have chaotic phenomena.
}
\label{scenario123}
\end{figure}

 Let us define the Melnikov function $\mathcal{M}:\R\to\R$  which, for planar
 piecewise smooth systems as \eqref{eq-disc}, takes the following form, see e.g.\ \cite{CaFr}:

\begin{equation}\label{melni-disc}
\begin{split}
   \mathcal{M}(\al) &= \int_{-\infty}^{0} \eu^{-\int_0^t
			\tr\boldsymbol{\f_x^-}(\ga(s))ds} \f^-(\ga(t))
		\wedge \g(t+\al,\ga(t),0) dt\\
		&\quad {}+  \int_{0}^{+\infty} \eu^{-\int_0^t
			\tr\boldsymbol{\f_x^+}(\ga(s))ds} \f^+(\ga(t))
		\wedge \g(t+\al,\ga(t),0) dt,
\end{split}
	\end{equation}
where ``$\wedge$" is the wedge product in $\R^2$ defined by $\vec{a} \wedge \vec{b}= a_1b_2-a_2b_1$
for any vectors $\vec{a}=(a_1,a_2)$, $\vec{b}=(b_1,b_2)$.
In fact, also in the piecewise smooth case the function $\mathcal{M}$ is $C^r$.

At this point, we need to distinguish between four possible scenarios, see Figures~\ref{scenario123} and \ref{scenario45}.
 \begin{description}
\item[Scenario 1] Assume \assump{K} and that there is $\rho>0$ such that $d\vec{v}_u^+ \in E^{\textrm{out}}$, $d\vec{v}_s^- \in
    E^{\textrm{out}}$ for any $0<d<\rho$.
\item[Scenario 2] Assume \assump{K} and that there is $\rho>0$ such that $d\vec{v}_u^+ \in E^{\textrm{in}}$, $d\vec{v}_s^- \in
    E^{\textrm{in}}$ for any $0<d<\rho$.
\item[Scenario 3]
Assume \assump{K} and that there is $\rho>0$ such that $d\vec{v}_u^+ \in E^{\textrm{in}}$, $d\vec{v}_s^- \in E^{\textrm{out}}$ for any
$0<d<\rho$, so
    \assump{F2} does not hold.
\item[Scenario 4]
Assume \assump{K} and that there is $\rho>0$ such that $d\vec{v}_u^+ \in E^{\textrm{out}}$, $d\vec{v}_s^- \in E^{\textrm{in}}$ for any
$0<d<\rho$, so
    \assump{F2} does not hold.
\end{description}
Notice that  \assump{F2} holds in both Scenarios 1 and 2, and our results apply to both the cases. Further, in Scenario 1
$W^{u,+}_{loc}(\tau)$ and $W^{s,-}_{loc}(\tau)$ both lie in $E^{\textrm{out}}$ for any $\tau \in \R$,
while in Scenario 2   $W^{u,+}_{loc}(\tau)$ and $W^{s,-}_{loc}(\tau)$ both lie in $E^{\textrm{in}}$ for any $\tau \in
    \R$.

    In Scenarios 3 and 4 sliding \michal{generically occurs} in $\Om^0$ close to the origin and \assump{F2} does not hold.
    Notice that Scenarios 1 and 2 have a smooth counterpart while Scenarios 3 and 4 may take place just if the system is discontinuous.
     We recall once again that
    in all the four scenarios the existence of a non-degenerate zero of the Melnikov function guarantees the persistence of the homoclinic
    trajectory, cf. \cite{CaFr}, but chaos is \michal{generically} not possible in Scenarios 3 and 4. We conjecture that chaos is still possible in Scenarios 1 and 2: this will
    be the object of  a future investigation in which the result of this article will be crucial.

    \textbf{In this paper we will just consider Scenario 1 to fix the ideas, even though Scenario 2 can be handled in a similar way.}

\begin{figure}[t]
\begin{center}
	\epsfig{file=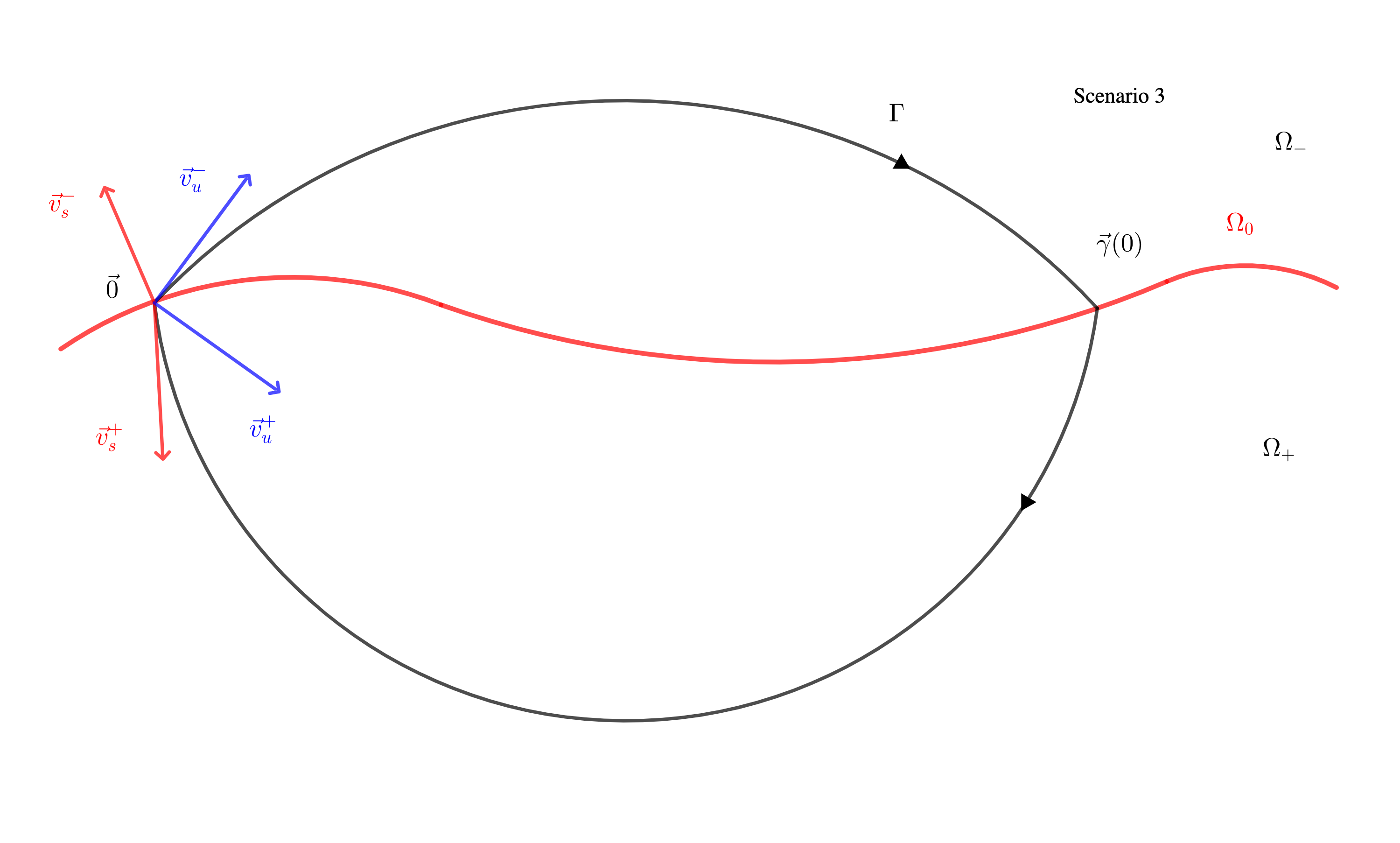, width = 8 cm}\\
	\epsfig{file=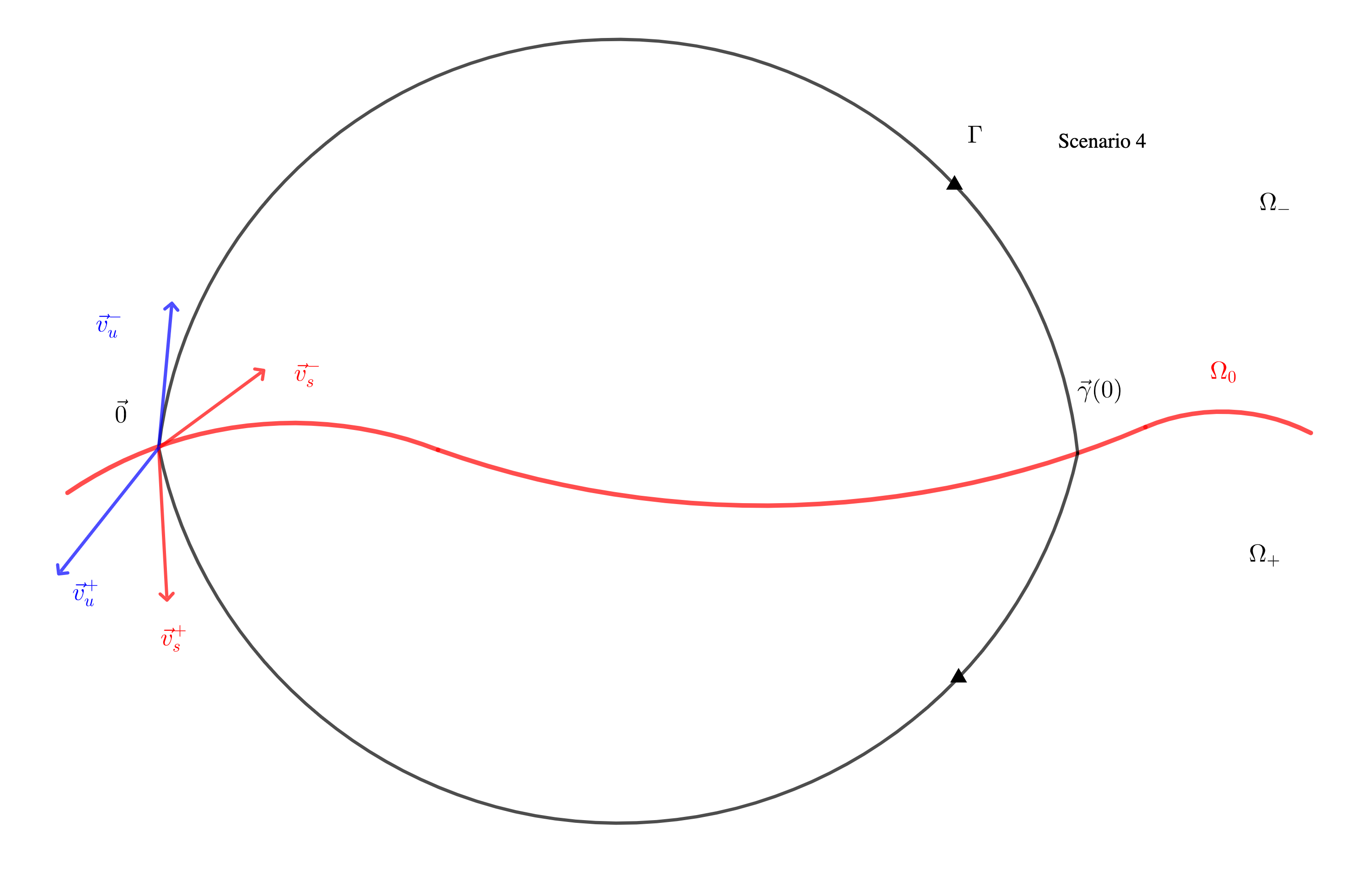, width = 8 cm}
\end{center}
\caption{Scenarios 3 and 4: in these settings we have persistence of the homoclinic trajectories   but sliding \michal{might occur} close to the origin. Here our analysis does not apply directly. Further Melnikov theory guarantees persistence of the homoclinic trajectories
  \cite{CaFr}, but chaos is forbidden \cite{FrPo} \michal{in general}.}
\label{scenario45}
\end{figure}

We collect here  for convenience of the reader and future reference,  the main constants which
will play a role in our argument:
\begin{equation}\label{defsigma}
\displaystyle
\begin{array}{ccc}
\sfwd_+= \frac{|\la_s^+|}{\la_u^++|\la_s^+|}, & \sfwd_-= \frac{\la_u^-+|\la_s^-|}{\la_u^-}, &
\sfwd=  \sfwd_+ \sfwd_-,  \\
  \sbwd_+=   \frac{1}{\sfwd_+}, & \sbwd_- =\frac{1}{\sfwd_-}, &  \sbwd=   \sbwd_+ \sbwd_-, \\
 \underline{\sigma} = \min \{ \sfwd_+ , \sbwd_-  \} , &  & \overline{\sigma}= \max \{ \sfwd_+ , \sbwd_-  \},
  \end{array}
\end{equation}
 \begin{equation*}
\displaystyle
\begin{array}{ccc}
  \sTfwd_+= \frac{1}{\la_u^++|\la_s^+|}, & & \sTbwd_- = \frac{1}{\la_u^- +|\la_s^-|},
  \\
\sTfwd=\frac{\la_u^-+|\la_s^+|}{\la_u^-(\la_u^++|\la_s^+|)} ,&  & \sTbwd= \frac{\la_u^{-}+|\la_s^+|}{|\la_s^+|(\la_u^-+|\la_s^-|)},\\
\underline{\Sigma}= \min \{ \sTfwd , \sTbwd  \} , &  & \overline{\Sigma}= \max \{ \sTfwd , \sTbwd  \},\\
\und{\la}= \min \{ \la_u^-; \la_u^+ ;|\la_s^-|;|\la_s^+| \} , & & \qquad \ov{\la}= \max \{ \la_u^-; \la_u^+ ;|\la_s^-|;|\la_s^+| \}.
\end{array}
\end{equation*}
\begin{remark}
Notice that in the smooth setting  we have $\la_u^+=\la_u^-$, and $\la_s^+= \la_s^-$ so we have the following simplifications
\begin{equation}\label{defsigmabis}
\begin{array}{c}
\sfwd= \frac{|\la_s|}{\la_u}, \quad \sbwd= \frac{\la_u}{|\la_s|}; \qquad  \sTfwd= \frac{1}{\la_u} , \quad \sTbwd= \frac{1}{|\la_s|}.
  \end{array}
\end{equation}
\end{remark}

\section{Construction of the Poincar\'e map} \label{S.prel}

Fix $\tau \in \R$ and consider $L^0(\delta)$  given by \eqref{L0}, where $\delta>0$  will be
chosen below.

The point $\P_s(\tau)$ splits $L^0$ in two parts, say $A^{\textrm{fwd},+}(\tau)$ and $B^{\textrm{fwd},+}(\tau)$, respectively
 ``inside'' and ``outside'', see
Figure~\ref{Kforward}.
The purpose of this section is to construct a Poincar\'e map
using the flow of \eqref{eq-disc} from $L^0$ back to itself remaining close to $\bs{\Gamma}$; i.e.\
$\mathscr{P}^{\textrm{fwd}}(\cdot, \tau): A^{\textrm{fwd},+}(\tau) \to L^0$ and a time map $\mathscr{T}^{\textrm{fwd}}(\cdot,\tau):
A^{\textrm{fwd},+}(\tau) \to \R$ such that
for any $\P \in A^{\textrm{fwd},+}(\tau)$ the trajectory
$\x(t,\tau; \P)$ will stay in a neighborhood of $\bs{\Gamma}$  (in fact in a neighborhood of $\tilde{W}(t)$) for any
$t \in [\tau, \mathscr{T}^{\textrm{fwd}}(\P,\tau)]$ and it will cross transversely
$L^0$ for the first time at $t=\mathscr{T}^{\textrm{fwd}}(\P,\tau)>\tau$ in the point $\mathscr{P}^{\textrm{fwd}}(\P, \tau)$. Further, we want
to show that
both the maps are $C^r$.
\\
Moreover, if $\P \in B^{\textrm{fwd},+}(\tau)$ then there is some
 $\mathscr{T}^{\textrm{out}}=\mathscr{T}^{\textrm{out}}(\P,\tau)>\tau$ such that the trajectory will leave a neighborhood of $\bs{\Gamma}$ at
 $t \ge
 \mathscr{T}^{\textrm{out}}$.

Similarly,   the point $\P_u(\tau)$ splits $L^0$ in two parts, say $A^{\textrm{bwd},-}(\tau)$ and $B^{\textrm{bwd},-}(\tau)$, respectively
 ``inside'' and
``outside'', see
Figure~\ref{Kbackward}.
Using the flow of \eqref{eq-disc} (but now going backward in time) we can construct a  $C^r$ Poincar\'e map
$\mathscr{P}^{\textrm{bwd}}(\cdot, \tau): A^{\textrm{bwd},-}(\tau) \to L^0$ and a $C^r$ time  map $\mathscr{T}^{\textrm{bwd}}(\cdot,\tau):
A^{\textrm{bwd},-}(\tau) \to \R$
such that
for any $\P \in A^{\textrm{bwd},-}(\tau)$ the trajectory
$\x(t,\tau; \P)$ will stay in a neighborhood of $\bs{\Gamma}$ (in fact in a neighborhood of $\tilde{W}(t)$) for any
$t \in [\mathscr{T}^{\textrm{bwd}}(\P,\tau), \tau]$ and it will cross transversely
$L^0$ for the first time at $t=\mathscr{T}^{\textrm{bwd}}(\P,\tau)<\tau$ in the point $\mathscr{P}^{\textrm{bwd}}(\P, \tau)$.
\\
Moreover if $\P \in B^{\textrm{bwd},-}(\tau)$ then there is some $\mathscr{T}^{\textrm{out}}(\P,\tau)<\tau$
 such that the trajectory will leave a neighborhood of $\bs{\Gamma}$ at $t \le \mathscr{T}^{\textrm{out}}$.

Hereafter it is convenient to denote $ \ga^-(t):=\ga(t)$ when $t \le 0$ and $ \ga^+(t):=\ga(t)$ when $t \ge 0$.

\begin{remark}\label{est.ga}
  Assume \assump{F0}, \assump{K}, then there is  a constant $c_0^*>0$ such that
  $\|\ga^-(t)\| \le \frac{c_0^*}{4} \eu^{\la_u^- t}$ for any $t \le 0$ and
  $\|\ga^+(t)\| \le \frac{c_0^* }{4}\eu^{\la_s^+ t}$ for any $t \ge 0$.
\end{remark}

We state now a classical result  concerning the possibility to estimate the position of the trajectories of the unstable manifold
$\tilde{W}^u(\tau)$ and of the stable manifold $\tilde{W}^s(\tau)$ using the homoclinic trajectory $\ga(t)$.
The proof is omitted, see, e.g., the nice introduction of \cite{JPY},  or \cite[\S 4.5]{GH}.
\begin{remark}\label{rMelnibis}
 Assume \assump{K} and \assump{F1}, then there is $\ep_0>0$ such that for any $0< \ep \le \ep_0$ we have the following. There is $\bar{c}^*>0$
 such that
 \begin{equation}\label{trajWuWs}
   \begin{split}
      \| \x(t,\tau ; \P_u(\tau)) -\ga^-(t-\tau) \| \le \bar{c}^* \ep & \qquad \qquad \textrm{for any $t \le \tau$}, \\
        \| \x(t,\tau ; \P_s(\tau)) -\ga^+(t-\tau) \| \le \bar{c}^* \ep & \qquad \qquad \textrm{for any $t \ge \tau$}.
   \end{split}
 \end{equation}
\end{remark}

By a slight modification of \cite[Lemmas 6.4, 6.7]{FrPo} and the argument of \cite[\S 6.2.2]{FrPo} we can construct
 two auxiliary curves $Z^{\textrm{fwd},\textrm{in}}$, $Z^{\textrm{fwd},\textrm{out}}$ with the following properties, see Figure
 \ref{Kforward},
 and also Figures \ref{fig-Zui},   \ref{fig-Zuo}.
These curves are useful to build up our Poincar\'e map in forward time.

\begin{figure}[t]
\centerline{\epsfig{file=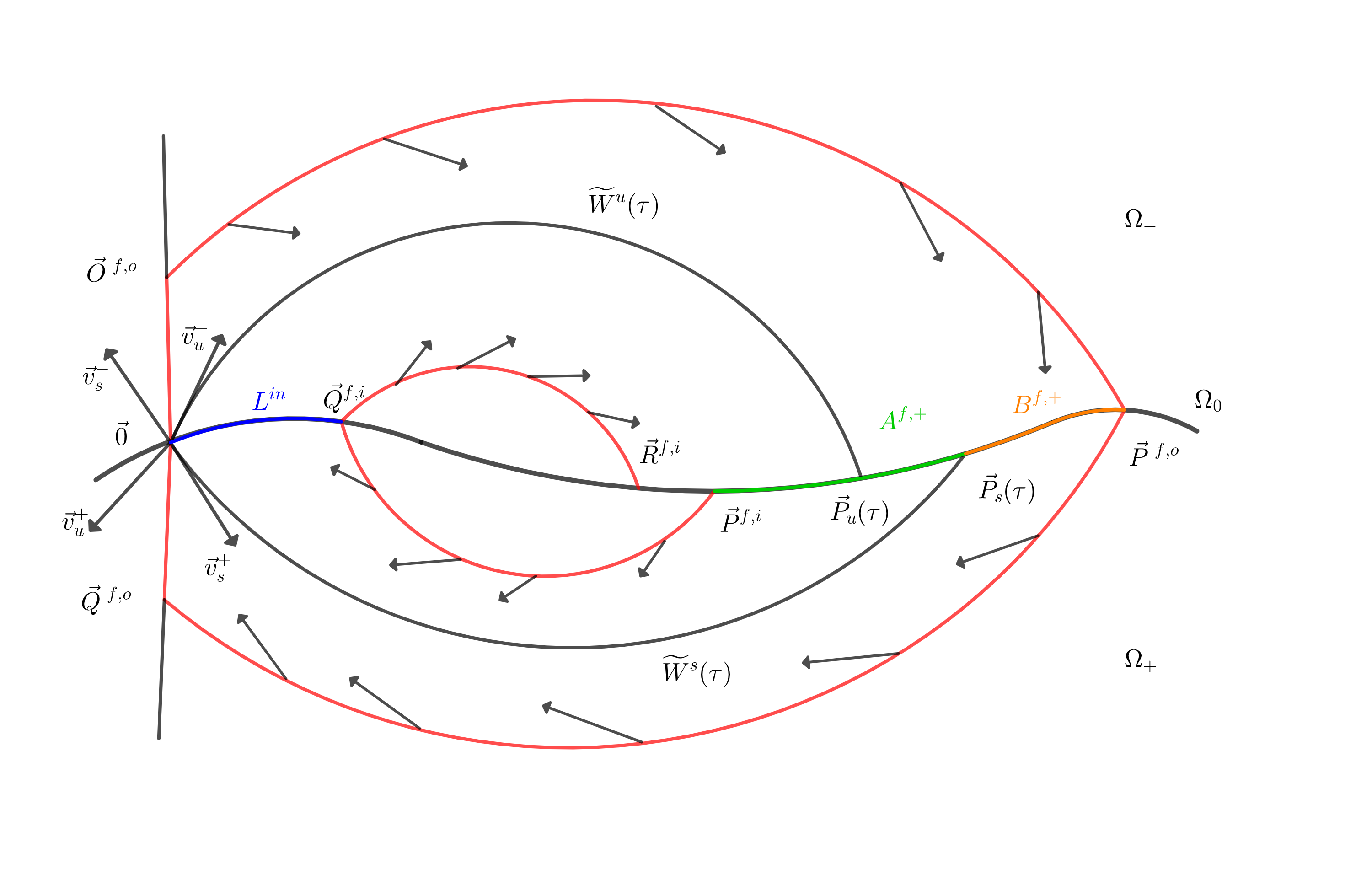, width = 8 cm}}
\caption{Construction of the set $K^{\textrm{fwd}}$ (the superscripts ``$^{\textrm{fwd}}$", ``$^{\textrm{in}}$", ``$^{\textrm{out}}$"
are denoted as ``$^{\textrm{f}}$", ``$^{\textrm{i}}$", ``$^{\textrm{o}}$", respectively, for short).
This picture enables us to control the trajectories of \eqref{eq-disc} in a neighborhood of $ \bs{\Gamma}$ in forward time.}
\label{Kforward}
\end{figure}

The curve $Z^{\textrm{fwd},\textrm{in}}$ is a continuous, piecewise--$C^r$ path starting from a point
$\vec{P}^{\textrm{fwd},\textrm{in}} \in (L^{0} \cap E^{\textrm{in}})$, passing through a point $\vec{Q}^{\textrm{fwd},\textrm{in}} \in
L^{\textrm{in}}$ and ending in
a point $\vec{R}^{\textrm{fwd},\textrm{in}} \in (L^{0} \cap E^{\textrm{in}})$,
 see \eqref{L0} and the beginning of Section \ref{S.WuWs}.

 The branch from $\vec{P}^{\textrm{fwd},\textrm{in}}$ to
$\vec{Q}^{\textrm{fwd},\textrm{in}}$ is
in $\Omega^+$,
while the branch from $\vec{Q}^{\textrm{fwd},\textrm{in}}$ to $\vec{R}^{\textrm{fwd},\textrm{in}}$ is in $\Omega^-$ and they are both $C^r$.

Analogously,
the curve $Z^{\textrm{fwd},\textrm{out}}$ is a continuous, piecewise--$C^r$ path starting from a point
$\vec{O}^{\textrm{fwd},\textrm{out}} \in L^{-,\textrm{out}}$, passing through a point $\vec{P}^{\textrm{fwd},\textrm{out}} \in (L^{0} \cap
E^{\textrm{out}})$ and ending in
a point $\vec{Q}^{\textrm{fwd},\textrm{out}} \in L^{+,\textrm{out}}$. The branch from $\vec{O}^{\textrm{fwd},\textrm{out}}$ to
$\vec{P}^{\textrm{fwd},\textrm{out}}$ is in
$\Omega^-$,
while the branch from $\vec{P}^{\textrm{fwd},\textrm{out}}$ to $\vec{Q}^{\textrm{fwd},\textrm{out}}$ is in $\Omega^+$ and they are both $C^r$.
Further    $Z^{\textrm{fwd},\textrm{in}}\subset E^{\textrm{in}}$ while
$Z^{\textrm{fwd},\textrm{out}}\subset E^{\textrm{out}}$.

We denote by $K^{\textrm{fwd}}$ the compact set enclosed by $Z^{\textrm{fwd},\textrm{out}}$, $L^{+,\textrm{out}}$,
$L^{-,\textrm{out}}$, $Z^{\textrm{fwd},\textrm{in}}$ and the path  of
$\Om^0$ between $\vec{P}^{\textrm{fwd},\textrm{in}}$ and $\vec{R}^{\textrm{fwd},\textrm{in}}$. We denote by $K^{\textrm{fwd},\pm}=
K^{\textrm{fwd}} \cap \Omega^{\pm}$.
Notice that by construction $Z^{\textrm{fwd},\textrm{in}}$ and  $Z^{\textrm{fwd},\textrm{out}}$ do not intersect each other or
themselves and that $\bs{\Gamma}
\subset K^{\textrm{fwd}}$.

 In order to construct the auxiliary curves $Z^{\textrm{fwd},\textrm{in}}$ and  $Z^{\textrm{fwd},\textrm{out}}$,
 we need to introduce a small parameter $\beta$, which however has to be larger than $\ep$, namely:
\begin{equation}\label{def-sifb}
\beta \ge \ep^{\frac{\sigma^{\textrm{fb}}}{2}}\, , \qquad \sigma^{\textrm{fb}}=\min \{ \sfwd, \sbwd    \}<1.
\end{equation}
We   introduce a further parameter $\mu$, whose size has to be small, independently of~$\ep$.
Roughly speaking, $\mu$  will play the role of controlling the errors in evaluating the positions of the ``barriers" $Z^{\textrm{fwd,in}}$ and
similar, just below,  and the space displacement $\PPP$ with respect to $\tilde{W}$, defined in the main results, Theorems~\ref{key} and \ref{keymissed} below.

Let us    define
\begin{equation}\label{mu0}
\mu_0= \frac{1}{4} \min \left\{ \sTfwd_+ , \sTbwd_- ,  \und{\sigma}^2    \right\}
\end{equation}

Then we have the following.

\begin{lemma}\label{lemma0s}
Assume \assump{F0}, \assump{F1}, \mat{\assump{F2}}, \assump{K}, \assump{G}.
We can choose  $\ep_0$ and $\beta_0$ such that for any $0< \ep \le \ep_0$, any $0< \beta \le \beta_0$ with $\beta \ge \ep^{\frac{\sigma^{\textrm{fb}}}{2}}$ we can define  $Z^{\textrm{fwd},\textrm{in}}$, $Z^{\textrm{fwd},\textrm{out}}$ so that the flow of \eqref{eq-disc} on
$Z^{\textrm{fwd},\textrm{in}} \setminus \{\P^{\textrm{fwd},\textrm{in}}, \vec{R}^{\textrm{fwd},\textrm{in}}\}$ and on
$Z^{\textrm{fwd},\textrm{out}} \setminus \{\vec{O}^{\textrm{fwd},\textrm{out}}, \vec{Q}^{\textrm{fwd},\textrm{out}}\}$
 aims towards the interior of $K^{\textrm{fwd}}$.
Then,
for any  $0< \mu \le \mu_0  $ we find
\begin{equation}\label{constant.s}
 \begin{gathered}
  \begin{array}{cc}
    \|\P^{\textrm{fwd},\textrm{in}}-\ga(0)\|= \beta , & \|\vec{P}^{\textrm{fwd},\textrm{out}}-\ga(0)\|=
    \beta , \\
        \beta^{ \sfwd_{+}+\mu} \le\|\Q^{\textrm{fwd},\textrm{in}}\|\le
          \beta^{   \sfwd_{+} -\mu} ,   &
          \beta^{    \sfwd_{+} +\mu} \le
   \|\vec{Q}^{\textrm{fwd},\textrm{out}}\| \le   \beta^{ \sfwd_{+} -\mu},   \end{array}
        \\
       	\beta^{\sfwd +\mu} \le   \|\vec{R}^{\textrm{fwd},\textrm{in}}-\ga(0)\|\le  \beta^{\sfwd -\mu},
        \\  \beta^{ \sbwd_{-} +\mu}\le\|\vec{O}^{\textrm{fwd},\textrm{out}}\|   \le \beta^{ \sbwd_{-} -\mu}.
	\end{gathered}
\end{equation}
 Finally we can assume that
 $ K^{\textrm{fwd}}   \subset  B(\bs{\Gamma}, \beta^{ \underline{\sigma}  -\mu})$,
 and that there is $c>0$ such that $B(\bs{\Gamma}, c \ep) \cap ( Z^{\textrm{fwd},\textrm{in}}\cup Z^{\textrm{fwd},\textrm{out}})=\emptyset$.
\end{lemma}
 A sketch of the proof of both Lemma \ref{lemma0s} and Lemma \ref{lemma0u} below is postponed to Appendix~\ref{A.technic}.

\begin{remark}
  We invite the reader not to focus on the precise values of the exponents appearing in \eqref{constant.s}.
  Notice that if $\la_u^+=\la_u^-=|\la_s^+|=|\la_s^-|$ as, e.g., in the smooth Hamiltonian case
 inequalities
  \eqref{constant.s} simplify as follows
  \begin{equation}\label{easyconstant.s}
\begin{gathered}
  \begin{array}{cc}
    \|\P^{\textrm{fwd},\textrm{in}}-\ga(0)\|= \beta , & \|\vec{P}^{\textrm{fwd},\textrm{out}}-\ga(0)\|= \beta , \\
        \beta ^{\frac{1}{2}+\mu} \le\|\Q^{\textrm{fwd},\textrm{in}}\|\le
          \beta ^{\frac{1}{2}-\mu} , &    \beta ^{\frac{1}{2}+\mu} \le
   \|\vec{Q}^{\textrm{fwd},\textrm{out}}\| \le   \beta ^{\frac{1}{2}-\mu}
      ,   \end{array}
        \\
       	\beta  ^{1+\mu} \le   \|\vec{R}^{\textrm{fwd},\textrm{in}}-\ga(0)\|\le  \beta ^{1-\mu},
        \\
   \beta ^{\frac{1}{2}+\mu}\le\|\vec{O}^{\textrm{fwd},\textrm{out}}\|   \le \beta ^{\frac{1}{2}-\mu}.
	\end{gathered}
\end{equation}
 Further, taking into account that $0\le \mu \le 1/16$, we get  $K^{\textrm{fwd}}   \subset  B(\bs{\Gamma}, \beta^{1/4})$.
\end{remark}

Similarly, we   construct some other auxiliary curves useful to build up the Poincar\'e map in backward time.
By a modification of \cite[Lemmas 6.4, 6.7]{FrPo} and from \cite[\S 6.2.2]{FrPo} we construct
the continuous, piecewise--$C^r$ curves $Z^{\textrm{bwd},\textrm{in}}$, $Z^{\textrm{bwd},\textrm{out}}$ with the following properties,
see Figure \ref{Kbackward},  see also Figures \ref{fig-Zsi} and \ref{fig-Zso}.

\begin{figure}[t]
\centerline{\epsfig{file=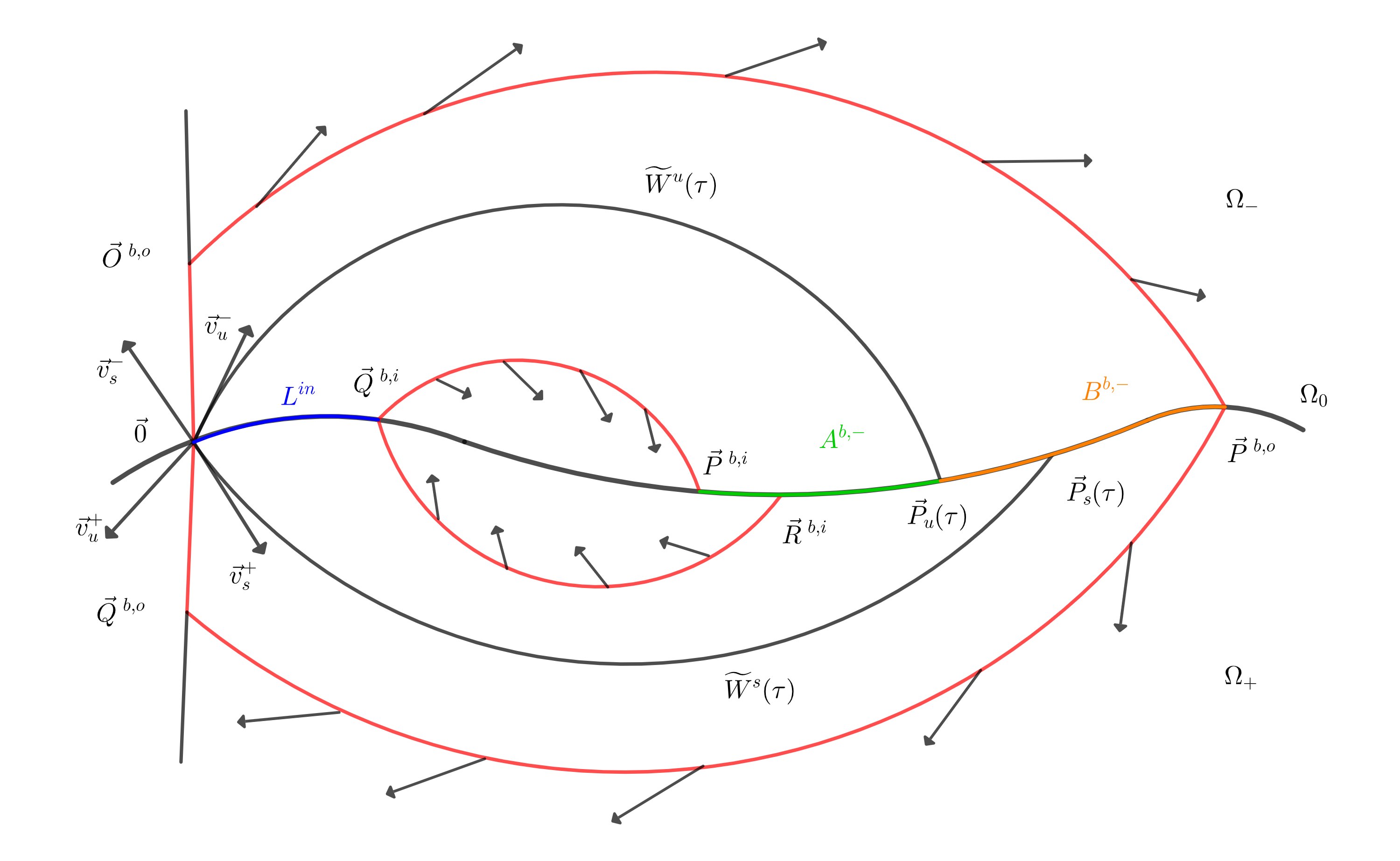, width = 8 cm}}
\caption{Construction of the set $K^{\textrm{bwd}}$ (the the superscripts ``$^{\textrm{bwd}}$", ``$^{\textrm{in}}$",
``$^{\textrm{out}}$" are denoted as ``$^{\textrm{b}}$", ``$^{\textrm{i}}$", ``$^{\textrm{o}}$", respectively, for short).
This picture enables us to control the trajectories of \eqref{eq-disc} in a neighborhood of $\bs{\Gamma}$ in backward time.}
\label{Kbackward}
\end{figure}

The curve $Z^{\textrm{bwd},\textrm{in}}$  starts from a point
$\vec{P}^{\textrm{bwd},\textrm{in}} \in (L^{0} \cap E^{\textrm{in}})$, passes through a point
 $\vec{Q}^{\textrm{bwd},\textrm{in}} \in L^{\textrm{in}}$ and ends in a point
$\vec{R}^{\textrm{bwd},\textrm{in}} \in (L^{0} \cap E^{\textrm{in}})$, so that the branch from $\vec{P}^{\textrm{bwd},\textrm{in}}$ to
$\vec{Q}^{\textrm{bwd},\textrm{in}}$ is
in \michal{$\Omega^-$},
while the branch from $\vec{Q}^{\textrm{bwd},\textrm{in}}$ to $\vec{R}^{\textrm{bwd},\textrm{in}}$ is in \michal{$\Omega^+$}, and they are both $C^r$.

Analogously,
the curve $Z^{\textrm{bwd},\textrm{out}}$  starts from a point
$\vec{O}^{\textrm{bwd},\textrm{out}} \in L^{-,\textrm{out}}$, passes through a point
 $\vec{P}^{\textrm{bwd},\textrm{out}} \in (L^{0} \cap E^{\textrm{out}})$ and ends in a point
$\vec{Q}^{\textrm{bwd},\textrm{out}} \in L^{+,\textrm{out}}$, so that the branch from $\vec{O}^{\textrm{bwd},\textrm{out}}$ to
$\vec{P}^{\textrm{bwd},\textrm{out}}$ is in
$\Omega^-$,
while the branch from $\vec{P}^{\textrm{bwd},\textrm{out}}$ to $\vec{Q}^{\textrm{bwd},\textrm{out}}$ is in $\Omega^+$, and they are both
$C^r$.
Further   $Z^{\textrm{bwd},\textrm{in}}\subset E^{\textrm{in}}$ while
$Z^{\textrm{bwd},\textrm{out}} \subset E^{\textrm{out}}$.

We denote by $K^{\textrm{bwd}}$ the compact set enclosed by $Z^{\textrm{bwd},\textrm{out}}$, $L^{+,\textrm{out}}$,
$L^{-,\textrm{out}}$, $Z^{\textrm{bwd},\textrm{in}}$ and the segment of
$\Om^0$ between $\vec{P}^{\textrm{bwd},\textrm{in}}$ and $\vec{R}^{\textrm{bwd},\textrm{in}}$. We denote by $K^{\textrm{bwd},\pm}=
K^{\textrm{bwd}} \cap \Omega^{\pm}$. Again,
by construction, $Z^{\textrm{bwd},\textrm{in}}$ and  $Z^{\textrm{bwd},\textrm{out}}$ do not intersect
each other or themselves and $\bs{\Gamma} \subset K^{\textrm{bwd}}$.

\begin{lemma}\label{lemma0u}
Assume \assump{F0}, \assump{F1}, \mat{\assump{F2}}, \assump{K}, \assump{G}.
We can choose  $\ep_0$ and $\beta_0$ such that for any $0< \ep \le \ep_0$ and any $0< \beta \le \beta_0$ with $\beta \ge \ep^{\frac{\sigma^{\textrm{fb}}}{2}}$ we can define
  $Z^{\textrm{bwd},\textrm{in}}$, $Z^{\textrm{bwd},\textrm{out}}$ so that the flow of \eqref{eq-disc} on
$Z^{\textrm{bwd},\textrm{in}} \setminus \{\P^{\textrm{bwd},\textrm{in}}, \vec{R}^{\textrm{bwd},\textrm{in}}\}$ and on
$Z^{\textrm{bwd},\textrm{out}} \setminus \{\vec{O}^{\textrm{bwd},\textrm{out}}, \vec{Q}^{\textrm{bwd},\textrm{out}}\}$
 aims towards the exterior of $K^{\textrm{bwd}}$.
Then,
for any  $0< \mu \le \mu_0  $ we find
\begin{equation}\label{constant.u}
	\begin{gathered}
  \begin{array}{cc}
       \|\P^{\textrm{bwd},\textrm{in}}-\ga(0)\|= \beta , & \|\vec{P}^{\textrm{bwd},\textrm{out}}-\ga(0)\|=
       \beta , \\
        \beta^{  \sbwd_{-}+\mu} \le\|\Q^{\textrm{bwd},\textrm{in}}\|\le
          \beta^{ \sbwd_{-} -\mu} ,   &
          \beta^{ \michal{\sfwd_{+}} +\mu} \le
   \|\vec{Q}^{\textrm{bwd},\textrm{out}}\| \le  \beta^{ \michal{\sfwd_{+}}  -\mu},   \end{array}
        \\
       	\beta^{ \sbwd  +\mu} \le   \|\vec{R}^{\textrm{bwd},\textrm{in}}-\ga(0)\|\le  \beta^{\sbwd
       -\mu},
        \\  \beta^{ \sbwd_{-} +\mu}\le \|\vec{O}^{\textrm{bwd},\textrm{out}}\|   \le \beta^{ \sbwd_{-} -\mu}.
  \end{gathered}
\end{equation}
 Finally, we can assume that
 $ K^{\textrm{bwd}}   \subset  B(\bs{\Gamma}, \beta^{ \underline{\sigma}  -\mu})$,
 and that there is $c>0$ such that $B(\bs{\Gamma}, c \ep) \cap ( Z^{\textrm{bwd},\textrm{in}}\cup Z^{\textrm{bwd},\textrm{out}})=\emptyset$.
\end{lemma}

  Again, the precise values of the exponents appearing in \eqref{constant.u}  are not so relevant in this paper, and
    if $\la_u^+=\la_u^-=|\la_s^+|=|\la_s^-|$,   \eqref{constant.u} simplifies analogously to \eqref{easyconstant.s}.

We continue with further constructions illustrated again by Figure \ref{Kforward}.

\medskip

 From Lemmas \ref{lemma0s} and \ref{lemma0u}, we know that
 $\tilde{W} (\tau) \subset K^{\textrm{fwd}}\cap K^{\textrm{bwd}}$. Further  $\tilde{W} (\tau)$ splits the compact set
$K^{\textrm{fwd}}(\tau)$ ($K^{\textrm{bwd}}(\tau)$) in two components, denoted by $K^{\textrm{fwd}}_A(\tau)$ and $K^{\textrm{fwd}}_B(\tau)$
($K^{\textrm{bwd}}_A(\tau)$ and
$K^{\textrm{bwd}}_B(\tau)$), respectively
in the interior and in the exterior  of the bounded set enclosed by
$\tilde{W}(\tau)$ and  the branch of $\Omega^0$ between $\vec{P}_u(\tau)$ and $\vec{P}_s(\tau)$.

We denote by $A^{\textrm{fwd},+}(\tau)$ the open branch (i.e.\ without endpoints) of $\Omega^0$ between $\vec{P}^{\textrm{fwd},\textrm{in}}$
and $\vec{P}_s(\tau)$,
 and by
  $A^{\textrm{fwd},-}(\tau)$ the open branch of $\Omega^0$ between $\vec{R}^{\textrm{fwd},\textrm{in}}$ and $\vec{P}_u(\tau)$.
Let $\vec{Q} \in A^{\textrm{fwd},+}(\tau)$ and follow $\x(t,\tau;\Q)$ forward  in $t$ for $t \ge \tau$;
    from Lemma \ref{lemma0s} we get that  $\x(t,\tau;\Q)$ will stay close to $\tilde{W} (t)$
    until it completes a loop and reaches $A^{\textrm{fwd},-}(\mathscr{T}^{\textrm{fwd}}(\Q,\tau))$ at
    $t=\mathscr{T}^{\textrm{fwd}}(\Q,\tau)>\tau$.

Similarly, we denote by $A^{\textrm{bwd},-}(\tau)$ the open branch   of $\Omega^0$ between $\vec{\michal{P}}^{\textrm{bwd},\textrm{in}}$ and
$\vec{P}_u(\tau)$ and
by $A^{\textrm{bwd},+}(\tau)$ the open branch of $\Omega^0$ between $\vec{\michal{R}}^{\textrm{bwd},\textrm{in}}$ and $\vec{P}_s(\tau)$.
Let $\vec{Q} \in A^{\textrm{bwd},-}(\tau)$ and follow $\x(t,\tau;\Q)$ backward  in $t$ for $t \le \tau$;
    from Lemma \ref{lemma0u} we know that  $\x(t,\tau;\Q)$ will stay close to $\tilde{W} (t)$
    until it completes a loop and reaches $A^{\textrm{bwd},+}(\mathscr{T}^{\textrm{bwd}}(\Q,\tau))$ at a suitable
    $t=\mathscr{T}^{\textrm{bwd}}(\Q,\tau)<\tau$.

    Let $K^{\textrm{fwd},\pm}_A(\tau)= K^{\textrm{fwd}}_A (\tau) \cap \Omega^{\pm}$,
     $K^{\textrm{bwd},\pm}_A(\tau)= K^{\textrm{bwd}}_A (\tau) \cap \Omega^{\pm}$,
then we have the following.
   \begin{lemma}\label{L.loop}
  Assume \assump{F0}, \assump{F1}, \mat{\assump{F2}}, \assump{K}, \assump{G}.
    Let $\Q \in A^{\textrm{fwd},+}(\tau)$. Then there are $\mathscr{T}^{\textrm{fwd}}(\Q,\tau)>\tau_1(\Q,\tau)>\tau$ such that
    the trajectory $\x(t,\tau; \Q) \in K^{\textrm{fwd},+}_A(t)$ for any $\tau<t<\tau_1(\Q,\tau)$, $\x(t,\tau; \Q) \in K^{\textrm{fwd},-}_A(t)$
    for any
    $\tau_1(\Q,\tau)<t<\mathscr{T}^{\textrm{fwd}}(\Q,\tau)$,
    and it crosses transversely $\Om^0$ at $t\in\{\tau, \tau_1(\Q,\tau), \mathscr{T}^{\textrm{fwd}}(\Q,\tau)\}$.
     Hence,
     $$\mathscr{P}^{\textrm{fwd}}_+(\Q,\tau):=\x(\tau_1(\Q,\tau),\tau; \Q) \in L^{\inn},$$
     $$\mathscr{P}^{\textrm{fwd}}(\Q,\tau):=\x(\mathscr{T}^{\textrm{fwd}}(\Q,\tau),\tau; \Q) \in
     A^{\textrm{fwd},-}(\mathscr{T}^{\textrm{fwd}}(\Q,\tau)).$$

    Analogously, let $\Q \in A^{\textrm{bwd},-}(\tau)$. Then there are $\mathscr{T}^{\textrm{bwd}}(\Q,\tau)<\tau_{-1}(\Q,\tau)<\tau$ such that
    $\x(t,\tau; \Q) \in K^{\textrm{bwd},-}_A(t)$ for any $\tau_{-1}(\Q,\tau)<t<\tau$, $\x(t,\tau; \Q) \in K^{\textrm{bwd},+}_A(t)$ for any
     $\mathscr{T}^{\textrm{bwd}}(\Q,\tau)<t<\tau_{-1}(\Q,\tau)$,
    and it crosses transversely $\Om^0$ at $t\in\{\tau, \tau_{-1}(\Q,\tau), \mathscr{T}^{\textrm{bwd}}(\Q,\tau)\}$. Hence
     $$ \mathscr{P}^{\textrm{bwd}}_-(\Q,\tau):=\x(\tau_{-1}(\Q,\tau),\tau; \Q) \in L^{\inn},$$
     $$ \mathscr{P}^{\textrm{bwd}}(\Q,\tau):=\x(\mathscr{T}^{\textrm{bwd}}(\Q,\tau),\tau; \Q) \in
     A^{\textrm{bwd},+}(\mathscr{T}^{\textrm{bwd}}(\Q,\tau)).$$
  \end{lemma}

From the smoothness of the flow of \eqref{eq-disc} it follows that all the functions defined in Lemma \ref{L.loop} are $C^r$: we add a sketch
of the proof for completeness.

\begin{remark}\label{r.smooth}
The functions   $\mathscr{P}^{\textrm{fwd}}_+(\Q,\tau)$, $\tau_1(\Q,\tau)$, $\mathscr{P}^{\textrm{fwd}}(\Q,\tau)$,
$\mathscr{T}^{\textrm{fwd}}(\Q,\tau)$, and
$\mathscr{P}^{\textrm{bwd}}_-(\vec{R},\tau)$, $\tau_{-1}(\vec{R},\tau)$ and
$\mathscr{P}^{\textrm{bwd}}(\vec{R},\tau)$, $\mathscr{T}^{\textrm{bwd}}(\vec{R},\tau)$ constructed via
Lemma \ref{L.loop} are $C^r$ in both the variables, respectively
 for $\Q \in A^{\textrm{fwd},+}(\tau)$, $\vec{R} \in A^{\textrm{bwd},-}(\tau)$ and $\tau \in \R$.
\end{remark}
\begin{proof}[Proof of Remark \ref{r.smooth}]
The smoothness of $\mathscr{P}^{\textrm{fwd}}_+(\Q,\tau)$, $\tau_1(\Q,\tau)$ follows from the smoothness of the flow of  \eqref{eq-disc}
on $\Omega^+$.
Let $\Q $ be a point in the branch of $L^{\inn}$ between the origin and $\Q^{\textrm{fwd},\inn}$, and $\tau \in\R$; let   us denote by
$\tau_2(\Q,\tau)$ the time such that
$\x( t  ,\tau; \Q) \in \Om^-$ for any $\tau<t< \tau_2(\Q,\tau)$ and it crosses
transversely $\Om^0$ at $t= \tau_2(\Q,\tau)$. Then $\tau_2(\Q,\tau)$ and $\mathscr{P}^{\textrm{fwd}}_-(\Q,\tau):=\x( \tau_2(\Q,\tau)  ,\tau;
\Q) $
are $C^r$ in both the variables due to the smoothness of the flow of  \eqref{eq-disc}  on $\Omega^-$.

Hence, the maps
\begin{gather*}
	\mathscr{T}^{\textrm{fwd}}(\Q,\tau)= \tau_1(\Q,\tau)+ \tau_2( \mathscr{P}^{\textrm{fwd}}_+(\Q,\tau), \tau_1(\Q,\tau)),\\
	\mathscr{P}^{\textrm{fwd}}(\Q,\tau)= \mathscr{P}^{\textrm{fwd}}_-( \mathscr{P}^{\textrm{fwd}}_+(\Q,\tau), \tau_1(\Q,\tau))
\end{gather*}
are $C^r$
since
they are obtained as a sum and   compositions of   smooth maps.

The smoothness of  $\mathscr{P}^{\textrm{bwd}}_-(\vec{R},\tau)$, $\tau_{-1}(\vec{R},\tau)$ and  $\mathscr{P}^{\textrm{bwd}}(\vec{R},\tau)$,
$\mathscr{T}^{\textrm{bwd}}(\vec{R},\tau)$ can be shown similarly.
 \end{proof}

In fact, using the argument of Remark \ref{r.smooth} with some minor changes we can show the following.
 \begin{remark}\label{data.smooth}
  Let $A$ be an open, connected and bounded subset of $\Omega$, let $\tau_2>\tau_1$ and
  denote by
  $$B(t)= \{ \x(t,\tau_1; \Q) \mid \Q \in A \}.$$
  Assume that in $\Omega^0 \cap B(t)$ there are no sliding phenomena for any $t \in [\tau_1, \tau_2]$.
  Then the functions
  \begin{gather*}
  	\Phi_{\tau_2,\tau_1} : A \to B(\tau_2), \qquad \quad  \Phi_{\tau_1,\tau_2} :   B(\tau_2)\to A,\\
  	\Phi_{\tau_2,\tau_1}(\Q)= \x(\tau_2,\tau_1; \Q)
  ,\qquad \Phi_{\tau_1,\tau_2}=\Phi_{\tau_2,\tau_1}^{-1}
  \end{gather*}
  are  homeomorphisms.

  Assume further that $A \cap \Omega^0=\emptyset$, $B(\tau_2) \cap \Om^0= \emptyset$, and
  that for any $\Q\in A$, if $\x(\bar{t},\tau_1; \Q) \in \Om^0$ for some $\bar{t} \in (\tau_1, \tau_2)$, then it crosses $\Om^0$
  transversely. Then $\Phi_{\tau_2,\tau_1} $ and $\Phi_{\tau_1,\tau_2}$  are $C^r$ diffeomorphisms.
\end{remark}

\begin{remark}\label{r.in1}
Let us denote by $\tilde{V}(\tau)$ the compact connected set enclosed by
$\tilde{W}(\tau)$ and by the branch of $\Om^0$ between $\P_u(\tau)$ and $\P_s(\tau)$.
If $\Q \in A^{\textrm{fwd},+}(\tau)$, then $\x(t,\tau; \Q) \in \tilde{V}(t)$ for any
$t \in [\tau, \TTT^{\textrm{fwd}}(\Q,\tau)]$.
\\
Analogously if $\Q \in A^{\textrm{bwd},-}(\tau)$, then $\x(t,\tau; \Q) \in \tilde{V}(t)$ for any
$t \in [ \TTT^{\textrm{bwd}}(\Q,\tau), \tau]$.
\end{remark}

\begin{remark}\label{out}
Denote by $B^{\textrm{fwd},+}(\tau)$ the branch of $\Omega^0$ between    $\vec{P}_s(\tau)$ and $\vec{P}^{\textrm{fwd,out}}$ and
by $B^{\textrm{bwd},-}(\tau)$ the branch of $\Omega^0$ between    $\vec{P}_u(\tau)$ and $\vec{P}^{\textrm{bwd,out}}$.

  	Observe that, if $\Q \in B^{\textrm{fwd},+}(\tau)$, there is $\tau_1(\Q)>\tau$ such that
    $\x(t,\tau; \Q) \in K^{\textrm{fwd},+}_B(t)$ for any $\tau<t<\tau_1(\Q)$, it crosses transversely $L^{+,\textrm{out}}$ at $t=\tau_1(\Q)$
    and leaves a neighborhood of $\bs{\Gamma}$ at some $t>\tau_1(\Q)$.

    Analogously, if $\Q \in B^{\textrm{bwd},-}(\tau)$,  there is $\tau_{-1}(\Q)<\tau$ such that
    $\x(t,\tau; \Q) \in K^{\textrm{bwd},-}_B(t)$ for any $\tau_{-1}(\Q)<t<\tau$,
    it crosses transversely $L^{-,\textrm{out}}$ at $t=\tau_{-1}(\Q)$
    and leaves a neighborhood of $\bs{\Gamma}$ at some $t<\tau_{-1}(\Q)$.
\end{remark}

 \section{Statement of the main results}\label{s.lemmakey}

In this section we need to measure the distance between the points on $\Om^0$, and to be able to determine their mutual positions.
Since $\Om^0$ is a regular curve we can define a directed distance for points in $\Om^0$ by arc length, once an orientation is fixed.
We choose as the positive orientation on $\Om^0$ the one that goes from the origin to $\ga(0)$.
So, for any $\Q\in L^0$ we define
$\ell(\Q)=\int_{\Om^0(\vec{0},\Q)}ds>0$ where $\Om^0(\vec{0},\Q)$ is the (oriented) path of $\Om^0$ connecting $\vec{0}$ with $\Q$,
and we define the directed distance
\begin{equation}\label{dist}
    \dist(\Q,\P):= \ell(\P)-\ell(\Q)
\end{equation}
for $\Q,\P\in L^0$.
Notice that $\dist(\Q,\P)>0$
 means that $\Q$ lies on $\Om^0$ between $\vec{0}$ and $\P$.
Now, we introduce some further crucial notation.

 \textbf{Notation.}
We denote by $\Q_s(d,\tau)$ the point in $A^{\textrm{fwd},+}(\tau)$ such that
\[
\dist(\Q_s(d,\tau), \P_s(\tau))=d>0,
\]
and by $\Q_u(d,\tau)$ the point in $A^{\textrm{bwd},-}(\tau)$ such that
\[
\dist(\Q_u(d,\tau), \P_u(\tau))=d>0.
\]

 We introduce a further small parameter $\delta>0$. This parameter can be chosen
independently of $\ep>0$, but we   need to set $\delta< \frac{\beta}{2}$.
Then we will always assume $0<d \le \delta$.

\begin{remark}\label{RJ0}
	We explicitly notice that, if $0<d \le \delta$  then $\Q_s(d,\tau)\in A^{\textrm{fwd},+}(\tau)$,
	i.e.,  the open branch of $\Omega^0$ between $\vec{P}^{\textrm{fwd},\textrm{in}}$ and $\vec{P}_s(\tau)$. Indeed
	$$\|\Q_s(d,\tau)-\P_s(\tau)\|\leq |\dist(\Q_s(d,\tau),\P_s(\tau))|=d
	\leq \delta < \beta/2;$$
	further from $\|\P_s(\tau)-\ga(0)\|=O(\ep)$ one can see that
	\begin{equation*}
		\|\Q_s(d,\tau)- \ga(0)\|\le \|\P_s(\tau)-\ga(0)\|+\|\Q_s(d,\tau)-\P_s(\tau)\|< \beta,
	\end{equation*}
 	since $\ep \ll \beta$;   while by construction
	 $\|\P^{\textrm{fwd},\textrm{in}}-\ga(0)\|=  \beta$.
	Hence for $\ep$ small enough $\Q_s(d,\tau)\in A^{\textrm{fwd},+}(\tau)$.
\\
 Similarly we see that if $0<d \le \delta$ then $\Q_u(d,\tau)\in A^{\textrm{bwd},-}(\tau)$,
	 i.e., $\Q_u(d,\tau)$ lies on the open branch of $\Omega^0$ between $\vec{P}^{\textrm{bwd},\textrm{in}}$ and $\vec{P}_u(\tau)$.
\end{remark}

From Lemma \ref{L.loop}, we see that for any $\tau \in \R$ and any    $0<d \le \delta$
we can define the maps
\begin{equation}\label{Q1T1}
\begin{gathered}
\mathscr{T}_1(d,\tau):=   \mathscr{T}^{\textrm{fwd}}(\Q_s(d,\tau),\tau) \, , \qquad \mathscr{P}_1(d,\tau):=
\mathscr{P}^{\textrm{fwd}}(\Q_s(d,\tau),\tau),\\
\mathscr{T}_{-1}(d,\tau):=   \mathscr{T}^{\textrm{bwd}}(\Q_u(d,\tau),\tau) \, , \qquad \mathscr{P}_{-1}(d,\tau):=
\mathscr{P}^{\textrm{bwd}}(\Q_u(d,\tau),\tau).
\end{gathered}
\end{equation}
Sometimes we will also make use of the maps
\begin{equation}\label{Q1T1half}
\begin{gathered}
\mathscr{T}_{\frac{1}{2}}(d,\tau):=   \tau_1(\Q_s(d,\tau),\tau) \, , \qquad \mathscr{P}_{\frac{1}{2}}(d,\tau):=
\mathscr{P}^{\textrm{fwd}}_+(\Q_s(d,\tau),\tau),\\
\mathscr{T}_{-\frac{1}{2}}(d,\tau):=   \tau_{-1}(\Q_u(d,\tau),\tau) \, , \qquad \mathscr{P}_{-\frac{1}{2}}(d,\tau):=
\mathscr{P}^{\textrm{bwd}}_-(\Q_u(d,\tau),\tau).
\end{gathered}
\end{equation}

Notice that by construction the trajectory $\x(t,\tau; \Q_s(d,\tau))$ is in $K^{\textrm{fwd},+}_A(t)$ for any $\tau<t<
\mathscr{T}_{\frac{1}{2}}(d,\tau)$ and it is
in
$ K^{\textrm{fwd},-}_A(t)$ for any $\mathscr{T}_{\frac{1}{2}}(d,\tau)<t< \mathscr{T}_{1}(d,\tau)$; further
  it intersects transversely $L^{\inn}$ at $t= \mathscr{T}_{\frac{1}{2}}(d,\tau)$ and      $A^{\textrm{fwd},-}(\mathscr{T}_1(d,\tau))\subset
  L^0$ at $t=
  \mathscr{T}_1(d,\tau)$.
Similarly, the trajectory $\x(t,\tau; \Q_u(d,\tau))$ is in $K^{\textrm{bwd},-}_A(t)$ for any $
\mathscr{T}_{-\frac{1}{2}}(d,\tau)<t<\tau$
and it is in $K^{\textrm{bwd},+}_A(t)$ for any $\mathscr{T}_{-1}(d,\tau)<t< \mathscr{T}_{-\frac{1}{2}}(d,\tau)$; further
  it intersects transversely $L^{\inn}$ at $t= \mathscr{T}_{-\frac{1}{2}}(d,\tau)$ and  $A^{\textrm{bwd},+}(\mathscr{T}_{-1}(d,\tau))\subset
  L^0$ at $t=
  \mathscr{T}_{-1}(d,\tau)$.

 \begin{theorem}\label{key}
Assume \assump{F0}, \assump{F1}, \mat{\assump{F2}}, \assump{K}, \assump{G}  and    let $\f^\pm$ and $\g$ be $C^r$ with $r>1$.
We can find $\ep_0>0$, $\delta>0$,  such that for any $0<\ep \le \ep_0$,
the functions $\TTT_{\pm 1}(d,\tau)$, $\PPP_{\pm 1}(d,\tau)$  are $C^r$ when
	$0<d \le \delta$ and $\tau \in \R$.
 Further, for any $0< \mu <\mu_0$ we get
	\begin{equation}\label{D1T1}
	\begin{split}
	 	   d^{\sfwd +\mu} & \le \dist (\PPP_1(d,\tau),\P_u(\TTT_1(d,\tau)))
 \le d^{\sfwd-\mu}, \\
	 d^{\sbwd+\mu} & \le \dist (\PPP_{-1}(d,\tau), \P_s (\TTT_{-1}(d,\tau)))
 \le d^{\sbwd-\mu},\\
 \|\PPP_{\frac{1}{2}}(d,\tau)\| & \le d^{\sfwd_+ -\mu}     , \qquad    \|\PPP_{-\frac{1}{2}}(d,\tau)\| \le d^{\sbwd_- -\mu},
	\end{split}
	\end{equation}
	\begin{equation}\label{T1T-1}
	\begin{split}
	 	 \left[\sTfwd -\mu \right] |\ln(d)|& \le  (\TTT_1(d,\tau)-\tau)    \le   \left[\sTfwd+\mu \right] |\ln(d)|, \\
	 \left[\sTbwd -\mu \right] |\ln(d)|
& \le  \tau-\TTT_{-1}(d,\tau)    \le    \left[\sTbwd+\mu \right] |\ln(d)|, \\
 \left[\sTfwd_+ -\mu \right] |\ln(d)|& \le  (\TTT_\frac{1}{2}(d,\tau)-\tau)   \le   \left[\sTfwd_++\mu \right] |\ln(d)|, \\
	 \left[\sTbwd_- - \mu \right] |\ln(d)|
& \le  \tau -\TTT_{-\frac{1}{2}}(d,\tau)     \le    \left[\sTbwd_-+\mu \right] |\ln(d)|,
	\end{split}
	\end{equation}
	and all the expressions in \eqref{D1T1} are uniform with respect to any   $\tau \in \R$ and $0< \ep \le \ep_0$.
\end{theorem}

In fact, with the same argument developed in \S \ref{S.key}, we are able to estimate the positions of the trajectories
$\x(t,\tau; \Q_s(d,\tau))$ and $\x(t,\tau; \Q_u(d,\tau))$ for any $t \in [\tau, \TTT_1(d,\tau)]$ and $t \in [ \TTT_{-1}(d,\tau), \tau]$,
respectively.

\begin{theorem}\label{keymissed}
Assume \assump{F0}, \assump{F1}, \mat{\assump{F2}}, \assump{K}, \assump{G}  and    let $\f^\pm$ and $\g$ be $C^r$ with $r>1$.
We can find $\ep_0>0$, $\delta>0$,  such that for any $0<\ep \le \ep_0$,   $0<d \le \delta$,   $0< \mu <\mu_0$   and any $\tau \in \R$
   we find
  \begin{equation}\label{keymissed.es-}
   \|\x(t,\tau; \Q_s(d,\tau))-\x(t,\tau; \P_s(\tau))\| \le  d^{\sfwd_+ - \mu }
  \end{equation}
  for any $\tau \le t \le \TTT_{\frac{1}{2}}(d,\tau)$,   and
  \begin{equation}\label{keymissed.es+}
  \|\x(t,\tau; \Q_s(d,\tau))-\x(t,\TTT_1(d,\tau); \P_u(\TTT_1(d,\tau)))\| \le  d^{\mat{\sfwd_-} - \mu }
  \end{equation}
    for any $\TTT_{\frac{1}{2}}(d,\tau) \le t \le \TTT_1(d,\tau)$.

Similarly, for any $0<\ep \le \ep_0$,   $0<d \le \delta$,   $0< \mu <\mu_0$   and any $\tau \in \R$  we find
\begin{equation}\label{keymissed.eu+}
	\|\x(t,\tau; \Q_u(d,\tau))-\x(t,\tau; \P_u(\tau))\| \le  d^{\sbwd_- - \mu}
\end{equation}
  for any $   \TTT_{-\frac{1}{2}}(d,\tau)  \le t \le \tau$,    and
  \begin{equation}\label{keymissed.eu-}
	\|\x(t,\tau; \Q_u(d,\tau))-\x(t,\TTT_{-1}(d,\tau); \P_s(\TTT_{-1}(d,\tau)))\| \le  d^{\mat{\sbwd_-} - \mu}
\end{equation}
    for any $\TTT_{-1}(d,\tau) \le t \le \TTT_{-\frac{1}{2}}(d,\tau)$.
\end{theorem}

\section{Proofs of Theorems \ref{key}  and \ref{keymissed}}\label{S.key}
 In this section \mat{we assume all the hypotheses of Theorems \ref{key}  and \ref{keymissed} without further mentioning,}  and we adopt the notation introduced in \S \ref{S.prel};
 we recall
  that $\f^{\pm}$ and $\g$ are $C^r$ with   $r \ge 1+\alpha$ for some $0<\alpha \le 1$, i.e., their derivatives are H\"older continuous.

For  simplicity we denote by
$$\vec{F}^\pm(\x,t,\ep)=\f^\pm(\x)+\ep\g(t,\x,\ep), \quad \x \in \Omega^{\pm} \cup \Omega^0$$
 and we assume the following:
\begin{description}
  \item[H] There are $0< \alpha \le 1$ and $N_{\al}>0$ such that,
 for $\| \vec{h} \|<1$,
 \begin{equation*}
\sup \{ \|\boldsymbol{F^{\pm}_{x}}(\P +\vec{h},t,\ep)- \boldsymbol{F^{\pm}_{x}}(\P,t,\ep) \|    \mid \vec{P} \in B(\bs{\Gamma},1) , \,
t \in \R, \,    |\ep| \le 1   \} \le N_{\al} \| \vec{h} \|^{\alpha}.
\end{equation*}
\end{description}

The  proofs of Theorems \ref{key}  and \ref{keymissed} is technical and lengthy and it is divided in several steps.

\begin{figure}[t]
\centerline{\epsfig{file=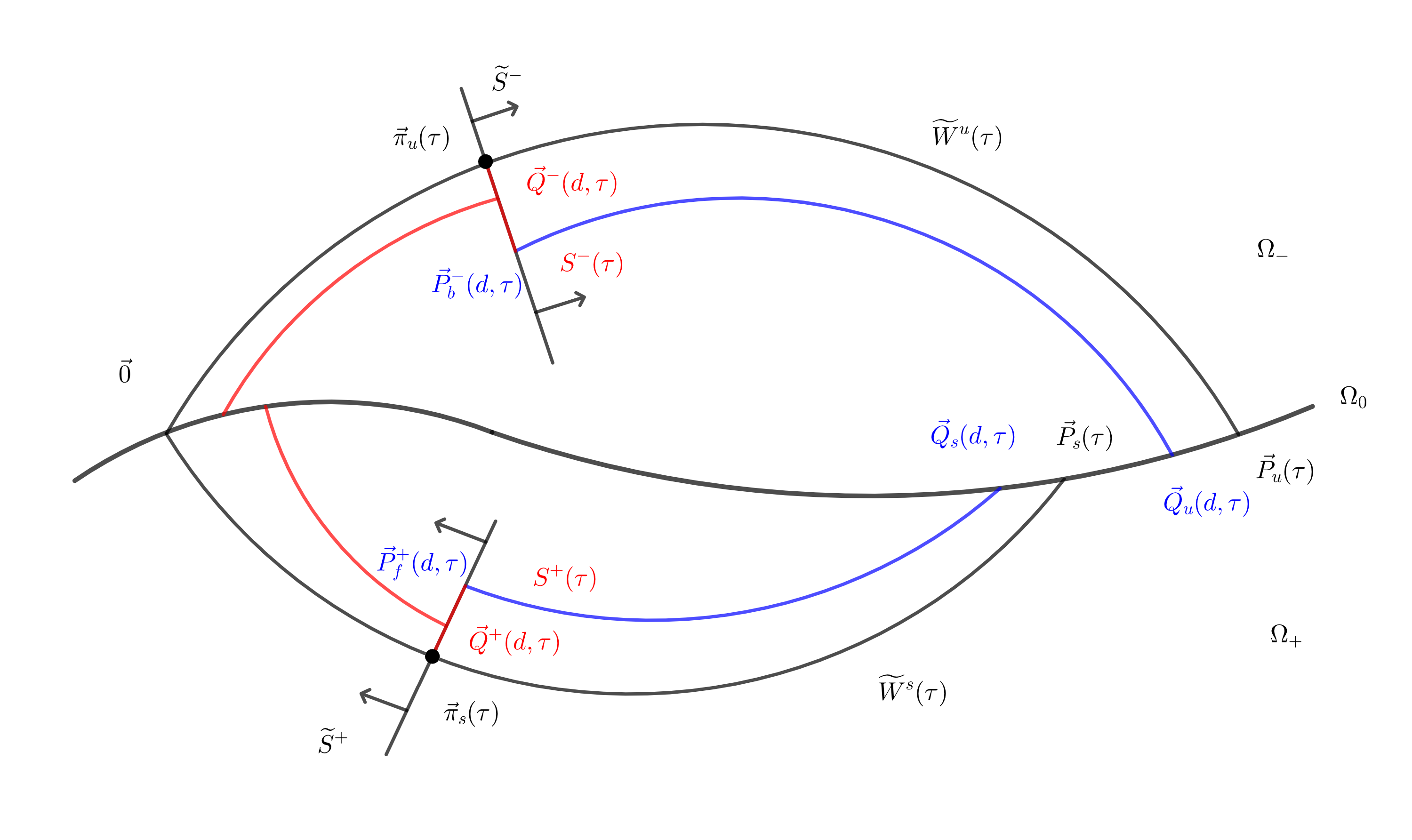, width = 9 cm} }
\caption{A scheme of the proofs of Lemmas \ref{key} and \ref{keymissed}.}
\label{fig.key}
\end{figure}

 Adapting \cite{FrPo}  we introduce a small auxiliary parameter
 $\varpi$,  $\varpi>\delta$,
  see Remark \ref{oneparam}, whose value will be fixed below, and the following set depending on it:
\begin{equation}\label{S-S+}
  \tilde{S}^+:= \left\{ d  \vec{v}_u^+  + \frac{ \vec{v}_s^+}{|\ln(\varpi)|}     \mid  |d| \le \frac{1}{|\ln(\varpi)|}
  \right\},
\end{equation}
 see Figure \ref{fig.key}.
Let us denote by $\PPs(\tau)$ the unique intersection between $\tilde{W}^s(\tau)$ and
$\tilde{S}^+$, whose existence follows from the fact that $\tilde{W}^s(\tau)$ is locally a graph on its tangent, see also Remark \ref{est.ys}.
Moreover we let $\y_s(\theta)=\x(\theta+\tau,\tau; \PPs(\tau))$;
notice that $\y_s(\theta) \in   K^{\textrm{fwd},+}$ for any $\theta \ge 0$ and $\lim_{\theta \to +\infty} \y_s(\theta)=\vec{0}$,
 see Figure \ref{Kforward}.
Then we denote by
\begin{equation}\label{S-S+tau}
  S^+(\tau):= \left\{ -d \vec{v}_u^+ + \PPs(\tau)     \mid  0<d\le \delta  \right\}.
\end{equation}
Since $\PPs(\tau)= \frac{\vec{v}_s^+}{|\ln(\varpi)|}+ O\left(\frac{1}{|\ln(\varpi)|^{1+\alpha}}    \right)$,
then $S^+(\tau) \subset \tilde{S}^+$.

Analogously we denote by
\begin{equation}\label{S+}
  \tilde{S}^-:= \left\{ d  \vec{v}_s^- + \frac{ \vec{v}_u^- }{|\ln(\varpi)|}     \mid  |d| \le \frac{1}{|\ln(\varpi)|}
  \right\}.
\end{equation}
Let $\PPu(\tau)$ be the unique intersection between $\tilde{W}^u(\tau)$ and
$\tilde{S}^-$, whose existence follows from the fact that $\tilde{W}^u(\tau)$ is locally a graph on its tangent.
We let $\y_u(\theta)=\x(\theta+\tau,\tau;\PPu(\tau))$, and
we observe that $\y_u(\theta) \in K^{\textrm{bwd},-}$ for any $\theta \le 0$ and $\lito \y_u(\theta)=\vec{0}$.
Then we set
\begin{equation}\label{S+tau}
  S^-(\tau):= \left\{ -d  \vec{v}_s^-  + \PPu(\tau)     \mid  0<d\le \delta  \right\}.
\end{equation}
Again by construction $S^-(\tau) \subset \tilde{S}^-$.

Roughly speaking, to prove the results of this section we consider a trajectory performing a loop and we estimate
flight time and space displacement: for this purpose we distinguish four different parts of the loop.
 Firstly in \S \ref{loop.easy} we follow the trajectories from $L^0$ to $\tilde{S}^+$
forward in time (and we deduce what happens backward in time  from $\tilde{S}^+$ to $L^0$): we get estimates of the fly time and the
displacement
with respect to $\tilde{W}^s$. Then with the same argument we follow the trajectories from $L^0$ to $\tilde{S}^-$  backward in time
 (and forward from $\tilde{S}^-$ to $L^0$) and we estimate  fly time and  displacement
with respect to $\tilde{W}^u$. In the proofs we will use the fact that the trajectories are close to $\tilde{W}$,
which is in turn close to~$\bs{\Gamma}$.

Then, in \S \ref{S-Lin} we follow the trajectories going from $\tilde{S}^+$ to $L^{\textrm{in}}$, forward and backward in time,
and we
evaluate
fly time and displacement from $\tilde{W}^s$.
Then in \S \ref{S+Lin},
 with an argument
analogous to the one of  \S \ref{S-Lin}, we follow  trajectories going from $\tilde{S}^-$ to $L^{\textrm{in}}$, forward and
backward in
time and we evaluate
fly time and displacement from $\tilde{W}^u$.
Eventually, putting all these results together, we get an estimate of the distance of the
 trajectories from $\tilde{W}$ and we prove Theorems \ref{key} and   \ref{keymissed}.

 \begin{remark}\label{oneparam}
In this section we introduce several small parameters, which need to be chosen in the right order.

 Equation \eqref{eq-disc} assigns
   $\ep>0$, which measures the size of the perturbation, and the constant   $\mu_0$   which depends only
   on the eigenvalues $\la^{\pm}_{s}$, $\la^{\pm}_{u}$,  see   \eqref{mu0}.

Then we introduce  a parameter $\beta> \ep^{\sigma^{\textrm{fb}}/2}$ and, via Lemma \ref{L.loop},
 we deduce the existence of $\ep_0$ and $\beta_0$ such that for any $0<\ep \le \ep_0$ and any $ 0< \ep^{\sigma^{\textrm{fb}}/2}< \beta \le \beta_0$
 we can construct the   maps $\PPP^{\textrm{fwd}}_+$, $\PPP^{\textrm{fwd}}$, $\PPP^{\textrm{bwd}}_-$ $\PPP^{\textrm{bwd}}$.

Nextly we set $\mu \in ]0, \mu_0]$ which   measures the size of the errors that we make in the estimates of flight time and displacement appearing in
 Theorems \ref{key} and~\ref{keymissed}.

In this section we introduce a further small parameter $\varpi \in ]0,1[$   which measures the distance between
the
origin and the lines $\tilde{S}^+$
and $\tilde{S}^-$, transversal to $\bs{\Gamma}$
(these distances are of order $|\ln(\varpi)|^{-1}$). This parameter is needed just for the proofs of Theorems \ref{key} and
\ref{keymissed} and it does not appear in the statement of the results.

Then we choose
 $\delta= \delta(\varpi,\ep) \in ]0,\varpi[$: $\delta$ measures the displacement of the trajectories performing a loop with respect to
 $\tilde{W}$:
in fact each trajectory appearing in Theorems \ref{key} and   \ref{keymissed}
 has distance $O(d)$ from  $\tilde{W}$ when it crosses $L^0$, where $0<d \le \de$. Finally we possibly reduce the size of $\ep_0$ so
 that the fixed point arguments described in \S \ref{loop.easy} and  \S \ref{S-Lin} work.

Let us spend a few more words about  the order in which these parameters need to be chosen. The parameter $\ep$ is assigned from
\eqref{eq-disc}; we choose $\beta_0>0$ and $\ep_0>0$ so that  Lemma \ref{L.loop} holds for any  $0< \ep^{\sigma^{\textrm{fb}}/2}<\beta \le \beta_0$ and
any $0< \ep \le \ep_0$, and the maps    appearing in Lemma \ref{L.loop},  Theorems \ref{key} and   \ref{keymissed}
are well defined.

Then we  choose $\mu>0$: in fact we have to choose a parameter $\mu_1>0$ which is needed for the estimates
of the trajectories going from $L^0$ to $\tilde{S}^+$ and from  $\tilde{S}^-$ to $L^0$, and a parameter $\mu_2>0$ which is
needed for the
estimates
of the trajectories going from   $\tilde{S}^+$ to $L^{\inn}$ and from $L^{\inn}$ to $\tilde{S}^-$.
Then we set $\mu_2=2 \mu_1$
 and $\mu= c_{\mu} \mu_2$ where $c_{\mu}  \geq 7$ is a constant
 (depending only on $\la^{\pm}_{s}$, $\la^{\pm}_{u}$, see \eqref{def-cmu} below),  and we need
to choose $\mu$ so that $0< \mu \le \mu_0$.  So, the choice of $\mu$ prescribes the values of $\mu_1$ and $\mu_2$.

Then we choose $\varpi=\varpi(\mu,\ep)$, nextly $\delta=\delta(\varpi, \mu,\ep)$;  finally
  for any fixed $0<\mu \le \mu_0$ we choose $\ep_0=\ep_0(\mu)$ so that our argument works for any $0< \ep \le \ep_0$.
 We think that in fact the estimates might be improved and it could be shown that $\mu$ could be chosen as a linear function of $\ep$, but
 this is
beyond the interest of this paper, see Remark \ref{totrash}. In fact we believe that one can improve further the estimates
if $\gx(t,0,\ep) \equiv  \bs{0}$.
\end{remark}

\begin{remark} \label{cumbersome}
By Lemma \ref{L.loop} we know that if $\Q \in A^{\textrm{fwd},+}(\tau)$, then $\x (t,\tau;\Q)$ is confined in the region
$K_A^{\textrm{fwd},+}(t)$ for any
$0<t<\tau_1(\Q,\tau)$.
To simplify the presentation and avoid cumbersome notation we will tacitly assume that the systems
\eqref{eq-disc+} and (\ref{eq-lin})$_+$
are well defined, through any smooth extension, even when $\x \in \Om^0 \cup \Om^-$ is ``sufficiently close'' to $\Om^+$. In fact, the choice
of the extension does not affect
our argument since we just focus on $\Om^0 \cup \Om^+$.

We will maintain the same kind of assumption when dealing with $\x (t,\tau;\Q)$ and $\Q \in A^{\textrm{bwd},-}(\tau)$.
\end{remark}

Before starting the actual proof we recall some facts
concerning exponential dichotomy.
The   roughness of exponential dichotomy (see \cite[Proposition~1, \S 4]{Co} and also \cite[Appendix]{CDFP})
yields that the linear systems
\begin{subequations}\label{eq-lin}
	\begin{equation}\tag*{(\ref{eq-lin})$_\pm$}
		\dot{\x}(t)=\boldsymbol{F^\pm_x}(\vec{0},t,\ep) \x(t)
	\end{equation}
\end{subequations}
admit exponential dichotomy on the whole of $\R$ with projections
$\Pe$, and exponents $\la_s^\pm+k\ep<0<\la_u^\pm-k\ep$ \michal{for some $k>0$}.

 Keeping in mind  the definitions of $\underline{\la}$ and $\overline{\la}$ given in \eqref{defsigma},
we shall assume w.l.o.g.~that $k\ep_0<\frac{1}{4}\underline{\la}$, so that
$$-2\overline{\la}< \la_s^\pm-k\ep<\la_s^\pm+k\ep<-
\frac12\underline{\la}<0<\frac12\underline{\la}<\la_u^\pm-k\ep<\la_u^\pm+k\ep<2\overline{\la}.$$

 In this paper we use the version of exponential dichotomy introduced in \cite[Appendix]{CDFP} in which we estimate the projections
 using actual solutions of \eqref{eq-lin}$_\pm$.

 In the whole section  (unless otherwise stated)
 $\ep>0$ is fixed: so from now on we leave this dependence unsaid.

Let us denote by
$\vec{w}_u^{\pm}(t)$ and by $\vec{w}_s^{\pm}(t)$ the unique solutions of  the linear equations \eqref{eq-lin}$_\pm$
satisfying the following conditions
\begin{equation}\label{wupm}
\begin{split}
    \|\vec{w}_u^{\pm}(0)\| = & 1 , \qquad \lito  \|\vec{w}_u^{\pm}(t)\|=0, \quad
    \lito \left( \frac{\vec{w}_u^{\pm}(t)}{\|\vec{w}_u^{\pm}(t)\|}\right)^* \, \vec{v}_u^{\pm}  >0,
\\
    \|\vec{w}_s^{\pm}(0)\| = & 1 , \qquad \lit  \|\vec{w}_s^{\pm}(t)\|=0,
     \quad
    \lit \left( \frac{\vec{w}_s^{\pm}(t)}{\|\vec{w}_s^{\pm}(t)\|} \right)^* \, \vec{v}_s^{\pm}  >0.
\end{split}
\end{equation}
Let us denote by $\vec{v}_u^{\pm}(t)$ and by $\vec{v}_s^{\pm}(t)$ the unit vectors
$\vec{v}_u^{\pm}(t)=\frac{ \vec{w}_u^{\pm}(t) }{\|\vec{w}_u^{\pm}(t)\|}$ and
$\vec{v}_s^{\pm}(t)=\frac{ \vec{w}_s^{\pm}(t) }{\|\vec{w}_s^{\pm}(t)\|}$.
Then, following \cite[\S 6.2]{CDFP} we define
\begin{equation}\label{defzus}
  z_u^{\pm}(t,s)=  \frac{\|\vec{w}_u^{\pm}(t)\|}{\|\vec{w}_u^{\pm}(s)\|} \, , \qquad    z_s^{\pm}(t,s)=
  \frac{\|\vec{w}_s^{\pm}(t)\|}{\|\vec{w}_s^{\pm}(s)\|}.
\end{equation}
Note that for any $t,r,s \in \R$ we have
\[
z_s^{\pm}(t,s)=z_s^{\pm}(t,r)z_s^{\pm}(r,s) \, , \qquad z_u^{\pm}(t,s)=z_u^{\pm}(t,r)z_u^{\pm}(r,s).
\]
Moreover, there are $C>0$  and $k_1 \geq 1$  such that for any $t,s \in \R$ we have
\begin{equation}\label{defzused}
\begin{gathered}
     \|\vec{v}_u^{\pm}(t) - \vec{v}_u^{\pm}\| \le C \ep \,, \qquad   \|\vec{v}_s^{\pm}(t) - \vec{v}_s^{\pm}\| \le C \ep\,, \\
  (k_1)^{-1}\eu^{\la_u^{\pm}(t-s)-k \ep|t-s|} \le  z_u^{\pm}(t,s) \le k_1 \eu^{\la_u^{\pm}(t-s)+k \ep|t-s|}\, , \\
  (k_1)^{-1}  \eu^{\la_s^{\pm}(t-s)-k \ep|t-s|} \le  z_s^{\pm}(t,s) \le k_1 \eu^{\la_s^{\pm}(t-s)+k \ep|t-s|}\,.
\end{gathered}
\end{equation}
So that, for any  $t<s$ we find
\begin{equation}\label{defk}
    z_u^{\pm}(t,s) \le k_1 \eu^{-\frac12\underline{\la} |t-s|} \, , \qquad
   z_s^{\pm}(s,t) \le k_1 \eu^{-\frac12\underline{\la} |t-s|} \,.
\end{equation}

 Denote by $\Xe(t)$ the fundamental solutions of \eqref{eq-lin}$_\pm$ satisfying $\Xe(0)=\I$.
We also need to define the shifted projections
$$\Pe(\tau):=\Xe(\tau)\Pe(\Xe(\tau))^{-1}.$$
Notice that both $\|\Pe (\tau)\|$ and $\|\I-\Pe(\tau) \|$ are bounded, so we can set

  $$k_2:= 2\sup [\{ \|\bs{P^{\pm,0}}(\tau)\| \mid \tau \in\R \} \cup \{ \|\bs{\I- P^{\pm,0}}(\tau) \| \mid \tau \in\R \}]$$ so that
$\max \{ \|\Pe(\tau)\|; \|\I-\Pe(\tau) \|\} \le k_2$  for any $\ep>0$ small enough and any $\tau \in \R$.
We need the following result.
\begin{lemma}\label{newdichot}
For any $s,t \in \R$ and any $\xxi \in \R^2$ we have the following.
\begin{gather*}
	\|\Xe(t)\Pe(\Xe(s))^{-1} \xxi \| \le k_2 z_s(t,s) \|\xxi\|,\\
	\|\Xe(t)(\I -\Pe)(\Xe(s))^{-1} \xxi \| \le k_2 z_u(t,s) \|\xxi\|
\end{gather*}
for some $k_2>0$.
\end{lemma}
 \begin{proof}
 Let $\bs{A}$ be a linear operator, we denote by  $\mathcal{R}\bs{A}$ its range and by $\mathcal{N}\bs{A}$ its kernel.
   Fix $\tau \in \R$, then for any $\xxi \in \R^2$ we find
   $\xxi^{\pm}_s(\tau) \in \textrm{span}[ \vec{v}_s^{\pm}(\tau)] = \mathcal{R}\Pe(\tau)$ and
   $\xxi^{\pm}_u(\tau) \in \textrm{span}[ \vec{v}_u^{\pm}(\tau)] = \mathcal{N}\Pe(\tau)$ such that
   $\xxi= \xxi^{\pm}_s(\tau)+ \xxi^{\pm}_u(\tau)$. Further $\|\xxi^{\pm}_s(\tau)\| \le k_2 \| \xxi\|$ and
   $\|\xxi^{\pm}_u(\tau)\| \le k_2 \| \xxi\|$.

   Next by the invariance of the projections of exponential dichotomy we have
   \begin{gather*}
   		\Xe(t)\Pe(\Xe(s))^{-1} = \Xe(t) (\Xe(s))^{-1}\Pe(s),\\
   		\Xe(t)[\I -\Pe](\Xe(s))^{-1} = \Xe(t) (\Xe(s))^{-1}[\I -\Pe(s)].
   	\end{gather*}
   Hence for any $\xxi \in \R^2$ we find
   \begin{equation*}
     \begin{split}
         \|\Xe(t)\Pe(\Xe(s))^{-1} \xxi \| &= \|\Xe(t) (\Xe(s))^{-1} \xxi^{\pm}_s(s)\| \\
         &=z_s(t,s) \|\xxi^{\pm}_s(s)\| \le
  k_2 z_s(t,s)  \| \xxi \|, \\
            \|\Xe(t)[\I -\Pe](\Xe(s))^{-1} \xxi \| &= \|\Xe(t) (\Xe(s))^{-1} \xxi^{\pm}_u(s)\| \\
         &=z_u(t,s) \|\xxi^{\pm}_u(s)\| \le
   k_2    z_u(t,s) \| \xxi \| ,
     \end{split}
   \end{equation*}
   so the lemma is proved.
 \end{proof}

Using standard arguments from exponential dichotomy theory  (see again \cite[\S 6.2]{CDFP}),
we get that  there is $c_k>1$ such that
\begin{equation}\label{est.stableunstable}
\begin{split}
   &  (c_k)^{-1}\eu^{-2\overline{\la}\theta_s}< \|\x(\theta_s+\tau,\tau; \P_s(\tau))\| < c_k \eu^{-\frac12\underline{\la}\theta_s}, \\
  & (c_k)^{-1}\eu^{-2\overline{\la}|\theta_u|} < \|\x(\theta_u+\tau,\tau; \P_u(\tau))\| < c_k \eu^{-\frac12\underline{\la}|\theta_u|}
\end{split}
\end{equation}
whenever $\theta_u \le 0 \le \theta_s$, for any $\tau \in \R$.

\subsection{The loop: from $\boldsymbol{L^0}$ to $\boldsymbol{S^+(\tau)}$ and from
$\boldsymbol{S^-(\tau)}$ to $\boldsymbol{L^0}$}\label{loop.easy}

In the whole section  $\varpi>0$ is a small parameter  and we choose $\delta=\delta(\varpi)<\varpi$
according to Lemmas \ref{defTf-}, \ref{defTb+},  and then we let $0<d \le \delta$.
We recall that $\P_s(\tau)$ is the unique intersection between $\tilde{W}^s(\tau)$ and $L^0$ while
$\P_u(\tau)$ is the unique intersection between $\tilde{W}^u(\tau)$ and $L^0$; further $\Q_s(d,\tau)$ and $\Q_u(d,\tau)$
are as defined at the beginning of \S \ref{s.lemmakey}.
We introduce the following notation
\begin{equation}\label{Q-Q+}
  \begin{split}
    \Q^+(d,\tau)= &-d \vec{v}_u^+ + \PPs(\tau) \in S^+(\tau) \, ,\\
     \Q^-(d,\tau)= & -d \vec{v}_s^- + \PPu(\tau) \in S^-(\tau) \, .
  \end{split}
\end{equation}

We focus firstly on the trajectory $\x(t,\tau; \Q_s(d,\tau))$
 going from $L^0$ to $\tilde{S}^+$.
 The next lemma is a consequence of Lemma \ref{L.loop}.

\begin{lemma}\label{defTf-}
  Let $ \varpi>0$ be  a small constant, then there is $\delta=\delta(\varpi)$ such that for any
  $0<d \le\delta$ and any $\tau \in \R$ there are $T^f_1(d,\tau)>0$ and
  $\P_f^+(d,\tau)$  such that $\x(t,\tau; \Q_s(d,\tau)) \in K_A^{\textrm{fwd},+}(t)$ for any
  $\tau<t \le \tau+T^f_1(d,\tau)$ and it crosses transversely $S^+(T^f_1(d,\tau)+\tau) \subset \tilde{S}^+$
  at $t=\tau+T^f_1(d,\tau)$ in $\P_f^+(d,\tau)=\x(T^f_1(d,\tau)+\tau,\tau; \Q_s(d,\tau))$.

  Conversely, going backwards in $t$,  there are $T^b_4(d,\tau)>0$ and
  $\P_b^+(d,\tau)$
  such that $\x(t,\tau; \Q^+(d,\tau)) \in K_A^{\textrm{\ev{bwd}},+}(t)$ for any
   $ \tau- T^b_4(d,\tau)< t\le\tau$ and it crosses transversely $L^0$
  at $t=\tau- T^b_4(d,\tau)$ in $\P_b^+(d,\tau)=\x(\tau-T^b_4(d,\tau),\tau; \Q^+(d,\tau))$.
\end{lemma}
 Now we proceed to estimate the functions $T^f_1$, $\P_f^+$, $T^b_4$, $\P^b_4$ defined in Lemma \ref{defTf-}.
We denote by $-\tilde{T}_4^b<0< \tilde{T}_1^f$,  $\tilde{T}_4^b= \tilde{T}_1^f$  the values such that
$\ga(\tilde{T}_1^f) \in \tilde{S}^+$.

\begin{remark}\label{r.close}
Observe that if $\ep=0$ then $T^f_1(0,\tau) \equiv \tilde{T}_1^f$ and $T^b_4(0,\tau) \equiv \tilde{T}_4^b$.
 From    \eqref{est.stableunstable} we find
$$\frac{1}{2\overline{\la}} \ln \left(\frac{1}{c_k\|\ga(\tilde{T}_1^f)\|} \right) \le    \tilde{T}^f_1 \le  \frac{2}{\underline{\la}}
\ln\left( \frac{c_k}{\|\ga(\tilde{T}_1^f)\|}\right).$$
Using $\|\ga(\tilde{T}_1^f)\|=\frac{c}{|\ln(\varpi)|}$ for some $0<c<2$, we derive
\begin{equation}\label{perhaps}
\begin{split}
     &    \frac{1}{2\overline{\la}} \ln \left(\frac{|\ln(\varpi)|}{c_k} \right) \le    \tilde{T}^f_1 \le  \frac{2}{\underline{\la}} \ln(  c_k
     |\ln(\varpi)|) , \\
     &  \frac{1}{2\overline{\la}} \ln \left(\frac{|\ln(\varpi)|}{c_k} \right) \le    \tilde{T}_4^b  \le  \frac{2}{\underline{\la}} \ln(
     c_k|\ln(\varpi)|).
\end{split}
\end{equation}
Hence for any $0<\ep \le \ep_0$ we have
\begin{equation}\label{perhapsbis}
\begin{split}
     &    \frac{1}{2\overline{\la}} \ln \left(\frac{|\ln(\varpi)|}{2c_k} \right) \le    T^f_1(0,\tau) \le  \frac{2}{\underline{\la}} \ln( 2
     c_k
     |\ln(\varpi)|) , \\
     &   \frac{1}{2\overline{\la}} \ln \left(\frac{|\ln(\varpi)|}{2c_k} \right) \le     T^b_4(0,\tau) \le  \frac{2}{\underline{\la}} \ln( 2
     c_k
     |\ln(\varpi)|) .
\end{split}
\end{equation}
\end{remark}

\begin{lemma}\label{close}
  Assume \assump{H}  and let $ \varpi>0$ be  a small constant, then there is $\delta=\delta(\varpi)$ such that for any
  $0<d \le \delta$ and any $\tau \in \R$ we have
   \begin{equation}\label{perhaps3}
    | T^f_1(d,\tau) -T^f_1(0,\tau)| \le 1,
    \end{equation}
 \begin{equation}\label{perhaps4}
    | T^b_4(d,\tau) -T^b_4(0,\tau)| \le 1.
    \end{equation}
\end{lemma}
\begin{proof}

Let us set $\bar{T}= \frac{2}{\underline{\la}} \ln( 4 c_k |\ln(\varpi)|)$, using continuous dependence of \eqref{eq-disc} on initial data (see
Remark \ref{r.smooth}),
we see that there is  $\de=\de(\varpi)$ such that
 \begin{equation}\label{perhaps0}
   \| \x(\theta+\tau, \tau; \Q_s(d,\tau))- \x(\theta+\tau, \tau; \P_s(\tau))\|\le  |\ln(\varpi)|^{-4}
\end{equation}
for any $0 \le \theta \le \bar{T}$ and for any $0 \le d \le \delta$.
 Since $\f(\x(T^f_1(0,\tau)+\tau, \tau; \P_s(\tau)))$ is transversal to $\tilde{S}^+$ and
 $1+T^f_1(0,\tau)+\tau<\bar{T}$, using
 \eqref{perhapsbis}  we find
 \begin{equation}\label{perhaps1}
 	\begin{gathered}
     \dot{\x}(T^f_1(0,\tau)+\tau, \tau; \P_s(\tau)) \\
     =
      \bs{f^+_x}(\vec{0})\x(T^f_1(0,\tau)+\tau, \tau; \P_s(\tau)) + o(\x(T^f_1(0,\tau)+\tau, \tau; \P_s(\tau)))+O(\ep)\\
      =\frac{  \la_s^+ }{|\ln(\varpi)|}\left(   \vec{v}_s^+ +O(|\ln(\varpi)|^{-\alpha}) +O(\ep)   \right).
      \end{gathered}
\end{equation}
Further, a similar estimate holds when $\theta \in [0, 1]$, namely
 \begin{equation}\label{perhaps2}
     \left\|  \dot{\x}(\theta+ T^f_1(0,\tau)+\tau, \tau; \P_s(\tau)) -\frac{ \la_s^+  \eu^{ \la_s^+ \theta}
     }{|\ln(\varpi)|}    \vec{v}_s^+     \right\| =
     \frac{O(\ep)}{|\ln(\varpi)|}+O\left(\frac{1}{|\ln(\varpi)|^{1+\alpha}}\right).
\end{equation}
So, from \eqref{perhaps1} and \eqref{perhaps2}  we find that
$$B(\x(1+ T^f_1(0,\tau)+\tau, \tau; \P_s(\tau)); 2|\ln(\varpi)|^{-2}) \cap \tilde{S}^+ = \emptyset.$$
Then using \eqref{perhaps0} we find
$$B( \x(1+ T^f_1(0,\tau)+\tau, \tau; \Q_s(d,\tau)); |\ln(\varpi)|^{-2}) \cap \tilde{S}^+ = \emptyset,$$
so we get $T^f_1(d,\tau)< T^f_1(0,\tau) +1$. Using a specular argument we prove
\eqref{perhaps3}; the proof of \eqref{perhaps4} is analogous and it is omitted.
\end{proof}

Adapting slightly the proof of Lemmas 6.2 and 6.3 in \cite{FrPo}, which are in fact based on Theorem 12.15 in \cite{HK}
we obtain the following
 estimates of the fly time and displacement with respect to $\tilde{W}^s$, moving forward and backward in time.

\begin{proposition}\label{estloglog}
  Assume \assump{H}  and let $ \varpi>0$ be  a small constant, then there is $\delta=\delta(\varpi)$ such that for any
  $0<d \le \delta$ and any $\tau \in \R$ we have
   \begin{equation}\label{maybe}
    \frac{\ln(|\ln(\varpi)|)}{4\overline{\la}} < T^f_1(d,\tau) < \frac{4\ln(|\ln(\varpi)|)}{\underline{\la}}.
    \end{equation}
  Further
  $\P_f^+(d,\tau)= -D^f_1(d,\tau) \vec{v}_u^+ + \PPs(\tau+T^f_1(d,\tau))$ where
  \begin{equation}\label{maybe0}
    0< d|\ln(\varpi)|^{-\tilde{C}}\leq  D^f_1(d,\tau)\le  d|\ln(\varpi)|^{\tilde{C}}.
  \end{equation}
  Here and below $\tilde{C}>0$ is a constant independent of $\varpi$, $\de$ and $\ep$.

 Respectively,
  $$ \frac{\ln(|\ln(\varpi)|)}{4\overline{\la}} <  T^b_4 (d,\tau)< \frac{4\ln(|\ln(\varpi)|)}{\underline{\la}},$$
 and  $D^b_4(d,\tau)=\DDD(\P_b^+(d,\tau),  \P_s(\tau-T^b_4(d,\tau)))$  verifies
  $$  0< d|\ln(\varpi)|^{-\tilde{C}}\leq    D^b_4(d,\tau) \le d |\ln(\varpi)|^{\tilde{C}}.$$
\end{proposition}
\begin{proof}
 The estimate  \eqref{maybe} follows immediately from \eqref{perhaps3} and  \eqref{perhapsbis}, and similarly for the estimates
concerning $T^b_4(d,\tau)$.

Let us now estimate the displacement
\begin{equation}\label{Df1}
	D^f_1(d,\tau)=D^f_1(d,\tau,\ep)=\frac{\partial D^f_1}{\partial d}(0,\tau,0)d
+o(d+\ep)
\end{equation}
since $D^f_1(0,\tau, \ep)=0$. Now we estimate  $\frac{\partial D^f_1}{\partial d}(0,\tau,0)$, so from now on in this proof we consider
$\ep=0$.
Let us set
\begin{equation*}
	\begin{split}
		M_f= &   \sup  \left\{ \left\{|\operatorname{tr} \bs{f^-_x}(\x)|  \mid \x \in B(\bs{\Gamma^-},1)\right\} \cup
\left\{|\operatorname{tr} \bs{f^+_x}(\x)|  \mid \x \in
		B(\bs{\Gamma^+},1)\right\}\right\},
	\end{split}
\end{equation*}
where $\bs{\Gamma^\pm}=\bs{\Gamma}\cap \Omega^\pm$.
In this proof we denote for the sake of brevity $\vec{p}(\theta,\P)=\x(\theta+\tau, \tau; \P, \ep=0)$. Notice that
$\vec{p}(\theta,\ga(0))=\ga(\theta)$  and
$\dot{\vec{p}}(\theta,\ga(0))=\f^{+}(\ga(\theta))$,
 if $\theta \ge 0$. Adapting Theorem 12.15 in \cite{HK} we consider the variational equation
\begin{equation}\label{linlin}
	\dot{\x}(\theta)= \bs{f^+_x}(\ga(\theta))\x(\theta)\,.
\end{equation}
Following \cite{HK}
we see that if  $\bs{X}(\theta)$ is the fundamental matrix of \eqref{linlin} then for any $T>0$ we get
$$ \det\bs{X}(T)= \exp \left( \int_0^T \textrm{tr} \bs{f^+_x}(\ga(\theta))  d \theta  \right). $$
Further, recalling that $\bs{f^+_x}(\ga(\theta))$ is close to  $\bs{f^+_x}(\vec{0})$
when $\theta \to +\infty$ we find
\begin{equation}\label{newmaybe1}
	\|\bs{X} (\tilde{T}^f_1)\| \le  \eu^{2\ov{\la}\tilde{T}^f_1} \le |\ln(\varpi)|^{8\ov{\la}/\und{\la}}.
\end{equation}
Next from \eqref{maybe} we get
\begin{equation}\label{maybe1}
	|\ln(\varpi)|^{-4M_f/\underline{\la}} \le \det(\bs{X} (\tilde{T}^f_1)) \le  |\ln(\varpi)|^{4M_f/\underline{\la}}
\end{equation}
for any $0 \le d \le \delta$.
Let now $\vec{v}$ be the
unit  vector tangent to $\Om^0$ in $\ga(0)$ aiming towards $E^{\inn}$: notice that
$\vec{v}$ is  transversal to $\dot{\ga}(0^{+})= \f^{+}(\ga(0^{+}))$, cf. \assump{K}.

Observe that by construction  $\Q_s(0,\tau)=\P_s(\tau)=\ga(0)$, $\frac{\partial \Q_s(0,\tau)}{\partial d}= \vec{v}$;
further, $\bs{\frac{\partial}{\partial \P}}\vec{p}(\theta,\ga(0))=\bs{X}(\theta)$ is
the fundamental matrix of \eqref{linlin}.

Now we proceed to evaluate $T'_d(\tau):=\tfrac{\partial}{\partial d} T_1^f(d,\tau)\lfloor_{d=0}$.
Note that, by construction,
$ [ \bs{P^+}\vec{p}(T_1^f(d,\tau),\Q_s(d,\tau))]^* \vec{v}_s^+  \equiv |\ln(\varpi)|^{-1}$ for any
$0<d \le \delta$ and any $\tau \in
\R$.
Hence differentiating we find
\begin{equation}\label{Tdiff}
	\begin{split}
		T'_d(\tau)  & = -
		\frac{ \left[\bs{P^+}\bs{\frac{\partial }{\partial \P}} \vec{p}(T_1^f(0,\tau), \P_s(\tau))  \frac{\partial
				\Q_s(0,\tau)}{\partial d}\right]^*\vec{v}_s^+}{
			\left[ \bs{P^+}\dot{\vec{p}}(T_1^f(0,\tau),\P_s(\tau)) \right]^*\vec{v}_s^+}\\
		& =  -\frac{  [ \bs{P^+}\bs{X}(\tilde{T}_1^f) \vec{v})]^*\vec{v}_s^+}{ [
\bs{P^+}\f^+(\ga(\tilde{T}_1^f))]^*\vec{v}_s^+  }.
	\end{split}
\end{equation}
Now, observing
that
$$\vec{p}( T^f_1(0,\tau),\P_s(\tau))=\ga(\tilde{T}^f_1) =  \frac{\vec{v}_s^+}{|\ln(\varpi)|} +
o\left(\frac{1}{|\ln(\varpi)|}\right)$$
we get
\begin{equation*}
	\begin{split}
		\f^+(\ga(\tilde{T}^f_1))= & \bs{f^+_x}(\vec{0})\ga(\tilde{T}_1^f)+o( \ga(\tilde{T}_1^f))
		=  \frac{ \la_s^+}{|\ln(\varpi)|} \vec{v}_s^+  + o\left(\frac{1}{|\ln(\varpi)|}\right)
	\end{split}
\end{equation*}
whence \begin{equation}\label{whence}
\begin{split}
&    [ \bs{P^+}\f^+(\ga(\tilde{T}_1^f))]^*\vec{v}_s^+ =  \frac{ \la_s^+}{|\ln(\varpi)|}
+o\left(\frac{1}{|\ln(\varpi)|} \right)  \qquad
\textrm{and}  \\
	& \frac{  \underline{\la}}{2|\ln(\varpi)|}\le \|\f^+(\ga(\tilde{T}^f_1))\| \le \frac{2 \overline{\la}}{|\ln(\varpi)|}.
	\end{split}
\end{equation}
 Recall that    $\max\{\| \bs{P^+}\| ; \| \I- \bs{P^-}\|\} \le k_2$;
then, using \eqref{newmaybe1} and \eqref{whence} in \eqref{Tdiff},
  we find
\begin{equation}\label{Tdiff1}
	|T'_d(\tau)| \le  \frac{ k_2}{\und{\la}} |\ln(\varpi)|^{8 \tfrac{\ov{\la}}{\und{\la}}+1}.
\end{equation}

To proceed in the   computation, it is convenient (cf.~p.~377--388 in \cite{HK}) to consider the matrix representation
$\bs{Y}(\tilde{T}^f_1)$ of
$\bs{X}(\tilde{T}^f_1)$ with respect to the bases $\big\{\frac{\f^+(\ga(0))}{\|\f^+(\ga(0))\|},\vec{v} \big\}$ in the domain and
$\big\{\frac{\f^+(\ga(\tilde{T}^f_1))}{\|\f^+(\ga(\tilde{T}^f_1))\|},\vec{v}_u^+ \big\}$ in the codomain,
i.e., to set $\bs{Y}(\tilde{T}^f_1)=\bs{B} \bs{X}(\tilde{T}^f_1) \bs{A}$ where
$$\bs{A}:= \left(
	\frac{\f^+(\ga(0))}{\|\f^+(\ga(0))\|}, \vec{v} \right)\, , \qquad  \bs{B}:=
\left(
	\frac{\f^+(\ga(\tilde{T}^f_1))}{\|\f^+(\ga(\tilde{T}^f_1))\|}, \vec{v}_u^+
\right)^{-1}. $$
Since $\dot{\vec{p}}(\theta,\ga(0))$ is a solution of the variational equation \eqref{linlin} we get
$$\f^+(\ga(\tilde{T}^f_1))=\dot{\vec{p}}(\tilde{T}^f_1,\ga(0))=\bs{X}(\tilde{T}^f_1) \dot{\vec{p}}(0,\ga(0))=
\bs{X}(\tilde{T}^f_1)\f^+(\ga(0))$$
or equivalently
$$  \left( \frac{\|\f^+(\ga(\tilde{T}^f_1))\|}{\|\f^+(\ga(0))\|},0 \right)^*= \bs{Y}(\tilde{T}^f_1) \, \cdot \, (1,0)^*. $$
We recall that
$$-D^f_1(d,\tau)  \vec{v}_u^+ =  \vec{p}( T^f_1(d,\tau),\Q_s(d,\tau))- \PPs(T^f_1(d,\tau)+\tau).  $$

Notice that $\PPs(\tau)\equiv\ga(\tilde{T}^f_1) \in \tilde{S}^+$ for any $\tau$ when $\ep=0$,
 so $\frac{d \PPs}{dt}(\tau) \equiv \vec{0}$ if $\ep=0$.

Now,
recalling that
$$ \dot{\vec{p}}(T^f_1(d,\tau),\Q_s(d,\tau))= \f^+(\vec{p}(T^f_1(d,\tau),\Q_s(d,\tau)))$$
we find
\begin{gather*}
	-\frac{\partial D^f_1}{\partial d}(0,\tau) \vec{v}_u^+=
\bs{\frac{\partial }{\partial \P}} \vec{p}(T_1^f(0,\tau), \P_s(\tau))  \frac{\partial
	\Q_s(0,\tau)}{\partial d}\\
{}+\dot{\vec{p}}(T_1^f(0,\tau),\P_s(\tau))
\frac{\partial T_1^f(0,\tau)}{\partial d}
=\bs{X}(\tilde{T}^f_1)\vec{v} +\f^+(\ga(\tilde{T}^f_1))T'_d(\tau).
\end{gather*}
Thus, $\bs{Y}(\tilde{T}^f_1)$ is the  matrix
\begin{equation}\label{Ytf}
\bs{Y}(T^f_1(0,\tau)) = \begin{pmatrix}
	\frac{\|\f^+(\ga(\tilde{T}^f_1))\|}{\|\f^+(\ga(0))\|} &  - \|\f^+(\ga(\tilde{T}^f_1))\| T'_d(\tau)   \\
	0 &   -\frac{\partial D^f_1}{\partial d}(0,\tau)
\end{pmatrix}.
\end{equation}
Now let $K_A$ and $K_B$ equal the determinants respectively of $\bs{A}$ and $\bs{B}$: since their columns are made up
by unit vectors which are transversal, then  $|K_A|>0$ and $|K_B|>0$.  Hence
\begin{equation}\label{maybe2}
	\det \bs{X}(\tilde{T}^f_1)=\frac{\det \bs{Y}(\tilde{T}^f_1)}{K_A K_B}=
	-\frac{\partial D^f_1}{\partial d}(0,\tau)
		\frac{\|\f^+(\ga(\tilde{T}^f_1))\|}{\|\f^+(\ga(0))\| K_A K_B}.
\end{equation}
So from  \eqref{maybe1}  and \eqref{whence}   we find $0<c_1<c_2$ such that
$$ c_1 |\ln(\varpi)|^{-\frac{4M_f }{\underline{\la}} +1} \le \left|\frac{\partial D^f_1}{\partial d}(0,\tau)\right| \le c_2
|\ln(\varpi)|^{\frac{4 M_f}{\underline{\la}} +1}.$$
So, taking into account \eqref{Df1} and setting $\tilde{C}=
\tfrac{4 M_f}{\underline{\la}} +2$
we conclude the proof of \eqref{maybe0}.

Notice that
\eqref{maybe} is independent of $\delta$ and $\tau$ so the   estimate of
$T^b_4(d,\tau)$ is obtained from \eqref{maybe}.
Then the proof concerning $D^b_4(d,\tau)$ is obtained from \eqref{maybe0} reversing the role of $d$ and $D^f_1$.
\end{proof}
We emphasize that we have proved also the following result, cf. \eqref{Tdiff1} which allows us to improve \eqref{perhaps3}.
\begin{remark}\label{r.Tdiff}
  Let the assumptions of Proposition \ref{estloglog} be satisfied. Then
  $$|T'_d(\tau)|:= \left| \frac{\partial}{\partial d} T^f_1(0,\tau)\right|\le  \frac{  k_2  }{\und{\la}} |\ln(\varpi)|^{8
  \tfrac{\ov{\la}}{\und{\la}}+1};$$
 so,    for any $0<d \le \delta \le \varpi$,  we have
  $$|T^1_f(d,\tau)- T^f_1(0,\tau)| \le  \frac{  2 k_2}{\und{\la}} |\ln(\delta)|^{8 \tfrac{\ov{\la}}{\und{\la}}+1} d  \le    d
  \delta^{-\mu_1/2}.$$
\end{remark}

Analogously let us consider the trajectory $\x(t,\tau; \Q_u(d,\tau))$ backward from $L^0$ to $\tilde{S}^-$;
using again Lemma \ref{L.loop} we obtain the following.

 \begin{lemma}\label{defTb+}
   Let $ \varpi>0$ be  a small constant, then there is $\delta=\delta(\varpi)$ such that for any
  $0<d \le\delta$ and any $\tau \in \R$ there are  $T^b_1(d,\tau)>0$ and
  $\P_b^-(d,\tau)$  such that $\x(t,\tau; \Q_u(d,\tau)) \in K_A^{\textrm{bwd},-}(t)$ for any
  $\tau- T^b_1(d,\tau)\le t <\tau$ and it crosses transversely $S^-( \tau-T^b_1(d,\tau)) \subset \tilde{S}^-$
  at $t=\tau-T^b_1(d,\tau)$ in $\P_b^-(d,\tau)=\x(\tau-T^b_1(d,\tau),\tau; \Q_u(d,\tau))$.

  Conversely, going forward in $t$,
   there are $T^f_4(d,\tau)>0$ and
  $\P_f^-(d,\tau)$  such that $\x(t,\tau; \Q^-(d,\tau)) \in K_A^{\michal{\textrm{fwd}},-}(t)$ for any
  $\tau \leq t < \tau+ T^f_4(d,\tau)$ and it crosses transversely $L^0$
  at $t=\tau+T^f_4(d,\tau)$ in $\P_f^-(d,\tau)=\x(\tau+T^f_4(d,\tau),\tau;  \Q^- (d,\tau))$.
\end{lemma}
Then, arguing as in Proposition  \ref{estloglog} we find
\begin{proposition}\label{estloglog+}
  Assume \assump{H}  and let $ \varpi>0$ be  a small constant. Then there is $\delta=\delta(\varpi)$ such that for any
  $0<d \le\delta$ we have
  $$ \frac{\ln(|\ln(\varpi)|)}{4\overline{\la}} < T^b_1(d,\tau)< \frac{4\ln(|\ln(\varpi)|)}{\underline{\la}},$$
  for any $\tau \in\R$.
  Further
  $\P_b^-(d,\tau)= -D^b_1(d,\tau) \vec{v}_s^- + \PPu(\tau-T^b_1(d,\tau))$ where
  $$  0<d |\ln(\varpi)|^{-\tilde{C}}\leq  D^b_1(d,\tau) \le d |\ln(\varpi)|^{\tilde{C}},$$
  and $\tilde{C}>0$ is as in Proposition \ref{estloglog}.
  Respectively,
   $$ \frac{\ln(|\ln(\varpi)|)}{4\overline{\la}} < T^f_4(d,\tau) < \frac{4\ln(|\ln(\varpi)|)}{\underline{\la}},$$
  and
  $D^f_4(d,\tau)= \DDD(\P_f^-(d,\tau),  \P_u(\tau+ T^f_4(d,\tau))$ verifies
  $$  0< d |\ln(\varpi)|^{-\tilde{C}}\leq    D^f_4(d,\tau) \le d|\ln(\varpi)|^{\tilde{C}}.$$
\end{proposition}

The proof is analogous to the one of Proposition  \ref{estloglog} and it is omitted.

 From an analysis of the proof of Proposition \ref{estloglog} we get the following.

\begin{lemma}\label{0keymissed}
Assume \assump{H} and  fix $0<\mu_1<1/2$ and $\tau \in \R$, then there are $\varpi=\varpi(\mu_1)$ and $\de=\delta(\mu_1,\varpi)$  such that
$$ \| \x(\theta+\tau, \tau; \Q_s(d,\tau))- \x(\theta+\tau, \tau; \P_s(\tau))\|\le  d\de^{-\mu_1},$$
$$ \| \x(\phi+\tau, \tau; \Q^+(d,\tau))- \x(\phi+\tau, \tau; \vec{\pi}_s(\tau))\|\le  d\de^{-\mu_1}$$
for any $- T^b_4(d,\tau) \le \phi \le 0\le \theta \le T^f_1(d,\tau)$, and any $0\le d \le \delta$; further
$$ \| \x(\phi+\tau, \tau; \Q_u(d,\tau))- \x(\phi+\tau, \tau; \P_u(\tau))\|\le d \de^{-\mu_1},$$
$$ \| \x(\theta+\tau, \tau; \Q^-(d,\tau))- \x(\theta+\tau, \tau; \vec{\pi}_u(\tau))\|\le d \de^{-\mu_1}$$
for any $-T^b_1(d,\tau)\le \phi \le 0 \le \theta \le T^f_4(d,\tau)$, and any $0\le d \le \delta$.

\end{lemma}

\subsection{The loop: from $\boldsymbol{S^+(\tau)}$ to $\boldsymbol{L^\textrm{in}}$}\label{S-Lin}
In this section we consider the fly time and the displacement of trajectories traveling from
$S^+(\tau)$ to $L^{\inn}$,  see Figure \ref{fig.key}.  In this case we linearize
centering in the origin
and we decompose the solution as a sum  of three terms: a
``stable'' nonlinear component,
i.e., a trajectory of the stable manifold,
 a linear  ``unstable'' component, i.e., a trajectory of the linearization \eqref{eq-lin}  of \eqref{eq-disc} in the origin, and a remainder.

Let us recall  that $\vec{Q}^{\pm}(d,\tau) \in S^{\pm}(\tau)$ are defined in \eqref{Q-Q+},  see also Figure \ref{fig.key}.

In the whole section we assume that $0<\delta< \varpi$ are suitably small parameters, see Remark \ref{defvarpi} (in fact for the results in
this section we do not need to  find a function $\de= \de(\varpi)$ small enough: we just need to assume $\de<\varpi$).

\begin{lemma}\label{L.0}
 Assume \textbf{H} and fix $\tau \in \R$; for any $0<d \le \delta$ there are $C^{r}$ functions $T^f_2(d,\tau)$  and $D^f_2(d,\tau)$ such that
  $\x (\theta+\tau,\tau;\Q^+(d,\tau))\in K_A^{\textrm{fwd},+}(\theta+\tau)$ for any $0 \leq \theta<T^f_2(d,\tau)$ and it crosses
  transversely
   $\Om^0$    at $\theta=T^f_2(d,\tau)$ in a point
  such that
  $\| \x(T^f_2(d,\tau)+\tau,\tau;\Q^+(d,\tau))\|=D^f_2(d,\tau)$.
\end{lemma}

\begin{proof}
   It follows from the argument of Lemma \ref{L.loop}, taking into account the  mutual positions of
  $\vec{v}_u^\pm$, $\vec{v}_s^\pm$.
\end{proof}

\begin{remark}\label{RemarkDT-injective}
   Note that the mapping $\Psi^+:(d,\tau)\mapsto (D^f_2(d,\tau),T^f_2(d,\tau)+\tau)$
 is injective on $]0,\delta]\times \R$ due to the reversibility of the flow.
\end{remark}

Let us recall that $\y_s(\theta)=\x(\theta+\tau,\tau; \PPs(\tau))$;
following \cite[\S 4.2]{CDFP} we can estimate $\y_s$ as indicated in the next remark.

\begin{remark}\label{est.ys}
  There   are $0<\bar{c}_s<c_s/2$, $\tilde{a}^+_u(\theta)$ and $\tilde{a}^+_s(\theta)$ such that
\begin{equation}\label{def:ys}
  \y_s(\theta)= \tilde{a}^+_s(\theta)  \frac{z_s^+(\theta+\tau,\tau)}{|\ln(\varpi)|} \,\vec{v}_s^+(\theta+\tau)
   + \tilde{a}^+_u(\theta) \left(\frac{z_s^+(\theta+\tau,\tau)}{|\ln(\varpi)|} \right)^{1+\alpha} \,\vec{v}_u^+(\theta+\tau)
\end{equation}
and $\bar{c}_s<\tilde{a}^+_s(\theta)<c_s/2$, $|\tilde{a}^+_u(\theta)|<c_s/2$ for any $\theta>0$.
\end{remark}
\noindent
Since $\ep>0$ is fixed in this section we leave this dependence unsaid, so we write $\bs{X^+}(\theta)$
for the fundamental matrix of \eqref{eq-lin}$_+$ and $\vec{F}^+(\x,\tau)$ instead of $\vec{F}^+(\x,\tau,\ep)$.
Denote by
\begin{equation}\label{def:ell-}
\vec{\ell}^+(\theta):=- \bs{X^+}(\theta+\tau) (\bs{X^+}(\tau))^{-1}   d\vec{v}_u^+(\tau)=
  -d z^+_u(\theta+\tau,\tau) \vec{v}_u^+(\theta+\tau).
  \end{equation}
Notice that by definition  $\dot{\vec{\ell}}^+(\theta)= \boldsymbol{F_x^+}(\vec{0},\theta+\tau)\vec{\ell}^+(\theta)$,
$\vec{\ell}^+(0)=-d\vec{v}_u^+(\tau)$.

Now, using \eqref{def:ys} and \eqref{def:ell-}, and taking into account Remark \ref{cumbersome} we   expand
  $\x(\theta+\tau,\tau;\Q^+(d,\tau))$ as
 \begin{equation}\label{exp.h-}
    \x(\theta+\tau,\tau;\Q^+(d,\tau))=\y_s(\theta)+\vec{\ell}^+(\theta)+\vec{h}^+(\theta),
 \end{equation}
where  $\theta \ge0$  and $\vec{h}^+(\theta)$ is a remainder.

We want to write a fixed-point equation for the remainder, in a suitable exponentially weighted space.
Notice first that $\vec{h}^+(\theta)$ satisfies
$$\vec{h}^+(0)=\Q^+(d,\tau)-\PPs(\tau)+d \vec{v}_u^+(\tau)=\vec{0}$$
and the equation
\begin{gather*}
 \dot{\h}^+(\theta)=\dot\x(\theta+\tau,\tau;\Q^+(d,\tau))-\dot\y_s(\theta)- \dot{\vec{\ell}}^+(\theta)\\
=\vec{F}^+(\y_s(\theta)+\vec{\ell}^+(\theta)+\h^+(\theta),\theta+\tau) -\vec{F}^+(\y_s(\theta),\theta+\tau)-
\boldsymbol{F_x^+}(\vec{0},\theta+\tau)\vec{\ell}^+(\theta) \\
=\boldsymbol{F_x^+}(\vec{0},\theta+\tau)\vec{h}^+(\theta)+ \vec{R}^+_1(\theta, \vec{h}^+(\theta))+  \vec{R}^+_2(\theta, \vec{h}^+(\theta))
\end{gather*}
where the remainder terms $\vec{R}^+_1(\theta, \vec{h}^+(\theta))$ and $\vec{R}^+_2(\theta, \vec{h}^+(\theta))$ equal
\begin{gather*}
\vec{R}^+_1(\theta, \vec{h}):= \left[\boldsymbol{F_x^+}(\y_s(\theta),\theta+\tau)-\boldsymbol{F_x^+}(\vec{0},\theta+\tau)
\right](\vec{\ell}^+(\theta)+\vec{h}),
\\
\vec{R}^+_2(\theta,  \vec{h})
	:= \vec{F}^+(\y_s(\theta)+\vec{\ell}^+(\theta)+\vec{h},\theta+\tau) -\vec{F}^+(\y_s(\theta),\theta+\tau)\\
	{}-\boldsymbol{F_x^+}(\y_s(\theta),\theta+\tau) (\vec{\ell}^+(\theta)+\vec{h}).
\end{gather*}

 Let us set
\begin{equation}\label{D0M-}
  D_0:= \frac{1}{|\ln(\delta)|^2}, \qquad M_0:=  \frac{2|\ln(\delta)|}{ \underline{\la}}.
\end{equation}

Let now $M \ge M_0$ be fixed, and  $X^+=C([0,M],\R^2)$ be the Banach space of continuous functions equipped with the norm
{\begin{equation*}
	\|\vec{u}\|_{X^+}:=\max_{\theta\in[0,M]}\frac{\|\vec{u}(\theta)\|}{z_u^+(\theta+\tau,M+\tau)}.
\end{equation*}}

Fix $0<D_{\ell} \le D_0$ and set $ d= D_{\ell}/ z_u^+(M+\tau,\tau)$, so that from \eqref{defk} we find
  $$ d \le \frac{D_0}{z_u^+(M_0+\tau,\tau)} \le  \frac{k_1\eu^{-\frac12\underline{\la}M_0}}{  | \ln(\delta)|^2}=  \frac{k_1}{  |
  \ln(\delta)|^2} \delta \leq \delta$$
since we can assume $\delta   \le \eu^{-\sqrt{k_1}}$, see Remark \ref{oneparam}.

From \eqref{def:ys} and \eqref{def:ell-}, we see that
\begin{equation}\label{norms-}
	\begin{gathered}
		\|\y_s(\theta)\| \le \frac{c_s}{|\ln(\varpi)|}\,z_s^+(\theta+\tau,\tau) \le \frac{c_s k_1}{|\ln(\varpi)|}
\eu^{-\frac12\underline{\la}\theta} , \\
		\|\vec{\ell}^+(\theta)\|=dz_u^+(\theta+\tau,\tau)
		=dz_u^+(M+\tau,\tau)z_u^+(\theta+\tau,M+\tau)\\
=D_{\ell} z_u^+(\theta+\tau,M+\tau)  \le \frac{k_1}{|\ln(\delta)|^2} \eu^{-\frac12\underline{\la}(M-\theta)}  .
	\end{gathered}
\end{equation}

 Moreover, assumption \assump{H} implies
\begin{equation}\label{est:ball}
\begin{split}
\|\vec{R}_1^+(\theta, \vec{h})\| &
	  \le    N_{\alpha} \|\y_s(\theta)\|^{\al} \cdot \|\vec{\ell}^+(\theta)+\vec{h} \|   , \\
\|\vec{R}_2^+(\theta, \vec{h})  \| & = \left\|\int_{0}^{1} \left[\boldsymbol{F_x^+}(\y_s(\theta)+\sigma [\vec{\ell}^+(\theta)+\vec{h}]
,\theta+\tau) \right.\right.\\
&\left.\left. {}-\boldsymbol{F_x^+}(\y_s(\theta)  ,\theta+\tau)
\right](\vec{\ell}^+(\theta)+\vec{h})d \sigma \right\| \le    N_{\alpha} \|\vec{\ell}^+(\theta)+\vec{h} \|^{1+ \al}   .
\end{split}
 \end{equation}

On the other hand,
\begin{gather*}
 \vec{R}^+_1(\theta,  \vec{h}_2)-
\vec{R}^+_1(\theta,  \vec{h}_1) \\
=    \left[\boldsymbol{F_x^+}(\y_s(\theta),\theta+\tau)-\boldsymbol{F_x^+}(\vec{0},\theta+\tau) \right](\h_2-\h_1).
\end{gather*}
Further
\begin{gather*}
\vec{R}^+_2(\theta,  \vec{h}_2)-
\vec{R}^+_2(\theta, \vec{h}_1)=
 \vec{F}^+(\y_s(\theta)+\vec{\ell}^+(\theta)+ \vec{h}_2,\theta+\tau) \\
 {}- \vec{F}^+(\y_s(\theta)+\vec{\ell}^+(\theta)+ \vec{h}_1,\theta+\tau)
- \boldsymbol{F_x^+}(\y_s(\theta),\theta+\tau)(\h_2-\vec{h}_1)
\\
=
\int_0^1 \left[\boldsymbol{F_x^+}\Big(\y_s(\theta)+\vec{\ell}^+(\theta)+ \vec{h}_1 +\sigma(\vec{h}_2-\vec{h}_1),\theta+\tau\Big)
-\boldsymbol{F_x^+}(\y_s(\theta),\theta +\tau)\right](\vec{h}_2-\vec{h}_1) d\sigma
\end{gather*}
Thus,
\begin{equation}\label{NN}
\begin{split}
&   \left\|\vec{R}^+_1(\theta,  \vec{h}_2)-
\vec{R}^+_1(\theta,  \vec{h}_1) \right\|     \le  N_{\al} \|\y_s(\theta)\|^{\al} \, \|\vec{h}_2-\vec{h}_1\|,
 \\
 &  \left\|\vec{R}^+_2(\theta,  \vec{h}_2)-
\vec{R}^+_2(\theta,  \vec{h}_1) \right\|       \\
& \le N_{\al} \left(\|\vec{\ell}^+(\theta)\|+\max\{ \|\vec{h}_1\|,\|\vec{h}_2\|\} \right)^{\al} \|\vec{h}_2-\vec{h}_1\|.
\end{split}
\end{equation}

Let us define the operator $\FF^+:X^+\to X^+$ as
\begin{gather*}
	\FF^+(\vec{u})(\theta)=
	\int_0^{\theta}\bs{X^+}(\theta+\tau)(\bs{X^+}(s+\tau))^{-1}
	\left(\vec{R}^+_1(s,\vec{u}(s))+  \vec{R}^+_2(s,\vec{u}(s)) \right) \,ds\\
	=\int_0^{\theta}\bs{X^+}(\theta+\tau)\boldsymbol{P^{+}}(\bs{X^+}(s+\tau))^{-1}
	\left(\vec{R}^+_1(s,\vec{u}(s))+  \vec{R}^+_2(s,\vec{u}(s)) \right) \,ds\\
	+\int_0^{\theta}\bs{X^+}(\theta+\tau)(\I-\boldsymbol{P^{+}})(\bs{X^+}(s+\tau))^{-1}
	\left(\vec{R}^+_1(s,\vec{u}(s))+  \vec{R}^+_2(s,\vec{u}(s)) \right) \,ds.
\end{gather*}
 Notice that $\h^+(\theta)$ is a fixed point of $\FF^+$ if and only if the function $\x(\theta+\tau,\tau;\Q^+(d,\tau))$ given
 by \eqref{exp.h-} solves \eqref{eq-disc} and  $\h^+(0)=\vec 0$.

\begin{remark}\label{defvarpi}
   In the next Lemma \ref{L-S-Lin-} we let
\begin{equation}\label{est.varpi1}
\begin{split}
& \ln(\varpi) \le -   \max \{(2C)^{2/\al}; 4\}, \qquad \qquad \textrm{where}\\
&  C=      (2c_s^{\al}+1) N_\al\frac{2k_2 k_1^{2+2\al}(\al+1)}{\underline{\la}\al}.
\end{split}
\end{equation}
Observe that this  is not  an optimal constant: however, the relevant fact is that both $C$ and $\varpi$ are independent of $d$, $\de$
and
$\ep$.
\end{remark}

 \begin{lemma}\label{L-S-Lin-}
 	Assume \textbf{H},  fix $\varpi>0$ as in \eqref{est.varpi1}
	let $\de=\de(\varpi)< \varpi$ be as in
 Lemma \ref{defTf-} and Remark \ref{defTb+}. With the above notation,
 	for any $0<d \le \delta$,
 	the fixed point equation
 	$\FF^+(\vec{u})=\vec{u}$
 	admits a unique solution $\vec{h}$ in $\overline{B}_r^{X^+}:= \{\vec{v}\in X^+\mid \|\vec{v}\|_{X^+}\leq r\}$ where $r= D_{\ell}
 |\ln(\varpi)|^{-\al/2}$.
 \end{lemma}
 \begin{proof}
 	First we show that $\FF^+$  maps  $\overline{B}_r^{X^+}$ into itself.
 	For this purpose, notice that if $\vec{u}\in \overline{B}_r^{X^+}$ then
 	$\|\vec{u}(\theta)\|\leq rz_u^+(\theta+\tau,M+\tau)$ for any $\theta\in [0,M]$.
 	Now, using
  $$\|\vec{\ell}^+(\theta)+\vec{u}(\theta) \|   \le (D_{\ell}+r) z_u^+(\theta+\tau,M+\tau)   \leq 2r|\ln(\varpi)|^{\al/2}
 z_u^+(\theta+\tau,M+\tau), $$
  from
 	\eqref{norms-} and \eqref{est:ball} we get
 	\begin{gather*}
 		\|\vec{R}_1^+(\theta, \vec{u}(\theta))\|
 		\le           2N_{\al} (c_s)^{\al}    r|\ln(\varpi)|^{-\al/2} [z_s^+(\theta+\tau,\tau)]^{\al}\,
 			z_u^+(\theta+\tau,M+\tau),\\
 		\|\vec{R}_2^+(\theta, \vec{u}(\theta))\|
 		\le  N_{\al} \big[ 2r |\ln(\varpi)|^{\al/2}  z_u^+(\theta+\tau,M+\tau) \big]^{1+\al}.
 	\end{gather*}
 	Moreover, from \eqref{defk} we get the following estimates:
 	\begin{gather*}
 		\int_0^\theta z_s^+(\theta+\tau,s+\tau)  [z_s^+(s+\tau,\tau)]^{\al} z_u^+(s+\tau,M+\tau)ds\\
 		 \le k_1^{1+\al}   z_u^+(\theta+\tau,M+\tau) \int_0^\theta    z_u^+(s+\tau,\theta+\tau)  ds \\
 		\leq   \frac{2(k_1)^{2+\al} }{\underline{\la}}\,z_u^+(\theta+\tau,M+\tau),
 	\end{gather*}

 \begin{gather*}
 		\int_0^\theta z_s^+(\theta+\tau,s+\tau) \big[z_u^+(s+\tau,M+\tau)\big]^{1+\al}ds\\
 		=\big[z_u^+(\theta+\tau,M+\tau)\big]^{1+\al} \int_0^\theta  z_s^+(\theta+\tau,s+\tau) \big[z_u^+(s+\tau,\theta+\tau)\big]^{1+\al}ds\,
 		\\ \le k_1^{2+2\al} z_u^+(\theta+\tau,M+\tau) \int_0^\theta \eu^{-\frac12\underline{\la}(2+\al)(\theta-s) }ds
 		\leq   \frac{2k_1^{2+2\al}}{\underline{\la}(2+\al)} z_u^+(\theta+\tau,M+\tau),
 	\end{gather*}

\begin{gather*}
 		\int_0^\theta z_u^+(\theta+\tau,s+\tau)\big[z_s^+(s+\tau,\tau)\big]^{\al} z_u^+(s+\tau,M+\tau)ds\\
 		=z_u^+(\theta+\tau,M+\tau)\int_0^\theta \big[z_s^+(s+\tau,\tau)\big]^{\al}ds\, \\
		\leq \frac{2(k_1)^{\alpha}}{\underline{\la} \alpha}\,z_u^+(\theta+\tau,M+\tau),
 	\end{gather*}
 	and
 	\begin{gather*}
 		\int_0^\theta z_u^+(\theta+\tau,s+\tau)\big[z_u^+(s+\tau,M+\tau)\big]^{1+\al} ds\\
 		=  z_u^+(\theta+\tau,M+\tau) \int_0^\theta \big[z_u^+(s+\tau,M+\tau)\big]^{\al}ds \,\\
 		\leq \frac{2(k_1)^{\al}}{ \underline{\la} \al }\,z_u^+(\theta+\tau,M+\tau).
 	\end{gather*}

 	Using the above, we get
 	\begin{gather*}
 		\left\|\int_0^\theta \bs{X^+}(\theta+\tau)\boldsymbol{P^{+}}(\bs{X^+}(s+\tau))^{-1}
 		\vec{R}^+_1(s,\vec{u}(s))ds\right\|\\
 		\leq  2N_{\al} (c_s)^{\al} k_2 r |\ln(\varpi)|^{-\al/2}
 		\int_0^\theta z_s^+(\theta+\tau,s+\tau)[z_s^+(s+\tau,\tau)]^{\al} z_u^+(s+\tau,M+\tau)ds\\
 		\leq  2N_{\al} (c_s)^{\al} k_2     \frac{2k_1^{2+\al}}{\underline{\la}} \,  r |\ln(\varpi)|^{-\al/2} \,z_u^+(\theta+\tau,M+\tau),
 	\end{gather*}

	\begin{gather*}
 		\left\|\int_0^\theta \bs{X^+}(\theta+\tau)\boldsymbol{P^+}(\bs{X^+}(s+\tau))^{-1}
 		\vec{R}^+_2(s,\vec{u}(s))ds\right\|\\
 		\leq  N_{\al}  k_2 \big[2r |\ln(\varpi)|^{\al/2}\big]^{1+\alpha}
 		\int_0^\theta z_s^+(\theta+\tau,s+\tau)[z_u^+(s+\tau,M+\tau)]^{1+\alpha}ds\\
 		\leq N_{\al}  k_2   \frac{2k_1^{2+2\al}}{\underline{\la}(2+\al)}  \,  \big[2r |\ln(\varpi)|^{\al/2}\big]^{1+\alpha}  \,
 \,z_u^+(\theta+\tau,M+\tau),
 	\end{gather*}

 	\begin{gather*}
 		\left\|\int_0^\theta \bs{X^+}(\theta+\tau)(\I-\boldsymbol{P^+})(\bs{X^+}(s+\tau))^{-1}
 		\vec{R}^+_1(s,\vec{u}(s))ds\right\|\\
 		\leq  2N_{\al} (c_s)^{\al} k_2 r   |\ln(\varpi)|^{-\al/2}
 		\int_0^\theta z_u^+(\theta+\tau,s+\tau) \big[ z_s^+(s+\tau,\tau) \big]^{\al}  z_u^+(s+\tau,M+\tau)ds\\
 		\leq 2N_{\al} (c_s)^{\al} k_2    \frac{2k_1^{\al}}{\underline{\la} \al} \,  r   |\ln(\varpi)|^{-\al/2} \,z_u^+(\theta+\tau,M+\tau),
 	\end{gather*}

 and

 	\begin{gather*}
 		\left\|\int_0^\theta \bs{X^+}(\theta+\tau)(\I-\boldsymbol{P^+})(\bs{X^+}(s+\tau))^{-1}
 		\vec{R}^+_2(s,\vec{u}(s))ds\right\|\\
 		\leq  N_{\al}  k_2 \big[2r |\ln(\varpi)|^{\al/2}\big]^{1+\alpha}
 		\int_0^\theta z_u^+(\theta+\tau,s+\tau) \big[ z_u^+(s+\tau,M+\tau) \big]^{1+\al}  ds\\
 		\leq  N_{\al}  k_2  \frac{2k_1^{\al}}{\underline{\la} \al}    \,   \big[2r
 |\ln(\varpi)|^{\al/2}\big]^{1+\alpha}\,z_u^+(\theta+\tau,M+\tau).
 	\end{gather*}

	Since $ r |\ln(\varpi)|^{\al/2}=   D_{\ell}\leq D_0=|\ln\de|^{-2}< |\ln(\varpi)|^{-2}$, we have
\begin{equation}\label{rde}
\begin{split}
\big[2 r |\ln(\varpi)|^{\al/2}\big]^{1+\alpha} & = r |\ln(\varpi)|^{-\al/2} 2^{1+\al}[r|\ln(\varpi)|^{\al/2}\cdot |\ln(\varpi)|]^{\al} \,
 \\ & \le r |\ln(\varpi)|^{-\al/2}  2^{1+\al} |\ln(\varpi)|^{-\al} \ll  r |\ln(\varpi)|^{-\al/2};
\end{split}
\end{equation}
further
 \begin{equation}\label{rde2}
   2r |\ln(\varpi)|^{\alpha/2}=2D_{\ell} \le 2|\ln(\varpi)|^{-2} \le |\ln(\varpi)|^{-1/2} ,
 \end{equation}
so in particular $ r |\ln(\varpi)|^{\alpha/2}<1/2$.
Hence  from the above estimates and Remark~\ref{defvarpi} it follows  that
$$\|\FF^+(\vec{u})\|_{X^+}\leq  C|\ln(\varpi)|^{-\al/2}\, r \le \frac{r}{2}.$$

	Next, we show that $\FF^+$ is a contraction.
	First, by   \eqref{norms-} and \eqref{NN} we have the estimates:
	\begin{gather*}
		\|\vec{R}_1^+(\theta, \vec{u}_1(\theta))-\vec{R}_1^+(\theta, \vec{u}_2(\theta))\|
		\leq  N_\al (c_s)^{\al} |\ln(\varpi)|^{-\al/2}\, \big[z_s^+(\theta+\tau,\tau)\big]^{\al}\|\vec{u}_1(\theta)-\vec{u}_2(\theta)\|,\\
		\|\vec{R}_2^+(\theta, \vec{u}_1(\theta))-\vec{R}_2^+(\theta, \vec{u}_2(\theta))\|
		\leq N_\al  \big[2r |\ln(\varpi)|^{\al/2} \big]^{\al} \big[z_u^+(\theta+\tau,M+\tau)
\big]^{\al}\|\vec{u}_1(\theta)-\vec{u}_2(\theta)\|
	\end{gather*}
	for any $\vec{u}_1,\vec{u}_2\in \overline{B}_r^{X^+}$.
	Consequently, using
	$$\|\vec{u}_1(s)-\vec{u}_2(s)\|=z_u^+(s+\tau,M+\tau)\frac{\|\vec{u}_1(s)-\vec{u}_2(s)\|}{z_u^+(s+\tau,M+\tau)}\leq
z_u^+(s+\tau,M+\tau)\|\vec{u}_1-\vec{u}_2\|_{X^+}$$
	for all $s\in[0,M]$,
	we derive
	\begin{gather*}
		\left\|\int_0^\theta \bs{X^+}(\theta+\tau)\boldsymbol{P^+}(\bs{X^+}(s+\tau))^{-1}
		\left(\vec{R}^+_1(s,\vec{u}_1(s))ds-\vec{R}^+_1(s,\vec{u}_2(s))\right)ds\right\|\\
		\leq \frac{k_2 N_\al (c_s)^{\al}}{|\ln(\varpi)|^{\al/2}}\int_0^\theta z_s^+(\theta+\tau,s+\tau)
		[z_s^+(s+\tau,\tau)]^{\al}\|\vec{u}_1(s)-\vec{u}_2(s)\|ds\\
		\leq  \frac{2k_1^{2+\al}}{\underline{\la}} \frac{k_2 N_\al (c_s)^{\al}}{  |\ln(\varpi)|^{\al/2}
}\|\vec{u}_1-\vec{u}_2\|_{X^+}z_u^+(\theta+\tau,M+\tau),
	\end{gather*}
	\begin{gather*}
		\left\|\int_0^\theta \bs{X^+}(\theta+\tau)\boldsymbol{P^+}(\bs{X^+}(s+\tau))^{-1}
		\left(\vec{R}^+_2(s,\vec{u}_1(s))ds-\vec{R}^+_2(s,\vec{u}_2(s))\right)ds\right\|\\
		\leq    k_2  N_\al  \big[2r |\ln(\varpi)|^{\al/2} \big]^{\al}   \int_0^\theta z_s^+(\theta+\tau,s+\tau)
		\big[ z_u^+(s+\tau,M+\tau)\big]^{\al} \|\vec{u}_1(s)-\vec{u}_2(s)\|ds\\
		\leq   \frac{2k_1^{2+2\al}}{\underline{\la}(2+\al)}   k_2  N_\al   \,
\big[2r |\ln(\varpi)|^{ \al/2} \big]^{\al} \|\vec{u}_1-\vec{u}_2\|_{X^+}z_u^+(\theta+\tau,M+\tau),
	\end{gather*}
	\begin{gather*}
		\left\|\int_0^\theta \bs{X^+}(\theta+\tau)(\I-\boldsymbol{P^+})(\bs{X^+}(s+\tau))^{-1}
		\left(\vec{R}^+_1(s,\vec{u}_1(s))ds-\vec{R}^+_1(s,\vec{u}_2(s))\right)ds\right\|\\
		\leq \frac{k_2 N_\al (c_s)^{\al}}{  |\ln(\varpi)|^{\al/2} } \int_0^\theta z_u^+(\theta+\tau,s+\tau)
		\big[z_s^+(s+\tau,\tau) \big]^{\al} \|\vec{u}_1(s)-\vec{u}_2(s)\|ds\\
		\leq   \frac{2k_1^{\al}}{\underline{\la} \al }   \frac{k_2 N_\al (c_s)^{\al}}{  |\ln(\varpi)|^{\al/2} }
\|\vec{u}_1-\vec{u}_2\|_{X^+}z_u^+(\theta+\tau,M+\tau),
	\end{gather*}
 	and
 	\begin{gather*}
 		\left\|\int_0^\theta \bs{X^+}(\theta+\tau)(\I-\boldsymbol{P^+})(\bs{X^+}(s+\tau))^{-1}
 		\left(\vec{R}^+_2(s,\vec{u}_1(s))ds-\vec{R}^+_2(s,\vec{u}_2(s))\right)ds\right\|\\
 		\leq   k_2  N_\al  \big[2r |\ln(\varpi)|^{\al/2} \big]^{\al}  \int_0^\theta z_u^+(\theta+\tau,s+\tau)
 		[z_u^+(s+\tau,M+\tau)]^{\alpha} \|\vec{u}_1(s)-\vec{u}_2(s)\|ds\\
 		\leq   \frac{2k_1^{\al}}{\underline{\la} \al }    k_2  N_\al  \big[2r |\ln(\varpi)|^{\al/2} \big]^{\al}
 \|\vec{u}_1-\vec{u}_2\|_{X^+}z_u^+(\theta+\tau,M+\tau).
 	\end{gather*}
 Now  using \eqref{rde2}
 	and the previous four estimates,  recalling Remark \ref{defvarpi} and \eqref{rde2}
 we find
 	$$\|\FF^+(\vec{u}_1)-\FF^+(\vec{u}_2)\|_{X^+}\leq \frac{C}{|\ln(\varpi)|^{\al/2}}\,\|\vec{u}_1-\vec{u}_2\|_{X^+}\leq
 \frac{1}{2}\|\vec{u}_1-\vec{u}_2\|_{X^+}.$$

 	The proof is finished by the Banach fixed point theorem.
\end{proof}

\begin{remark}
	By fixing $\de>0$, the values of $M_0$ and $D_0$ are given. Then for any fixed $M\geq M_0$, $\vec{h}\in \overline{B}_r^{X^+}$ if and only
if
	$\|\vec{h}(\theta)\|\leq rz_u^+(\theta+\tau,M+\tau)$ for all $\theta\in[0,M]$. Here one may consider a sufficiently smooth extension of
\eqref{eq-disc+} beyond $\Om^0$ if needed. On the other hand, if $0<T<M_0$, taking $M\geq T+\frac{2}{\und{\la}}\ln k_1$ yields that if
$\vec{h}\in \overline{B}_r^{X^+}$ then
	\begin{gather*}
		\|\vec{h}(\theta)\|\leq rz_u^+(\theta+\tau,T+\tau)z_u^+(T+\tau,M+\tau)\\
		\leq rz_u^+(\theta+\tau,T+\tau)k_1\eu^{-\frac{1}{2}\und{\la}(M-T)}
		\leq rz_u^+(\theta+\tau,T+\tau)
	\end{gather*}
	for any $\theta\in[0,T]$.
\end{remark}

Now, the trajectory $\x(\theta+\tau,\tau; \Q^{+}(d,\tau))$   crosses the curve
  $\Om^0$   when its projection on the
``stable'' and on the ``unstable'' linear spaces are comparable.
This idea allows us to estimate the crossing time and displacement: see Lemma~\ref{estTD} below.

Observe that there are $K^{\pm}>0$ such that the unit vector
\begin{equation}\label{tangent}
  \frac{\vec{v}_s^+ - K^+ \vec{v}_u^+}{\|\vec{v}_s^+ - K^+ \vec{v}_u^+\|}
  =\frac{\vec{v}_u^- - K^- \vec{v}_s^-}{\|\vec{v}_u^- - K^- \vec{v}_s^-\|}
\end{equation}
is  tangent to
$\Omega^0$ in the origin and aims towards $E^{\textrm{in}}$.

Hence the vector $\rho(\vec{v}_s^+ - 2K^+\vec{v}_u^+)$ points towards $\Omega^-$ and $\rho(\vec{v}_s^+ -
\frac{1}{2}K^-\vec{v}_u^+)$ points
towards $\Omega^+$
for $\rho>0$ small enough, see Figure \ref{scenario123}.
Fix $\delta<\varpi$  and let $0<d<\delta$; then
\begin{align*}
	\rho \left(\vec{v}_s^{+}(\theta)- 2K^+\vec{v}_u^{+}(\theta) \right)  &\in \Omega^{-},\\
	\rho \left(\vec{v}_s^{+}(\theta)- \frac{1}{2}K^+\vec{v}_u^{+}(\theta) \right)  &\in \Omega^{+},
\qquad \textrm{for any $\theta \in \R$,}
\end{align*}
for $\rho>0$ small enough.
Hence there is $\tilde{K}^+=\tilde{K}^+(d,\tau) \in [\frac{1}{2}K^+,2 K^+]$ such that
$\x(T^f_2+\tau, \tau; \Q^+(d,\tau)) \in \Om^0$, where $T^f_2=T^f_2(d,\tau)$, verifies
$$\x(T^f_2+\tau, \tau; \Q^+(d,\tau)) \in \operatorname{span}\{\vec{v}_s^+  -\tilde{K}^+  \vec{v}_u^+   \}, $$
i.e., there is $\tilde{\rho}^+=\tilde{\rho}^+(T^f_2+\tau,d)>0$ such that
\begin{equation}
	\begin{gathered}\label{eq:span}
		\vec{\ell}^+(T^f_2)+\y_s(T^f_2)+\vec{h}^+(T^f_2) \\
		= \tilde{\rho}^+\{\vec{v}_s^+(T^f_2+\tau) -\tilde{K}^+  \vec{v}_u^+(T^f_2+\tau) \}.
	\end{gathered}
\end{equation}

\begin{remark}\label{defvarpibis}
   In the next Proposition \ref{estTD} we fix $\mu_0$ such that \eqref{mu0} holds, then
  for $\mu_2 \in ]0 , \mu_0/2]$ we let $\varpi=\varpi(\mu_2)$ be
   such that the following estimates hold true:
\begin{equation}\label{varpibis}
  |\ln(\varpi)| > K^+ c_sk_1^2 , \qquad \quad  \frac{k_1^2|\ln(\varpi)|}{K^+c_s} <\varpi^{-\mu_2/4}.
\end{equation}
\end{remark}

\begin{proposition}\label{estTD}
Assume \textbf{H}; let  $\mu_2 \in ]0, \mu_0/2]$  where $\mu_0$ satisfies \eqref{mu0} and choose $\varpi=\varpi(\mu_2)$ satisfying both
\eqref{est.varpi1} and
\eqref{varpibis}.
 Then there is  $\delta=\delta(\mu_2,\varpi)>0$ such that
for any  $0<d \le \delta $ and any $\tau \in \R$  we find
\begin{equation}\label{est.T2}
	\begin{gathered}
		\sTfwd_+\left(1-\frac{\mu_2}{3}\right) |\ln(d)|
		\le T^f_2(d,\tau)
		\le  \sTfwd_+\left(1+\frac{\mu_2}{3}\right)  |\ln(d)|
	\end{gathered}
\end{equation}
and
\begin{equation}\label{est.D-d}
	   d^{\sfwd_+  \left( 1+ \frac{\mu_2 \la_u^+}{|\la_s^+|} \right) }
	\le   D^f_2(d,\tau)
	\le  d^{\sfwd_+\left( 1- \frac{\mu_2 \la_u^+}{|\la_s^+|} \right) }
\end{equation}
where
\begin{equation}\label{D-d}
D^f_2(d,\tau):= \|\x(T^f_2(d,\tau)+\tau,\tau ; \Q^+(d,\tau))\|.
\end{equation}
\end{proposition}

\begin{proof}
Let us fix $\tau \in \R$; we choose $\varpi$ and $\delta<\varpi$ so that Lemma \ref{L-S-Lin-} holds
and we set $M=T^f_2(d,\tau)$, so that
 for any $0<d< \delta$  we find
$D_{\ell}(d,\tau):=dz_u^+(T^f_2(d,\tau)+\tau,\tau)$.  In this proof we write  $\Q^+$, $T^f_2$ and $D^+$ for $T^f_2(d,\tau)$,
$\Q^+(d,\tau)$ and $D^f_2(d,\tau)$
 to deal with less cumbersome notation.
  Notice that $T^f_2 \ge M_0$ (cf.~\eqref{D0M-})  so that the expansion \eqref{exp.h-} and Lemma
\ref{L-S-Lin-} hold for $0< t \le T^f_2$. Let
 $h^+_u(T^f_2)$ and $h^+_s(T^f_2)$ be such that
$$ \h^+(T^f_2)=  h^+_u(T^f_2) \vec{v}_u^+(T^f_2+\tau)+ h^+_s(T^f_2) \vec{v}_s^+(T^f_2+\tau).  $$
 Using \eqref{def:ys} and \eqref{def:ell-}, from \eqref{eq:span} we see that
\begin{equation}\label{est.T}
\begin{gathered}
    -\frac{1}{\tilde{K}^+(T^f_2+\tau)} \left[-dz_u^+(T^f_2+\tau,\tau)+ \tilde{a}^+_u(T^f_2) \left(\frac{z_s^+(T^f_2+\tau,\tau)}{|\ln(\varpi)|}
    \right)^{1+\alpha}+ h^+_u(T^f_2)\right]\\
     = \tilde{\rho}^+(T^f_2+\tau,d)
    = \tilde{a}^+_s(T^f_2) \frac{z_s^+(T^f_2+\tau,\tau)}{|\ln(\varpi)|}     \, + h^+_s(T^f_2).
\end{gathered}
\end{equation}
We want to compare the highest order terms of \eqref{est.T}.
First note that \\$\tilde{a}^+_u(T^f_2) \left(\frac{z_s^+(T^f_2+\tau,\tau)}{|\ln(\varpi)|} \right)^{1+\al}$ is negligible with respect
to $ \tilde{a}^+_s(T^f_2)\frac{z_s^+(T^f_2+\tau,\tau)}{|\ln(\varpi)|} $. Moreover,
$|h^+_s(T^f_2)|= \|\bs{P^{+,\ep}}(T^f_2+\tau) \h^+(T^f_2)\| \leq k_2 \|\h^+(T^f_2)\|$ and $|h^+_u(T^f_2)|=
\| (\I- \bs{P^{+,\ep}}(T^f_2+\tau) )\h^+(T^f_2)\| \leq k_2 \|\h^+(T^f_2)\|$, and using Lemma \ref{L-S-Lin-} we
get
\begin{gather*}
	 \|\h^+(T^f_2)\|\leq \frac{D_{\ell} z_u^+(T^f_2+\tau,T^f_2+\tau)}{ |\ln(\varpi)|^{\al/2}}
	=\frac{D_{\ell}}{ |\ln(\varpi)|^{\al/2}}
	\ll D_{\ell}.
\end{gather*}
Hence also $h^+_s(T^f_2)$ and $h^+_u(T^f_2)$ are negligible with respect to $D_{\ell}= d z_u^+(T^f_2+\tau,\tau)$.
 Thus the highest order terms in \eqref{est.T} are
$\frac{D_{\ell}}{\tilde{K}^+(T^f_2+\tau)}$
and $\tilde{a}^+_s(T^f_2)\frac{z_s^+(T^f_2+\tau,\tau)}{|\ln(\varpi)|} $.

So, picking up the highest order terms and neglecting the others, we can find a constant $\tilde{C}_K \in
\left[\frac{K^+\bar{c}_s}{2},K^+c_s\right]$ (cf.~\eqref{def:ys})
such that
\begin{equation}\label{D-}
   D_{\ell}(d,\tau)= d z_u^+(T^f_2+\tau,\tau) =\tilde{C}_K\frac{z_s^+(T^f_2+\tau,\tau)}{|\ln(\varpi)|}.
\end{equation}
Now using
$$
\frac{1}{k_1^2}\eu^{(\la_u^++|\la_s^+|-2k\ep)T^f_2}
	\leq \frac{z_u^+(T^f_2+\tau,\tau)}{z_s^+(T^f_2+\tau,\tau)}=
   \frac{\tilde{C}_K}{d|\ln(\varpi)|}\leq k_1^2\eu^{(\la_u^++|\la_s^+|+2k\ep)T^f_2}
$$
we get

\begin{equation}\label{est.T22}
\begin{gathered}
  \frac{\ln\left(\frac{\tilde{C}_K}{k_1^2d|\ln\varpi|}\right)}{\la_u^++|\la_s^+|+2k\ep}
\le  T^f_2 \le \frac{\ln\left(\frac{\tilde{C}_Kk_1^2}{d|\ln(\varpi)|}\right)}{\la_u^++|\la_s^+|-2k\ep},
\end{gathered}
\end{equation}
$$ \ln\left(\frac{\tilde{C}_Kk_1^2}{d|\ln(\varpi)|}\right) = |\ln(d)| + \ln(\tilde{C}_Kk_1^2)-\ln(|\ln(\varpi)|)< |\ln(d)|$$
by the first inequality in \eqref{varpibis};
further from the second inequality in \eqref{varpibis} we get
$$\ln\left(\frac{\tilde{C}_K}{k_1^2d|\ln\varpi|}\right) \ge -\ln(d\varpi^{-\mu_2/4}) \ge  \left( 1-\frac{\mu_2}{4}\right) |\ln(d)|.
$$
So plugging these inequalities in \eqref{est.T22} we get
\begin{equation}\label{est.T22a}
\begin{gathered}
\frac{1-\mu_2/3}{\la_u^++|\la_s^+|} |\ln(d)|   \le \frac{(1-\mu_2/4)(1-c_1 \ep)}{\la_u^++|\la_s^+|}
|\ln(d)| \\
\le \frac{1-\mu_2/4}{\la_u^++|\la_s^+|+2k\ep}|\ln(d)|
\le  T^f_2 \le \frac{1}{\la_u^++|\la_s^+|-2k\ep}|\ln(d)|\\
\le  \frac{1+c_1 \ep}{\la_u^++|\la_s^+|} |\ln(d)| \le \frac{1+ \mu_2/3}{\la_u^++|\la_s^+|}|\ln(d)|
\end{gathered}
\end{equation}
where $c_1= 4k /(|\la_s^+|+\la_u^+)$; so \eqref{est.T2} is proved.

Then plugging \eqref{est.T2} in \eqref{D-}     we find
\begin{gather*}
	D_{\ell}(d,\tau)\leq k_1 d \left(\eu^{T^f_2}\right)^{\la_u^++k\ep} \le k_1 d^{1- \left(
\frac{1+\mu_2/3}{\la_u^++|\la_s^+|}
\right)(\la_u^++k\ep)}  \\
= k_1 d^{\sfwd_+ -  \left( \frac{\mu_2 \la_u^+}{3(\la_u^++|\la_s^+|)}  \right)- \left( \frac{k \ep (1+\mu_2/3)
}{\la_u^++|\la_s^+|}  \right) }
\\
\le  d^{\sfwd_+ -    \frac{\mu_2 \la_u^+}{2(\la_u^++|\la_s^+|)} }= d^{\sfwd_+ \left(1 -    \frac{\mu_2
\la_u^+}{2|\la_s^+|} \right)}.
\end{gather*}
Similarly, plugging again \eqref{est.T2} in \eqref{D-},
\begin{gather*}
	D_{\ell}(d, \tau)\geq \frac{d}{k_1}\left(\eu^{T^f_2}\right)^{\la_u^+-k\ep}
	\geq \frac{d^{1- \left( \frac{1-\mu_2/3}{\la_u^++|\la_s^+|}  \right)(\la_u^+-k\ep)}}{k_1} \\
=  \frac{d^{\sfwd_+ +  \left( \frac{\la_u^+\mu_2}{3(\la_u^++|\la_s^+|)}  \right)+\left( \frac{k \ep (1-\mu_2/3)
}{\la_u^++|\la_s^+|}  \right)
}}{k_1}\\
\ge  d^{\sfwd_+ +    \frac{\mu_2 \la_u^+}{2(\la_u^++|\la_s^+|)} }=  d^{\sfwd_+ \left(1  +  \frac{\mu_2
\la_u^+}{2|\la_s^+|} \right)}.
\end{gather*}

Summing up we have shown that
\begin{equation}\label{newDell}
d^{\sfwd_+ \left(1  + \frac{\mu_2 \la_u^+}{2|\la_s^+|} \right)} \le  D_{\ell}(d, \tau)
\le  d^{\sfwd_+ \left(1  - \frac{\mu_2 \la_u^+}{2|\la_s^+|} \right)}.
\end{equation}
Now, again from \eqref{est.T} we see that there are $\bar{c}_2> \bar{c}_1>0$ (independent of
$d$ and $\tau$) such that
$$\bar{c}_1 D_{\ell}(d, \tau) \le  \tilde{\rho}^+(T^f_2+\tau,d) \le \bar{c}_2 D_{\ell}(d, \tau);$$
further from \eqref{eq:span} and \eqref{D-d} we see that there are $\bar{c}_4> \bar{c}_3>0$ such that
$$\bar{c}_3  \tilde{\rho}^+(T^f_2+\tau,d) \le  D^f_2(d, \tau)
\le \bar{c}_4  \tilde{\rho}^+(T^f_2+\tau,d) .  $$
So, these last two inequalities  imply that there are $C>c>0$, independent of $d$, $\tau$, $\ep$   and $\mu_2$,
such that
\begin{equation}\label{DandDell0}
c D_{\ell}(d,\tau) \le D^f_2(d,\tau) \le C D_{\ell}(d,\tau).
\end{equation}
Now assuming  that $\delta^{  \frac{\mu_2 \la_u^+}{2|\la_s^+|} \sfwd_+} \le  \min  \{ C^{-1}, c \}$,
 the proof of \eqref{est.D-d} follows from \eqref{newDell} and \eqref{DandDell0}.
 This completes the proof of Proposition \ref{estTD}.
\end{proof}

For future reference it is convenient to state the following remark concerning estimate \eqref{DandDell0}.
\begin{remark}\label{DandDell}
  There are $C>c>0$ independent of $d$, $\tau$, $\ep$   and $\mu_2$  such that
  $$c D_{\ell}(d,\tau) \le D^f_2(d,\tau) \le C D_{\ell}(d,\tau).$$
\end{remark}

Now we want to prove a result which is, to some extent, the converse of Proposition \ref{estTD}. Let us set
$\de_1=\de^{\sfwd_+\left( 1 + \frac{\mu_2 \la_u^+}{|\la_s^+|} \right)}$
and consider the set $L^{\inn}(\de_1)$.
For any $0<D<\de_1$ we denote by  $\vec{R}(D)$  the unique point in $ L^{\inn}(\de_1)$ such that
$\|\vec{R}(D)\|=D$.

Then, adapting slightly Lemma \ref{L.loop}, we prove the following.
\begin{remark}\label{welldefined}
Fix $\tau_1 \in \R$,
$0<D< \de_1=\de^{\sfwd_+\left( 1 + \frac{\mu_2 \la_u^+}{|\la_s^+|}
 \right) }$ and follow  the trajectory $\x(t,\tau_1; \vec{R}(D))$ backward in time.
Then there are  $D^b_3(D,\tau_1)>0$, $T^b_3(D,\tau_1)>0$ and $\tau_0^+(D,\tau_1)= \tau_1-T^b_3(D,\tau_1)$
such that $\x(t,\tau_1; \vec{R}(D)) \in K_A^{\textrm{bwd},+}(t)$ for any $  \tau^+_0(D,\tau_1) \le t < \tau_1$
and it crosses transversely $\tilde{S}^+(\tau^+_0(D,\tau_1) )$ in the point $\Q^+(D^b_3(D,\tau_1),\tau_0^+(D,\tau_1))$,
where  $\Q^+(\cdot,\cdot)$ is as defined in \eqref{Q-Q+}.
\end{remark}

Using the uniformity in $\tau$ of the estimates in \eqref{estTD} we show the following.

\begin{figure}[t]
\centerline{\epsfig{file=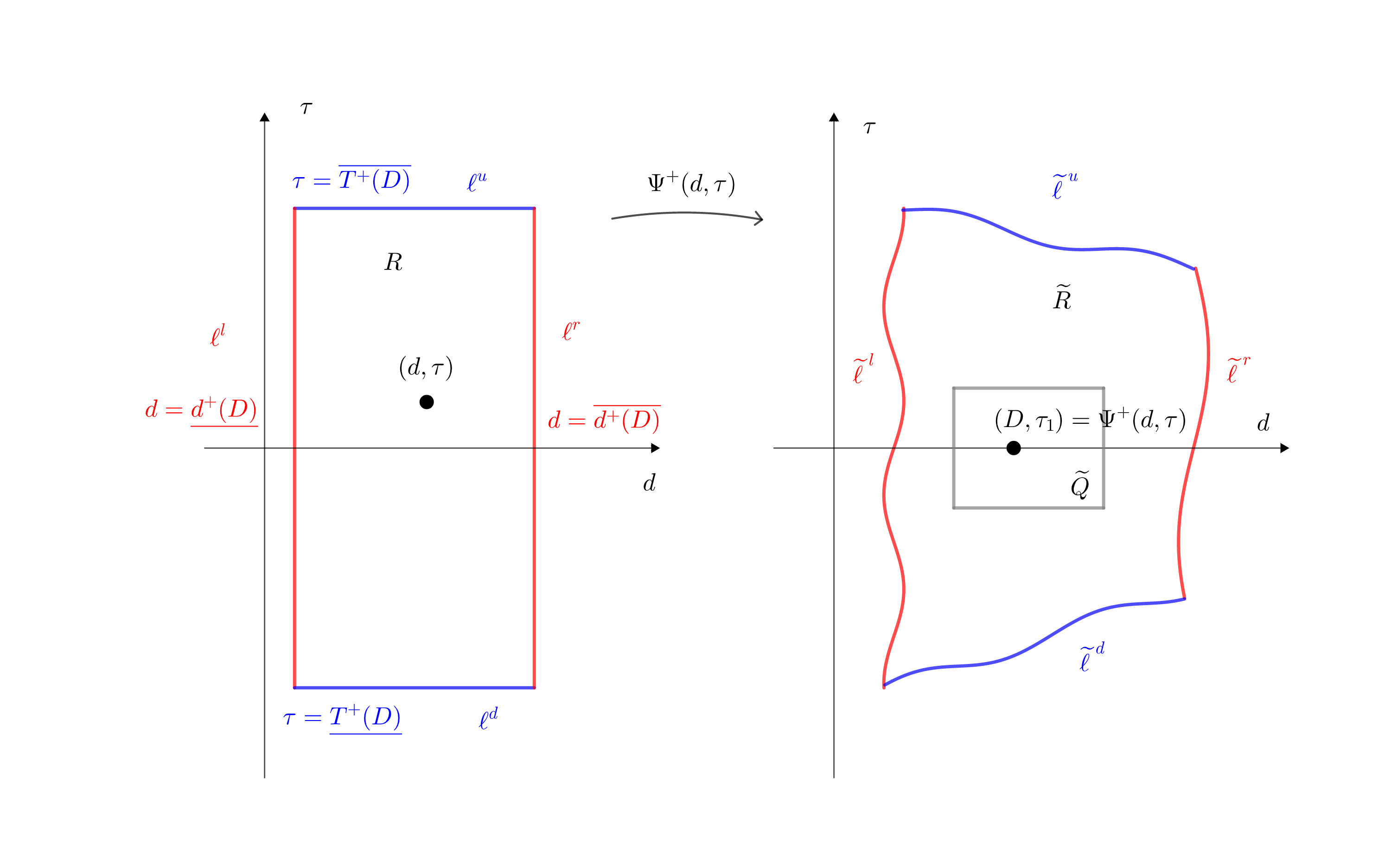, width = 11 cm}}
\caption{An explanation of the proof of Proposition \ref{lemmaconverse-}.}
\label{nextstable}
\end{figure}

\begin{proposition}\label{lemmaconverse-}
Assume \textbf{H}; let  $\mu_2 \in ]0, \mu_0/2]$  where $\mu_0$ satisfies \eqref{mu0} and choose  $\varpi=\varpi(\mu_2)$
satisfying both \eqref{est.varpi1} and \eqref{varpibis}. Then there is  $\delta=\delta(\mu_2,\varpi)>0$ such that
for any $0<D< \de_1=\de^{\sfwd_+\left( 1 + \frac{\mu_2 \la_u^+}{|\la_s^+|} \right) }$ and any $\tau_1 \in \R$  the functions
$T^b_3(D,\tau_1)$ and $D^b_3(D,\tau_1)$ are well
defined, $C^{r}$ and one-to-one. Further
    \begin{equation*}
      \begin{split}
          &    \frac{1}{|\la_s^+|} \left[ 1-   \mu_2 \left( 2 \tfrac{\la_u^+}{|\la_s^+|} +\tfrac{1}{2} \right)
          \right] |\ln(D)| \le
          T^b_3(D,\tau_1) \le
          \frac{1}{|\la_s^+|} \left[ 1+   \mu_2 \left( 2
          \tfrac{\la_u^+}{|\la_s^+|} +\tfrac{1}{2} \right) \right] |\ln(D)|,
           \\
           &      D^{\frac{1}{\sfwd_+}\left(1+2\frac{\mu_2 \la_u^+}{|\la_s^+|} \right)} \le D^b_3 (D,\tau_1) \le
           D^{\frac{1}{\sfwd_+}\left(1-2\frac{\mu_2
           \la_u^+}{|\la_s^+|} \right) }.
      \end{split}
    \end{equation*}
    In particular $\lim_{D \to 0} D^b_3(D,\tau_1)=0$ uniformly in $\tau_1 \in \R$.
\end{proposition}
\begin{proof}
	The proof is easier to follow while keeping in mind Figure \ref{nextstable}.
  From Remark \ref{welldefined} we see that the functions $ T^b_3 (D,\tau_1)$ and $D^b_3(D,\tau_1)$ are well defined; further they are $C^r$
  and
  one-to-one due to the smoothness of the flow, see Remark \ref{data.smooth}.

  Let us consider the following function, as in Remark \ref{RemarkDT-injective}:
     \begin{equation*}
      \begin{split}
           &      \Psi^+(d,\tau): ]0,\de] \times \R \to  ]0,1] \times \R, \qquad  \Psi^+(d,\tau)= (D^f_2(d,\tau),  T^f_2(d,\tau)+\tau ).
      \end{split}
    \end{equation*}
  $\bullet$\ \emph{We claim that $\big( ]0,\delta_1] \times \R \big) \subset \Psi^+ \big( ]0,\de] \times \R     \big)$}.\\
  Let $(D,\tau_1) \in \big( ]0,\delta_1] \times \R \big)$: we need to show that there is $(d,\tau) \in \big( ]0,\de] \times \R \big)$
  such  that $\Psi^+(d,\tau)=(D,\tau_1)$.
  We denote by $d_0= D^{1/\sfwd_+}$; then we set
  $$\underline{d^+}(D)=   D^{\frac{1}{\sfwd_+} \left(1+\frac{3\mu_2 \la_u^+}{|\la_s^+|}\right)} = d_0^{1+\frac{3\mu_2
  \la_u^+}{|\la_s^+|} } \,
  , $$
  $$\overline{d^+}(D)=  D^{\frac{1}{\sfwd_+} \left(1-\frac{3\mu_2 \la_u^+}{|\la_s^+|} \right)}= d_0^{1-\frac{3\mu_2
  \la_u^+}{|\la_s^+|} }\,
  ,$$
  $$\underline{T^+}(D)= \tau_1 - \left[  \sTfwd_+ \left(1  + \frac{\mu_2}{3} \right) \left( 1+\frac{3\mu_2 \la_u^+}{
  |\la_s^+|}  \right) +2
  \sfwd_+ \right] |\ln(d_0)|   \, ,
  $$
  $$ \overline{T^+}(D)= \tau_1 -   \left[ \sTfwd_+\left(1  - \frac{\mu_2}{3} \right) \left( 1-\frac{3\mu_2 \la_u^+}{
  |\la_s^+|}  \right) -2
  \sfwd_+  \right] |\ln(d_0)|. $$
  Then consider the rectangle
  $$\mathcal{R}=  \big([\underline{d^+}(D), \overline{d^+}(D)] \times [\underline{T^+}(D), \overline{T^+}(D)] \big)\subset ]0, 1/10] \times
  \R, $$
  and denote by $\partial \mathcal{R}$ its border. Further set
  $$\ell^l= \{\underline{d^+}(D)\} \times [\underline{T^+}(D), \overline{T^+}(D)],$$
  $$\ell^r= \{\overline{d^+}(D)\} \times [\underline{T^+}(D), \overline{T^+}(D)],$$
  $$\ell^d= [\underline{d^+}(D), \overline{d^+}(D)]  \times \{\underline{T^+}(D)\},$$
  $$\ell^u= [\underline{d^+}(D), \overline{d^+}(D)]  \times \{\overline{T^+}(D)\},$$
   so that
  $\partial \mathcal{R}= \big(\ell^l \cup \ell^d \cup \ell^r \cup \ell^u \big)$.

  Now observe that the restriction $\tilde{\Psi}^+$ of $\Psi^+$ defined by $\tilde{\Psi}^+: \mathcal{R} \to \tilde{\mathcal{R}}$
  where $\tilde{\mathcal{R}}= \Psi^+(\mathcal{R})$, is a $C^r$ diffeomorphism,  see Remark \ref{data.smooth}.
  Hence, from Brouwer's domain invariance theorem (cf.~\cite[Proposition 16.9]{Zeidler}), we see that
  $\tilde{\mathcal{R}}$ is a compact and simply connected set and its border $\partial \tilde{\mathcal{R}}$ satisfies
    $\partial\tilde{\mathcal{R}}= \Psi(\partial\mathcal{R})$.

  Let $\tilde{\ell}^{l}= \Psi(\ell^{l})$ and similarly for  $\tilde{\ell}^{r}$, $\tilde{\ell}^{d}$,
  $\tilde{\ell}^{u}$: all these sets are compact images of curves.
  Let
  $$\tilde{\mathcal{Q}}= \left[ D^{1+  \frac{\mu_2 \la_u^+}{|\la_s^+|} }, D^{1-  \frac{\mu_2
  \la_u^+}{|\la_s^+|} } \right] \times [\tau_1 -
  |\ln(D)|, \tau_1 +   |\ln(D)|]$$
  and denote by $\partial  \tilde{\mathcal{Q}}$ its border: obviously $(D, \tau_1) \in \tilde{\mathcal{Q}}$.
  Let $\tau',\tau'' \in [\underline{T^+}(D), \overline{T^+}(D)]$  and consider  $(\underline{d^+}(D),\tau') \in \ell^l$ and
$(\overline{d^+}(D),\tau'') \in \ell^r$;  from Proposition \ref{estTD} we find

\begin{equation*}
  \begin{split}
  D^f_2 (\underline{d^+}(D),\tau') &  \le [\underline{d^+}(D)]^{\sfwd_+ \left( 1- \mu_2 \frac{ \la_u^+ }{|\la_s^+|}  \right)
  }  = D^{\left(
  1- \mu_2 \frac{ \la_u^+
  }{|\la_s^+|}  \right) \left( 1+ 3\mu_2 \frac{ \la_u^+ }{|\la_s^+|}  \right)  }         \\
& = D^{  1  +2\mu_2 \frac{ \la_u^+ }{|\la_s^+|}- 3\mu_2^2 \frac{ (\la_u^+)^2 }{(\la_s^+)^2}  } <
D^{  1  + \mu_2 \frac{ \la_u^+
}{|\la_s^+|}  },
  \end{split}
\end{equation*}
\begin{equation*}
  \begin{split}
     D^f_2 (\overline{d^+}(D),\tau'') & \ge  [\overline{d^+}(D)]^{\sfwd_+ \left( 1+ \mu_2 \frac{ \la_u^+ }{|\la_s^+|}
     \right)}  =
   D^{\left( 1+ \mu_2 \frac{ \la_u^+ }{|\la_s^+|}  \right) \left( 1- 3\mu_2 \frac{ \la_u^+ }{|\la_s^+|}
   \right)  }  \\ &
= D^{  1  -2\mu_2 \frac{ \la_u^+ }{|\la_s^+|}- 3\mu_2^2 \frac{ (\la_u^+)^2 }{(\la_s^+)^2}  } >
D^{  1  - \mu_2 \frac{ \la_u^+
}{|\la_s^+|}  }.
  \end{split}
\end{equation*}

Whence  $\tilde{\ell}^{l} \cap \tilde{\mathcal{Q}} = \emptyset= \tilde{\ell}^{r} \cap \tilde{\mathcal{Q}} $.

Similarly observe that for any $d',d'' \in  [\underline{d^+}(D), \overline{d^+}(D)]$ (so that $(d',\underline{T^+}(D)) \in \ell^d$ and
$(d'',\overline{T^+}(D)) \in \ell^u$), from Proposition \ref{estTD} we find
\begin{equation*}
  \begin{split}
         T^f_2(d',\underline{T^+}(D))+\underline{T^+}(D)& \le \underline{T^+}(D)+ \sTfwd_+\left(1 +  \frac{\mu_2}{3}\right)
         |\ln(\underline{d^+}(D))|\\
       & = \tau_1 -  2\sfwd_+ |\ln(d_0)| < \tau_1 -  \sfwd_+ |\ln(d_0)|=\tau_1- |\ln(D)| , \\
         T^f_2(d'',\overline{T^+}(D)) +\overline{T^+}(D)& \ge \overline{T^+}(D)+ \sTfwd_+\left(1 -  \frac{\mu_2}{3}\right)
         |\ln(\overline{d^+}(D))| \\
      & = \tau_1 +  2\sfwd_+ |\ln(d_0)| > \tau_1 +  \sfwd_+ |\ln(d_0)|=\tau_1+ |\ln(D)| .
  \end{split}
\end{equation*}
Whence  $\tilde{\ell}^{d} \cap \tilde{\mathcal{Q}} = \emptyset= \tilde{\ell}^{u} \cap \tilde{\mathcal{Q}} $.
\\
Summing up we have shown that $(\tilde{\mathcal{Q}} \cap \partial \tilde{\mathcal{R}})= \emptyset$;
in fact we easily see that $\tilde{\mathcal{Q}}$ is contained in the bounded set enclosed by $\partial \tilde{\mathcal{R}}= \Psi(\partial
\mathcal{R})$.
So we see that there is $(d,\tau)\in \mathcal{R}$ such that
$\Psi^+(d,\tau)=(D,\tau_1)$
  and the claim is proved, by the
arbitrariness of $(D,\tau_1) \in \big(]0,\de_1] \times \R \big)$.

Now from the claim we see that for any $(D,\tau_1) \in ]0,\delta_1] \times \R$ there is a unique couple
$(d,\tau_0) \in ]0, \de] \times \R$  such that
$(D,\tau_1) = \Psi^+(d,\tau_0)$, since $\Psi^+$ is  defined through the flow of \eqref{eq-disc}.

 Hence from Proposition \ref{estTD} we get
$$ d^{\sfwd_+ \left( 1+\frac{\mu_2 \la_u^+}{ |\la_s^+|}\right)}\le D \le
d^{\sfwd_+ \left( 1-\frac{\mu_2 \la_u^+}{ |\la_s^+|}\right)}$$
from which we get
$$ D^{   \frac{1}{\sfwd_+} \left( 1+\frac{2\mu_2 \la_u^+}{ |\la_s^+|}\right) }\le  D^{ \frac{1}{\sfwd_+} \left(
1-\frac{\mu_2  \la_u^+}{
|\la_s^+|}\right)^{-1}} \le     d \le
D^{  \frac{1}{\sfwd_+} \left( 1+\frac{\mu_2 \la_u^+}{ |\la_s^+|}\right)^{-1}} \le D^{   \frac{1}{\sfwd_+} \left(
1-\frac{2\mu_2 \la_u^+}{
|\la_s^+|}\right) }.
$$
So the assertion concerning $D$ is proved.

Again from Proposition \ref{estTD}  we find
$$\sTfwd_+ \left( 1-\frac{\mu_2}{3} \right)   |\ln(d)|  \le
\tau_1-\tau= T^b_3 (D,\tau_1) = T^f_2(d,\tau) \le \sTfwd_+ \left( 1+\frac{\mu_2}{3} \right)  |\ln(d)|.
$$
So we   conclude
\begin{equation*}
  \begin{split}
T^b_3 (D,\tau_1) \le & \frac{\sTfwd_+}{\sfwd_+}  \left( 1+\frac{\mu_2}{3} \right) \left(1 +2\mu_2 \frac{\la_u^+}{|\la_s^+|}
\right)
|\ln(D)|
\\ \le &\frac{1}{|\la_s^+|}
\left[ 1+   \mu_2 \left( 2 \frac{\la_u^+}{|\la_s^+|} +\frac{1}{2} \right) \right] |\ln(D)|,\\
T^b_3 (D,\tau_1) \ge & \frac{\sTfwd_+}{\sfwd_+}  \left( 1-\frac{\mu_2}{3} \right) \left(1 -2\mu_2 \frac{\la_u^+}{|\la_s^+|}
\right)
|\ln(D)|
\\ \ge & \frac{1}{|\la_s^+|}
\left[ 1-   \mu_2 \left( 2 \frac{\la_u^+}{|\la_s^+|} +\frac{1}{2} \right) \right] |\ln(D)|.
  \end{split}
\end{equation*}
   Then the assertion concerning $\michal{T^b_3} (D,\tau_1)$ is  proved.
\end{proof} 

\begin{lemma}\label{1keymissed}
Assume \textbf{H} and  let $\mu_2$, $\varpi$ and $\de=\delta(\mu_2,\varpi)$ be as in Proposition~\ref{estTD}.
  Let $0<d \le \de$, $\tau_0 \in \R$, and consider the trajectory
  $$\x(t,\tau_0; \Q^+(d,\tau_0)) \equiv x(t,\tau_1; \vec{R}(D)) \qquad  \tau_0 \le t \le \tau_1$$
  where $D=D^f_2(d,\tau_0)$, $\vec{R}(D)$ is the point of $L^{\textrm{in}}$ such that $\|\vec{R}(D)\|=D \in ]0, \delta_1[$  and $\tau_1=\tau_0
  +T^f_2(d,\tau_0)$.
  Then
  \begin{gather*}
  	\|\x(t,\tau_0; \Q^+(d,\ev{\tau_0}))- \vec{y}_s(t-\tau_0)\| \\
  \le
   \frac{2}{c} d^{\sfwd_+\left(1-\mu_2 \frac{\la_u^+}{|\la_s^+|}\right)  } z_u^+(t, \tau_1) \le  \frac{2k_1}{c}d^{\sfwd_+\left(1-\mu_2
   \frac{\la_u^+}{|\la_s^+|}\right)  } \eu^{-\frac12\underline{\la}(\tau_1-t)}
  \end{gather*}
   whenever $\tau_0 \le t \le \tau_1$,  where $c>0$ is as in Remark \ref{DandDell}.

     Analogously let  $0<D \le \de_1$, $\tau_1 \in \R$, and consider the trajectory
  $$\x(t,\tau_0; \Q^+(d,\tau_0)) \equiv x(t,\tau_1; \vec{R}(D)) \qquad  \tau_0 \le t \le \tau_1$$
  where $d= D^b_3 (D,\tau_1)$ and $\tau_0=\tau_1 - T^b_3 (D,\tau_1)$.
  Then
  $$\|x(t,\tau_1; \vec{R}(D))- \vec{y}_s(t-\tau_0)\| \le
   \frac{2}{c} D z_u^+(t, \tau_1) \le  \frac{2k_1}{c} D \eu^{-\frac12\underline{\la}(\tau_1-t)}  $$
   whenever $\tau_0 \le t \le \tau_1$.
\end{lemma}
\begin{proof}
  Let us decompose $\x(t,\tau_0; \Q^+(d,\tau))$ as in \eqref{exp.h-}, and set $\tau_1=\tau_0 +T^f_2(d,\tau_0)$. From Lemma \ref{L-S-Lin-} and
  Remark \ref{DandDell}
  we find
  $$ \|\vec{h}^+(\theta)\| \le \| \vec{h}^+ \|_{X^+} z_u^+(t,  \tau_1) \le
  \frac{k_1}{c} |\ln( \varpi )|^{-\al/2} D  \eu^{-\frac12\underline{\la} (\tau_1-t)} $$
  where $t=\theta+\tau_0$.
  Further from \eqref{norms-} and Remark \ref{DandDell}
  $$\|\vec{\ell}^+(\theta)\| \le D_{\ell} z_u^+(t,  \tau_1) \le
  \frac{k_1}{c}  D  \eu^{-\frac12\underline{\la} (\tau_1-t)} .$$
  Therefore
  \begin{equation}\label{justD}
  \begin{split}
      &  \|\x(\theta+ \tau_0,\tau_0; \Q^+(d,\tau))- \vec{y}_s(\theta)\| \le  2\frac{k_1}{c} \eu^{-\frac12\underline{\la}
      (T^f_2(d,\tau_0)-\theta)}  D.
  \end{split}
  \end{equation}
   Whence, using Proposition \ref{estTD}, we conclude the first part of the  proof; the second part of the proof is analogous and it is
   omitted.
\end{proof}

 \subsection{The loop: from  $\boldsymbol{S^-(\tau)}$ to $\boldsymbol{L^\textrm{in}}$}\label{S+Lin}

Let $\tilde{S}^-$ and $S^-(\tau)$ be as in \eqref{S+} and \eqref{S+tau};
let us choose $\Q^-(d,\tau)=-d \vec{v}_s^- + \PPu(\tau) \in S^-(\tau)$.
We aim to follow the trajectory $\x(t,\tau; \Q^-(d,\tau))$ from  $\tilde{S}^-$ backwards to $\Omega^0$.

The next lemma is analogous to Lemma \ref{L.0} and again it follows from the argument of Lemma \ref{L.loop}, taking into account the  mutual
positions of
  $\vec{v}_u^\pm$, $\vec{v}_s^\pm$.
\begin{lemma}\label{L.0bis}
 Assume \textbf{H} and fix $\tau\in\R$.
  For any $0<d \le \delta$ there are $C^r$ functions $T^b_2(d,\tau)$  and $D^b_2(d,\tau)$ such that
  $\x (\theta+\tau,\tau;\Q^-(d,\tau))\in K_A^{\textrm{bwd},-}$ for any $-T^b_2(d,\tau)<\theta\leq 0$ and it crosses transversely
  $\Om^0$  at $\theta=-T^b_2(d,\tau)$ in a point such that
  $\| \x(-T^b_2(d,\tau)+\tau,\tau;\Q^-(d,\tau))\|=D^b_2(d,\tau)$.
\end{lemma}

In order to get estimates of $T^b_2(d,\tau)$ and $D^b_2(d,\tau)$, we perform an inversion of time
argument.

Let us denote by $\underline{\f}^{\pm}(\x)=- \f^{\pm}(\x)$ and $\underline{\g}(t,\x,\ep)=-\g(-t,\x,\ep)$, and notice that
  if $\x(t)$ is a solution of \eqref{eq-disc}, then $\underline{\x}(t)=\x(-t)$ is a solution of
\begin{equation}\label{eq.modbis}
  \dot{\x}=\underline{\f}^\pm(\x)+\ep\underline{\g}(t,\x,\ep),\quad \x\in\Om^\pm.
\end{equation}
  Further, $\bs{\frac{\partial\underline{\f}^+}{\partial x}}(\vec{0})$ has eigenvalues $-\la_u^+<0<-\la_s^+$ with
  eigenvectors $\vec{v}_u^+$
  and  $\vec{v}_s^+$
  respectively, while $\bs{\frac{\partial\underline{\f}^-}{\partial x}}(\vec{0})$ has eigenvalues $-\la_u^-<0<-\la_s^-$ with
  eigenvectors
  $\vec{v}_u^-$ and  $\vec{v}_s^-$
  respectively.

  So applying Proposition \ref{estTD} to \eqref{eq.modbis} and then going back to the original system we obtain the following.
\begin{proposition}\label{estTDbis}
Assume \textbf{H}; let  $\mu_2 \in ]0, \mu_0/2]$  where $\mu_0$ satisfies \eqref{mu0} and choose $\varpi=\varpi(\mu_2)$ satisfying both \eqref{est.varpi1} and \eqref{varpibis}.
 Then there is  $\delta=\delta(\mu_2,\varpi)>0$ such that
for any  $0<d< \delta $ and any $\tau \in \R$  we find
\begin{equation}\label{est.T2bis}
	\begin{gathered}
		\sTbwd_- \left(1-\frac{\mu_2 }{3}\right) |\ln(d)|
		\le T^b_2(d,\tau)
		\le  \sTbwd_- \left(1+\frac{\mu_2 }{3}\right)  |\ln(d)|
	\end{gathered}
\end{equation}
and
\begin{equation}\label{est.D-d+}
	   d^{\sbwd_- \left( 1+ \frac{\mu_2 |\la_s^-|}{\la_u^-} \right)}
	\le   D^b_2(d,\tau)
	\le  d^{\sbwd_- \left( 1 -\frac{\mu_2 |\la_s^-|}{\la_u^-} \right)}
\end{equation}
where
\begin{equation}\label{D-d+}
D^b_2(d,\tau):= \|\x(\tau-T^b_2(d,\tau),\tau ; Q^-(d,\tau))\|.
\end{equation}
\end{proposition}
Similarly applying \michal{Remark \ref{welldefined} and} Proposition \ref{lemmaconverse-} to \eqref{eq.modbis} and then going back to the original system
we get the following.
\michal{
	\begin{remark}\label{welldefined+}
		Fix $\tau_{-1} \in \R$,
		$0<D< \de_{-1}=\de^{\sbwd_- \left(1+\frac{\mu_2 |\la_s^-|}{ \la_u^-} \right)}$
		and follow the trajectory $\x(t,\tau_{-1}; \vec{R}(D))$ forward in time.
		Then there are  $D^f_3(D,\tau_{-1})>0$, $T^f_3(D,\tau_{-1})>0$ and $\tau_0^-(D,\tau_{-1})= \tau_{-1}+T^f_3(D,\tau_{-1})$
		such that $\x(t,\tau_{-1}; \vec{R}(D)) \in K_A^{\textrm{fwd},-}(t)$ for any $ \tau_{-1}<t\leq \tau^-_0(D,\tau_{-1}) $
		and it crosses transversely $\tilde{S}^-(\tau^-_0(D,\tau_{-1}) )$ in the point $\Q^-(D^f_3(D,\tau_{-1}),\tau_0^-(D,\tau_{-1}))$,
		where  $\Q^-(\cdot,\cdot)$ is as defined in \eqref{Q-Q+}.
	\end{remark}
}
\begin{proposition}\label{lemmaconverse+}
Assume \textbf{H}; let  $\mu_2 \in ]0, \mu_0/2]$  where $\mu_0$ satisfies \eqref{mu0} and choose $\varpi=\varpi(\mu_2)$ satisfying both \eqref{est.varpi1} and  \eqref{varpibis}.
 Then there is  $\delta=\delta(\mu_2,\varpi)>0$ such that
for any $0<D< \de_{-1}=\de^{\sbwd_- \left(1+\frac{\mu_2 |\la_s^-|}{ \la_u^-} \right)}$ and any $\tau_{-1} \in \R$   the
functions $
T^{f}_3(D,\tau_{-1})$ and
$D^{f}_3(D,\tau_{-1})$ are well defined, $C^{r}$ and one-to-one. Further
    \begin{equation*}
      \begin{split}
           &    \frac{1}{\la_u^-}\left[ 1-   \mu_2 \left( 2 \frac{ |\la_s^-|}{\la_u^-} +\frac{1}{2} \right) \right]
           |\ln(D)| \le
           T^{f}_3(D,\tau_{-1}) \le
           \frac{1}{\la_u^-}\left[ 1+   \mu_2 \left( 2 \frac{ |\la_s^-|}{\la_u^-} +\frac{1}{2} \right) \right]
           |\ln(D)|,
           \\
           &      D^{\frac{1}{\sbwd_-}  \left(1+ 2 \frac{\mu_2 |\la_s^-|}{\la_u^-}   \right) } \le
           D^{f}_3(D,\tau_{-1}) \le
            D^{\frac{1}{\sbwd_-}  \left(1- 2 \frac{\mu_2 |\la_s^-|}{\la_u^-}   \right) }.
      \end{split}
    \end{equation*}
    In particular $\lim_{D \to 0} D^{f}_3(D,\tau_{-1})=0$ uniformly in $\tau_{-1} \in \R$.
\end{proposition}

 Let us recall that $\y_u(\theta)=\x(\theta+\tau,\tau;\PPu(\tau))$, see \eqref{S+}.
We conclude the section with the following result, obtained applying the inversion of time argument to Lemma \ref{1keymissed}.

\begin{lemma}\label{2keymissed}
Assume \textbf{H} and  let $\mu_2$, $\varpi$ and $\de=\delta(\mu_2,\varpi)$ be as in Proposition~\ref{estTDbis}.
  Let $0<d \le \de$, $\tau_0 \in \R$, and consider the trajectory
  $$\x(t,\tau_0; \Q^-(d,\tau_0)) \equiv x(t,\tau_{-1}; \vec{R}(D)) \qquad  \tau_{-1} \le t \le \tau_0$$
  where $D=D^b_2(d,\tau_0)$, $\vec{R}(D)$ is the point of $L^{\textrm{in}}$ such that
  $\|\vec{R}(D)\|=D\in ]0, \de_{-1}]$ where  $\de_{-1}=\de^{\sbwd_-\left( 1 + \frac{\mu_2 |\la_s^-|}{\la_u^-}\right)}$  and $\tau_{-1}=\tau_0
  -T^b_2(d,\tau_0)$.
  Then
  \begin{gather*}
  	\|\x(t,\tau_0; \Q^-(d,\tau))- \vec{y}_u(t-\tau_0)\| \le   \frac{2}{c} d^{\sbwd_- \left(1-\mu_2\frac{|\la_s^-|}{\la_u^-}\right)} z_s^-(t,
  \tau_{-1}) \\
  \le
  \frac{2k_1}{c} d^{\sbwd_- \left(1-\mu_2\frac{|\la_s^-|}{\la_u^-}\right)} \eu^{-\frac12\underline{\la}(t-\tau_{-1})}
  \end{gather*}
  whenever $\tau_{-1} \le t \le \tau_0$.

   Analogously let  $0<D \le \de_{-1} $, $\tau_{-1} \in
   \R$, and
   consider the trajectory
  $$\x(t,\tau_0; \Q^-(d,\tau_0)) \equiv x(t,\tau_{-1}; \vec{R}(D)) \qquad  \tau_{-1} \le t \le \tau_0$$
  where $d=D^{f}_3(D,\tau_{-1})$ and $\tau_0=\tau_{-1} +T^{f}_3(D,\tau_{-1})$.
  Then
  $$\|\x(t,\tau_{-1}; \vec{R}(D))- \vec{y}_u(t-\tau_0)\| \le   \frac{2}{c} D z_s^-(t, \tau_{-1}) \le  \frac{2k_1}{c} D
  \eu^{-\frac12\underline{\la}(t-\tau_{-1})}  $$
   whenever $\tau_{-1} \le t \le \tau_0$.
\end{lemma}

\subsection{Final step: estimate of time and space displacement}\label{SLinFinal}

Now we are ready to complete the proofs of Theorems \ref{key}, \ref{keymissed}.

We start with Theorem \ref{key}: the proof of the result follows immediately from combining Propositions  \ref{estloglog},
\ref{estTD},
\ref{lemmaconverse+},
\ref{estloglog+} in forward time and
Propositions  \ref{estloglog+}, \ref{estTDbis}, \ref{lemmaconverse-}, \ref{estloglog} in backward time.

However we proceed with the lengthy and tedious computation of the estimate of $\mu_0$, the upper bound in the errors
performed to evaluate time and space displacement. In some sense, the only relevant fact is that the value of
$\mu_0$ (that is the upper bound of $\mu$)  depends only on the eigenvalues $\la^{\pm}_{s}$, $\la^{\pm}_{u}$.

We make use of the following estimates
$$ \frac{\und{\la}}{2\ov{\la}} \le \underline{\sigma}\le   \overline{\sigma} \le  \frac{\ov{\la}}{2\und{\la}}, \quad
\qquad \frac{\und{\la}}{\ov{\la}^2} \le \underline{\Sigma}\le   \overline{\Sigma} \le  \frac{\ov{\la}}{ \und{\la}^2}.
$$

\begin{proof}[\textbf{Proof of Theorem \ref{key}}]
Let us set
\begin{equation}\label{def-cmu}
 c_{\mu}= c_T+c_d +\frac{1}{6}
\end{equation}
  where
\begin{equation}\label{cmu}
  c'_d = \frac{\overline{\lambda}}{\underline{\lambda}}   \, , \qquad
  c_d= 7    \frac{\overline{\lambda}^3}{ \underline{\lambda}^3 } \,,\qquad
  c_T = \frac{1}{6}+ \frac{1}{ 2\underline{\lambda}}.
\end{equation}
Notice that $c_{\mu} > c_d \ge 7$.
Assume that   $0<2 \mu_1  = \mu_2$ and let $0<\mu= c_{\mu} \mu_2< \mu_0$,
where $\mu_0$ is as in \eqref{mu0}.
We can choose $0<d \le \delta \le \varpi$ small enough so that
$$d^{\mu_1} \le \delta^{\mu_1} \le \varpi^{\mu_1} \le |\ln(\varpi)|^{-\tilde{C}}$$
where $\tilde{C}$ is a constant as in Proposition \ref{estloglog}.

Combining  Propositions \ref{estloglog}, \ref{estTD}  we find
the following estimates of the space displacement
$$ \|\PPP_{1/2}(d,\tau) \|= D^f_2(D^f_1(d,\tau), \tau+ T^f_1(d,\tau)), $$
\begin{equation}\label{d.half}
   d^{\sfwd_++\tilde{A}} \le \|\PPP_{1/2}(d,\tau) \| \le d^{\sfwd_++\tilde{B}}
\end{equation}
where, using $\frac{\la_u^+}{|\la_s^+|} \sfwd_+  \le \frac{ \overline{\lambda} }{ 2\underline{\lambda} }$, we find
\begin{equation}\label{shorten0}
  \tilde{B} = (1-\mu_1)\, \cdot \,   \sfwd_+ \left(1-\frac{\la_u^+}{|\la_s^+|}\mu_2\right)-   \sfwd_+ \ge - \frac{
  \overline{\lambda} }{
  2\underline{\lambda} }   (\mu_1+\mu_2
  ) \ge - c'_d \mu_2  \ge
   - \mu,
\end{equation}
\begin{equation}\label{shorten1}
\begin{gathered}
\tilde{A} =(1+\mu_1)\, \cdot \,  \sfwd_+ \left(1+\frac{\la_u^+}{|\la_s^+|}\mu_2 \right)-   \sfwd_+ \le  \frac{
\overline{\lambda} }{
2\underline{\lambda}
}(\mu_1+\mu_2
+\mu_1\mu_2) \\ \le   \frac{ \overline{\lambda} }{ 2\underline{\lambda} }   (2\mu_1+\mu_2) = c'_d \mu_2  \le
 \mu.
\end{gathered}
\end{equation}

So we have shown that
$$ d^{\sfwd_+ + \mu} \le \|\PPP_{1/2}(d,\tau) \| \le d^{\sfwd_+ - \mu}. $$

Let us set
 $$\tau^f_0(d,\tau)=\tau\, , \quad
 D^f_0(d,\tau)=d\,,\quad  \tau^f_{i}= T^f_i(D^f_{i-1}(d,\tau),\tau^f_{i-1}(d,\tau))+ \tau^f_{i-1}(d,\tau)$$ for $i=1,2,3,4$.

Now combining  Propositions \ref{estloglog}, \ref{estloglog+}, \ref{estTD} and \ref{lemmaconverse+}  we find
the following
   $$\DDD(\PPP_1(d,\tau), \vec{P}_u(\TTT_1(d,\tau)))=
D^f_4(D^f_3(D^f_2(D^f_1(d,\tau), \tau^f_1(d,\tau)),  \tau^f_2(d,\tau)), \tau^f_3(d,\tau)),
    $$
$$\DDD(\PPP_1(d,\tau), \vec{P}_u(\TTT_1(d,\tau)))=
D^f_4(D^f_3(\|\PPP_{1/2}(d,\tau) \|, \tau^2_f(d,\tau))    , \tau^f_3(d,\tau)),
    $$
$$ d^{\sfwd+A} \le \DDD(\PPP_1(d,\tau), \vec{P}_u(\TTT_1(d,\tau))) \le d^{\sfwd+B} $$
where,
using  \eqref{d.half} and $\frac{\sfwd_+}{ \sbwd_-}   \le \frac{ \overline{\lambda}^2 }{ \underline{\lambda}^2} $,
we find
\begin{equation}\label{muB}
\begin{gathered}
  B = (1-\mu_1)\, \cdot \,  \left(\sfwd_+- c'_d  \mu_2     \right)
  \, \cdot \,  \frac{1}{\sbwd_-} \left(1- 2\mu_2 \frac{|\la_s^-|}{\la_u^-}   \right)- \frac{\sfwd_+}{\sbwd_-}  \\
  \ge  - \mu_1\frac{\sfwd_+ }{\sbwd_- } - \mu_2 \frac{c'_d}{\sbwd_-}    -2 \mu_2 \frac{\sfwd_+|\la_s^-|}{\sbwd_- \la_u^-}\\
   \ge  -\mu_2 \left(\frac{1}{2}\frac{ \overline{\lambda}^2 }{ \underline{\lambda}^2}
   +\frac{2c'_d\overline{\lambda} }{ \underline{\lambda}}
   +\frac{2\overline{\lambda}^3 }{ \underline{\lambda}^3} \right)
   \ge -\mu_2\left(4+\frac{1}{2}\right)\frac{\overline{\lambda}^3 }{ \underline{\lambda}^3}
    > - c_d \mu_2 \ge - \mu,
\end{gathered}
\end{equation}

\begin{equation}\label{muA}
\begin{gathered}
  A
  =
   (1+\mu_1)\, \cdot \,  \left(\sfwd_++ c'_d  \mu_2     \right)
  \, \cdot \,  \frac{1}{\sbwd_-} \left(1+ 2\mu_2   \frac{|\la_s^-|}{\la_u^-} \right)- \frac{\sfwd_+}{\sbwd_-}  \\
  =   \frac{ \mu_1}{\sbwd_-} (\sfwd_++ c'_d \mu_2) \left(1+ 2\mu_2  \frac{|\la_s^-|}{\la_u^-} \right) + \frac{
  \mu_2
   c'_d}{\sbwd_-}
  \left(1+2 \mu_2
  \frac{|\la_s^-|}{\la_u^-} \right)
  + 2 \mu_2  \frac{ \sfwd_+}{\sbwd_-} \frac{|\la_s^-|}{\la_u^-}    \\
 =\mu_1\frac{\sfwd_{+}}{\sbwd_{-}}\left(1+2\mu_2\frac{|\la_s^-|}{\la_u^-}\right)
  	+c'_d\frac{\mu_2}{\sbwd_{-}}\left(1+2\mu_2\frac{|\la_s^-|}{\la_u^-}\right)(1+\mu_1)
  	+ 2 \mu_2  \frac{ \sfwd_+}{\sbwd_-} \frac{|\la_s^-|}{\la_u^-}\\
  \le\mu_1\frac{\overline{\lambda}^2 }{\underline{\lambda}^2}\left(1+\frac{1}{2}\frac{\overline{\lambda}}{ \underline{\lambda}}\right)
  +c'_d\mu_2\frac{2\overline{\lambda}}{ \underline{\lambda}}\left(1+\frac{1}{2}\frac{\overline{\lambda}}
  { \underline{\lambda}}\right)\left(1+\frac{1}{8}\right)
  +2\mu_2\frac{\overline{\lambda}^3}{ \underline{\lambda}^3}\\
 \le \mu_2\frac{\overline{\lambda}^3}{ \underline{\lambda}^3}
  \left(\frac{3}{2}+2\cdot\frac{3}{2}\cdot\frac{9}{8}+2\right)
  <7\mu_2\frac{\overline{\lambda}^3}{ \underline{\lambda}^3}
  =c_d \mu_2 \le \mu
\end{gathered}
\end{equation}
using $\mu_2\leq\frac{1}{4}$.
So, we have shown the first estimate of \eqref{D1T1}, i.e.
\begin{equation}\label{d.whole}
  d^{\sfwd +\mu} \le \DDD(\PPP_1(d,\tau), \vec{P}_u(\TTT_1(d,\tau))) \le d^{\sfwd -\mu} .
\end{equation}

Now we evaluate the space displacement backward in time: for this purpose we simply need to apply the previous analysis
to the modified system \eqref{eq.modbis}, i.e., to perform an inversion of time argument, analogous to the one detailed in \S \ref{S+Lin}.
This way from \eqref{d.half} and \eqref{d.whole} we get respectively
$$ d^{\sbwd_- + \mu} \le \|\PPP_{-1/2}(d,\tau) \| \le d^{\sbwd_- - \mu}, $$
  $$ d^{\sbwd + \mu} \le \DDD(\PPP_{-1}(d,\tau), \vec{P}_s(\TTT_{-1}(d,\tau))) \le d^{\sbwd -\mu} .$$

Now we consider the time displacement; again we start by following the trajectory forward in time.
Let us set
$$ \tau'_a + \tau''_a   \le  \TTT_{\frac{1}{2}}(d,\tau) -\tau\le  \tau'_b + \tau''_b,$$
$$ \tau'_a + \tau''_a+ \tau'''_a+ \tau^{iv}_a  \le  \TTT_1(d,\tau)- \tau \le  \tau'_b + \tau''_b+ \tau'''_b+ \tau^{iv}_b$$
where $\tau'_a$ and $\tau'_b$ are estimates from below and from above of the time spent by $\x(t,\tau; \Q_s(d,\tau))$
 to travel from $L^0$ to  $\tilde{S}^+$, $\tau''_a$ and $\tau''_b$ are estimates   of the time spent
 to travel from    $\tilde{S}^+$ to $L^{\inn}$, $\tau'''_a$ and $\tau'''_b$ are estimates   of the time spent
 to travel from    $L^{\inn}$ to $\tilde{S}^-$, $\tau^{iv}_a$ and $\tau^{iv}_b$ are estimates   of the time spent
 to travel from   $\tilde{S}^-$ to $L^0$.

 From Proposition \ref{estloglog}, possibly choosing  a smaller $\varpi>0$  we see that
\begin{equation}\label{est.tau'}
  0< \tau'_a  \le \tau'_b  \le  \frac{4}{\und{\la}} \ln(|\ln(\varpi)|) \le \frac{4}{\und{\la}}
  \ln(\varpi^{-\frac{\mu_1\und{\la}}{12}}) \le
  \frac{\mu_1}{3}   |\ln(\delta)|.
\end{equation}
 Analogously, from Proposition \ref{estloglog+} we see that
 $$
 0< \tau^{iv}_a \le   \tau^{iv}_b \le  \frac{\mu_1}{3} |\ln(\delta)|.$$
 Then combining  Propositions  \ref{estloglog} and \ref{estTD} we find
 $$\tau''_a \ge \sTfwd_+\left(1- \frac{\mu_2}{3}\right) (1-\mu_1)|\ln(d)| \ge \sTfwd_+\left(1- \frac{5\mu_2}{6}\right) |\ln(d)|\ge \sTfwd_+(1- \mu_2)
 |\ln(d)|,$$
 $$\tau''_b \le \sTfwd_+\left(1+ \frac{\mu_2}{3}\right) (1+\mu_1)|\ln(d)| $$
 $$
 = \sTfwd_+\left(1+ \frac{5\mu_2+\mu_2^2}{6}\right) |\ln(d)|   \le \sTfwd_+(1+ \mu_2) |\ln(d)|. $$

Summing up we have found
 \begin{equation*}
   \begin{split}
      \TTT_{\frac{1}{2}}(d,\tau)-\tau & \le
       \left(\sTfwd_+ + \frac{\mu_1}{3}+  \sTfwd_+ \mu_2  \right) |\ln(d)| \le
     \left[\sTfwd_+ +  \left(\frac{1}{6}+ \frac{1  }{2 \underline{\lambda}} \right) \mu_2 \right]  |\ln(d)| \\&
      = \left(\sTfwd_+ +    \michal{c_T} \mu_2  \right)  |\ln(d)| \le  (\sTfwd_+ + \mu) |\ln(d)|, \\
      \TTT_{\frac{1}{2}}(d,\tau)-\tau & \ge
       \left(\sTfwd_+ -  \sTfwd_+ \mu_2  \right) |\ln(d)| \ge
     \left(\sTfwd_+ - \frac{ \mu_2   }{2 \underline{\lambda}}  \right)  |\ln(d)| \\&
      \ge \left(\sTfwd_+ -    \michal{c_T} \mu_2  \right)  |\ln(d)| \ge  (\sTfwd_+ -\mu) |\ln(d)|.
   \end{split}
 \end{equation*}
Thus we find
\begin{equation}\label{TT.half}
  (\sTfwd_+ - \mu) |\ln(d)| \le \TTT_{\frac{1}{2}}(d,\tau) -\tau \le    (\sTfwd_+ + \mu) |\ln(d)|.
\end{equation}
  Finally  combining  \eqref{d.half}, \eqref{shorten1}, \eqref{shorten0} and  Proposition  \ref{lemmaconverse+}   we find
\begin{equation*}
	\begin{split}
		&   \tau'''_b \le \frac{1}{\la_u^-}\left[ 1+ \mu_2\left( 2\frac{|\la_s^-|}{\la_u^-} +\frac{1}{2} \right)
		\right] |\ln (d^{\sfwd_+
			+\tilde{A}})|
		\\ & \le \frac{1}{\la_u^-}\left[1 + \mu_2\left( 2\frac{\ov{\la}}{\und{\la}} +\frac{1}{2} \right)\right] (\sfwd_+  + c_d'
		\mu_2)
		|\ln(d)|
		\\  & \le   \left[ \frac{\sfwd_+}{ \la_u^-} + \frac{\mu_2}{\la_u^-} \left(  \sfwd_{+}\frac{5}{2}\frac{\ov{\la}}{\und{\la}} +
c'_d\left(1+\mu_2\frac{5}{2}
		\frac{\ov{\la}}{\und{\la}}\right)\right)\right] |\ln(d)| \\
		& \le \left[ \frac{\sfwd_+}{ \la_u^-} + \frac{\mu_2}{\la_u^-} \left(
		\frac{5}{4}\frac{\ov{\la}^2}{\und{\la}^2} + \frac{13}{8}		\frac{\ov{\la}^2}{\und{\la}^2}\right)\right] |\ln(d)| \\
		& < \left( \frac{\sfwd_+}{ \la_u^-} + 3\frac{\mu_2}{\la_u^-} \frac{\ov{\la}^2}{\und{\la}^2} \right) |\ln(d)| \le \left(
\frac{\sfwd_+}{\la_u^-}+ c_d \mu_2 \right)|\ln(d)|.
	\end{split}
\end{equation*}

  Analogously we find
  $$\tau'''_a \ge    \left( \frac{\sfwd_+}{\la_u^-}- c_d \mu_2 \right)|\ln(d)|.$$

Observe now that $\sTfwd= \sTfwd_+ + \frac{\sfwd_+}{\la_u^-}$, then,
  summing up all the estimates, we find

 \begin{equation*}
    \begin{split}
    & \TTT_1(d,\tau) -\tau \le  \tau'_b+\tau''_b+\tau'''_b+ \tau^{iv}_b \\
      &\le   (\sTfwd_+  + \michal{c_T} \mu_2) |\ln(d)| + \left( \frac{\sfwd_+}{\la_u^- }+ c_d \mu_2 \right)|\ln(d)| + \frac{\mu_1}{3}
      |\ln(\delta)|
      \le \\
    & \le \left[\sTfwd + \left(\michal{c_T}+c_d  + \frac{1}{6}  \right) \mu_2 \right] |\ln(d)| =
    (\sTfwd+ c_{\mu} \mu_2)|\ln(d)| \le
     (\sTfwd+ \mu)|\ln(d)|.
\end{split}
\end{equation*}
Analogously, using the fact that $\TTT_1(d,\tau) -\tau \ge  \tau'_a+\tau''_a+\tau'''_a+ \tau^{iv}_a$ and the previous estimates, we find
\begin{equation}\label{TT.whole}
  (\sTfwd - \mu) |\ln(d)| \le \TTT_{1}(d,\tau) -\tau \le    (\sTfwd + \mu) |\ln(d)|.
\end{equation}

 Then, applying again an inversion of time argument, from \eqref{TT.half} and \eqref{TT.whole} respectively we find
$$
    (\sTbwd_- - \mu) |\ln(d)| \le \tau-  \TTT_{-1/2}(d,\tau)    \le  (\sTbwd_- + \mu) |\ln(d)|,
$$
$$ (\sTbwd - \mu) |\ln(d)| \le \tau- \TTT_{-1}(d,\tau) \le   (\sTbwd + \mu) |\ln(d)| $$
and the proof is completed.
\end{proof}

\begin{remark}\label{totrash}
  Modifying slightly the proof of Proposition \ref{estTD} we might show that we can find a constant $K_2^\mu>1$
  (independent of all the small parameters) such that we can choose
  $\mu_2= K_2^\mu \ep$; then we could replace \eqref{est.T2} and \eqref{est.D-d} by the following
  \begin{equation}\label{e.trash1}
	\begin{gathered}
		\sTfwd_+\left(1-K_2^\mu \ep \right) |\ln(d)|- |\ln(\varpi)|
		\le T^f_2(d,\tau)
		\le  \sTfwd_+\left(1+K_2^\mu \ep\right)  |\ln(d)| \\
	|\ln(\varpi)|^{-2}   d^{\sfwd_+  \left( 1+ K_2^\mu \ep \right) }
	\le   D^f_2(d,\tau)
	\le |\ln(\varpi)|^2 d^{\sfwd_+\left( 1- K_2^\mu \ep   \right) }.
	\end{gathered}
\end{equation}
  Then, we conjecture that with a proper choice of a fixed $\varpi=\varpi_0>0$ satisfying \eqref{est.varpi1}, eventually we would obtain, as we said in Remark \ref{oneparam},
  	\begin{equation}\label{e.trash2}
	\begin{split}
	 	 C^{-1}  d^{\sfwd+ K^\mu \ep} & \le \dist (\PPP_1(d,\tau),\P_u(\TTT_1(d,\tau)))
 \le C d^{\sfwd+ K^\mu \ep}, \\
	 (\sTfwd - K^\mu \ep) |\ln(d)| -C& \le  (\TTT_1(d,\tau)-\tau)    \le  (\sTfwd + K^\mu \ep)  |\ln(d)|+C
	\end{split}
	\end{equation}
for suitable $K^\mu>K_2^\mu >0$ and $C>1$. Further we think that if $\bs{g_x}(t,\vec{0},\ep)=\bs{0}$, so that for the exponential dichotomy estimates in \eqref{defzused}
we can use $\la^{\pm}_s$ and $\la^{\pm}_u$ as exponents, we can even set $K^\mu=0$ in \eqref{e.trash2}.
\end{remark}

Now we prepare the proof of Theorem \ref{keymissed}.

 Notice that in \michal{Theorem} \ref{keymissed} we need the distance between $\x(t,\tau; \Q_s(d,\tau))$
 and $\x(t,\tau; \P_s(\tau))$ to be small but in fact in \S \ref{S-Lin} we measure the distance between
 $\x(t,\tau; \Q_s(d,\tau))$ and $\x(t, T_1^f(d,\tau)+\tau;  \vec{\pi}_s(T_1^f(d,\tau)+\tau))$. We need Lemmas \ref{junction} and
 \ref{bothstable}
 below to measure this mismatch.

With the notation of Lemma \ref{defTf-} we set
 $$\vec{\ell}_s(d,\tau):= \P^+_f(d,\tau) - \x(T^f_1(d,\tau)+\tau, \tau ; \P_s(\tau)) ,$$
 $$\vec{\upsilon}_s(d,\tau) := \PPs(T^f_1(d,\tau)+\tau)  - \x(T^f_1(d,\tau)+\tau, \tau ; \P_s(\tau));$$
then we have the following.
 \begin{lemma}\label{junction}
   Let the assumptions of Proposition \ref{estloglog} be satisfied;  recalling
   that $0<d \le \de \le \varpi$ we find
   \begin{equation}\label{est.ell}
   \begin{split}
  & \left\| \frac{\partial }{\partial d} \vec{\ell}_s(d,\tau) \lfloor_{d=0} \right\|  \le  |\ln(\varpi)|^{8 \ov{\la}/\und{\la}}  \, , \\
   &  \| \vec{\ell}_s(d,\tau)\| \le 2 |\ln(\varpi)|^{8 \ov{\la}/\und{\la}} \; d \le 2 |\ln(\de)|^{8 \ov{\la}/\und{\la}} \; d \le
   d\delta^{-\mu_1/2}.
    \end{split}
 \end{equation}
    Further
    \begin{equation}\label{est.needed}
      \begin{split}
          & \left\|  \frac{\partial }{\partial d} \vec{\upsilon}_s(d,\tau)\lfloor_{d=0} \right\|
    \le   \frac{2 k_2 \ov{\la}}{\und{\la}}     | \ln(\varpi)|^{ 8 \ov{\la}/{\und{\la}}}, \\
           & \|  \vec{\upsilon}_s(d,\tau)\| \le   \frac{4 k_2 \ov{\la}}{\und{\la}}    | \ln(\varpi)|^{8 \ov{\la}/{\und{\la}}}
           \; d
  \le \frac{4 k_2 \ov{\la}}{\und{\la}}   | \ln(\de)|^{ 8 \ov{\la}/{\und{\la}}} \; d  \le   d\delta^{-\mu_1/2}.
   \end{split}
 \end{equation}
 \end{lemma}
 \begin{proof}
  Notice that $\x(T^f_1(0,\tau)+\tau, \tau ; \P_s(\tau)) \in [\tilde{W}^s(T^f_1(0,\tau)+\tau) \cap \tilde{S}^+]$
 and $\tilde{W}^s(T^f_1(0,\tau)+\tau) \cap \tilde{S}^+= \{\PPs(T^f_1(0,\tau)+\tau)\}= \{\P^+_f(0,\tau)\}$  for all $\ep$. So we use an
 expansion as in \eqref{Df1} and in the rest of the proof we consider $\ep=0$.
 Hence $\vec{\ell}_s(0,\tau)= \vec{0}= \vec{\upsilon}_s(0,\tau)$.

   Differentiating $\vec{\ell}_s(d,\tau)$ with respect to $d$ we get
   \begin{equation}\label{elldiff1}
   \begin{split}
        &  \frac{\partial }{\partial d} \vec{\ell}_s(d,\tau) \lfloor_{d=0}= \frac{\partial \x}{\partial \P}  (T^f_1(d,\tau)+\tau, \tau ;
        \Q_s(d,\tau))\lfloor_{d=0} \frac{\partial \Q_s(0,\tau)}{\partial d}  \\
        & {}+ [\dot{\x}(T^f_1(d,\tau)+\tau, \tau ; \Q_s(d,\tau)) - \dot{\x}(T^f_1(d,\tau)+\tau, \tau ; \P_s(\tau))]\lfloor_{d=0} T'_d(\tau)
        \\
        &  = \bs{X} (T^f_1(0,\tau)) \vec{v},
   \end{split}
   \end{equation}
  where $\bs{X}(t)$ is the fundamental matrix of \eqref{linlin}.
   So  the first inequality in \eqref{est.ell} follows from \eqref{newmaybe1}; then the second inequality in \eqref{est.ell} follows
    immediately.

    Now notice that if $\ep=0$, then $\vec{\pi}_s(\tau)\equiv  \vec{\pi}_{s,\Gamma}= \bs{\Gamma} \cap \tilde{S}^+$ for any $\tau
   \in\R$,
    so that   $\frac{\partial}{\partial\tau}\vec{\pi}_s(\tau) \equiv \vec{0}$. Hence
     \begin{equation}\label{elldiff2}
   \begin{split}
        &  \frac{\partial }{\partial d} \vec{\upsilon}_s(0,\tau)= -T'_d(\tau) \f^+(\ga(\tilde{T}_1^f)) ,
   \end{split}
   \end{equation}
     then the first \michal{estimate} in  \eqref{est.needed} follows from \eqref{whence} and \eqref{Tdiff1}.
 Finally the second inequality in \eqref{est.needed} follows immediately from the first and the fact that $\|\vec{\upsilon}_s(0,\tau)\|=0$.
 \end{proof}

 We denote by   $\tau^f_1= \tau^f_1(d,\tau)= \tau + T^f_1(d,\tau)$  and
 $\vec{\pi}_{s,1}(\tau^f_1) =\x(\tau^f_1,\tau; \vec{P}_s(\tau))$.
 Notice that $\vec{\pi}_{s,1}(\tau^f_1) \in \tilde{W}^s(\tau^f_1)$
  but it is not necessarily in  $\tilde{S}^+$ and   recall that by construction
$\{\vec{\pi}_s(\tau^f_1)\}=  \tilde{W}^s(\tau^f_1) \cap \tilde{S}^+$, cf.~\eqref{S-S+}.
Let us set
$$\vec{\upsilon}_s(\theta;d,\tau) := \x(\theta+ \tau^f_1, \tau^f_1 ; \vec{\pi}_{s,1} (\tau^f_1))  - \x(\theta+ \tau^f_1, \tau^f_1 ;
\vec{\pi}_{s} (\tau^f_1)),$$
and recall that $\f^{\pm}$ and $\g$ are $C^r$ and  $\alpha= \min \{ r-1 ;1 \}$.
Then we have the following estimate.
\begin{lemma}\label{bothstable}
Let $0<d \le \delta$, then
\begin{equation}\label{estupsilon}
\|\vec{\upsilon}_s(\theta;d,\tau)\| \le d^{1-\mu_1} \eu^{ \frac{\la_s^+ \theta}{1+\al/2} }
\end{equation}
 for any $\theta \ge 0$.
\end{lemma}
 \begin{proof}
Observe that both $\x(t, \tau^f_1 ; \vec{\pi}_{s,1} (\tau^f_1))$ and $\x(t, \tau^f_1 ; \vec{\pi}_{s} (\tau^f_1))$
belong to $\tilde{W}^s(t)$ for any $t \ge \tau^f_1$, and set $\y_s(\theta)=\x(\theta+ \tau^f_1, \tau^f_1 ; \vec{\pi}_{s} (\tau^f_1))$ for
short.
From Lemma~\ref{junction} we know that
\begin{equation}\label{hzero}
  \| \vec{\upsilon}_s(0;d,\tau) \|=  \| \vec{\upsilon}_s(d,\tau) \| =  \|\vec{\pi}_{s,1}(\tau^f_1)- \vec{\pi}_{s}(\tau^f_1)\| \le
  d^{1-\mu_1/2}.
\end{equation}
  To prove \eqref{estupsilon}  we use exponential dichotomy and a fixed point argument in an exponentially weighted space.
Notice that $\vec{\upsilon}_s(\theta;d,\tau)$ is a solution of
\begin{equation}\label{newnonlin}
  \dot{\vec{h}}(\theta)= \bs{F_x^+}(\theta+\tau^f_1,\y_s(\theta), \ep) \vec{h}(\theta) + \vec{N} (\theta+\tau^f_1, \vec{h}(\theta), \ep)
\end{equation}
where
\begin{equation*}
  \begin{split}
      \vec{N}(\theta+\tau^f_1, \vec{h}(\theta), \ep)= & \vec{F}^+(\theta+\tau^f_1,\y_s(\theta)+\vec{h}(\theta), \ep)-
 \vec{F}^+(\theta+\tau^f_1,\y_s(\theta), \ep) \\
       & - \bs{F_x^+}(\theta+\tau^f_1,\y_s(\theta), \ep) \vec{h}(\theta).
  \end{split}
\end{equation*}
Recall that $\vec{F}^+ \in C^r$, so there is $C_{\alpha}>0$ independent of $\ep$, $t$,   such that
\begin{equation*}
  \begin{split}
  &   \| \vec{N}(t, \vec{h} , \ep) \| \le C_{\alpha}
      \|\vec{h} \|^{1+\alpha}, \\
  &     \| \vec{N}(t, \vec{h}_1 , \ep)- \vec{N}(t, \vec{h}_2 , \ep)  \| \le C_{\alpha}
      [\max\{\|\vec{h}_1 \|, \|\vec{h}_2 \|\}]^{\alpha} \|\vec{h}_1 -\vec{h}_2 \|
  \end{split}
\end{equation*}
(cf.~$\vec{R}_2^+(\theta,\vec{h})$ in \eqref{est:ball} and \eqref{NN}).

Observe further, that $\|\y_s(\theta)- \ga(\tilde{T}_1^f+\theta)\|=O(\ep)$ uniformly for $\theta \ge 0$, so
\begin{equation}\label{newlin}
  \dot{\vec{h}}(\theta)= \bs{F_x^+}(\theta+\tau^f_1,\y_s(\theta), \ep) \vec{h}(\theta)
\end{equation}
 admits exponential dichotomy in $[0,+\infty[$,
and  the exponents are $\ep$-close to the respective ones of $\dot{\x}= \bs{f_x^+}(\ga(\tilde{T}_1^f+\theta)) \x$, i.e., to
 $\la_s^+$, $\la_u^+$. Further the optimal dichotomy constant $\bar{k}_1$ is $|\ln(\varpi)|^{-1}$ close to $k_1$ since
 $\bs{f_x^+}(\ga(\tilde{T}_1^f+\theta))$ is $|\ln(\varpi)|^{-1}$ close to $\bs{f_x^+}(\vec{0})$, so we can assume
 $\bar{k}_1= 2k_1$.

  So let $\bs{X_h}(\theta)$ and $\bs{P_h}$ be the fundamental matrix and the projection (of the exponential dichotomy)
of \eqref{newlin}. We can assume that
\begin{equation}\label{dicesth}
  \begin{split}
      & \|\bs{X_h}(t)\bs{P_h}[\bs{X_h}(s)]^{-1} \vec{\xi}\| \le 2 k_1\eu^{\tfrac{\la_s^+}{1+\al/2} (t-s)} \qquad \text{if $t>s>0$}, \\
       & \|\bs{X_h}(t) [\I- \bs{P_h}] [\bs{X_h}(s)]^{-1} \vec{\xi}\| \le 2 k_1\eu^{\tfrac{\la_u^+}{1+\al/2} (t-s)} \qquad \text{if
       $s>t>0$}.
  \end{split}
\end{equation}
Further, since  $\vec{\upsilon}_s(\theta;d,\tau)$ converges to $0$ exponentially, for $\theta \to +\infty$, we see that it is on the stable
manifold of
\eqref{newnonlin}.
In particular $ \vec{\upsilon}_s(\theta;d,\tau)\eu^{\tfrac{|\la_s^+|}{1+\al/2} \theta}$ is bounded for $\theta \ge 0$.

    In fact $\vec{\upsilon}_s(\theta;d,\tau)$ is the unique bounded solution of \eqref{newnonlin}
    such that
 \begin{equation}\label{newnonlindata}
  \bs{P_h}  \dot{\vec{h}}(0)=  \bs{P_h}  [\vec{\pi}_{s,1} (\tau^f_1)-\vec{\pi}_{s} (\tau^f_1)] =: \vec{c}_0.
\end{equation}

So let $\mathcal{H}$ be the space of \michal{all} continuous functions from $[0,+\infty[$ to $\R^2$ endowed with the norm
$\|\y\|_{\mathcal{H}}= \sup_{\theta \ge 0} \{\|\y(\theta)\|\textrm{exp}[\tfrac{|\la_s^+|}{1+\al/2} \theta]\}$.

Let us consider the operator $\mathcal{T}: \mathcal{H} \to \mathcal{H}$ defined by
\begin{equation}\label{fixpoint}
\begin{split}
     & \mathcal{T}(\y)(\theta)= \bs{X_h}(\theta) \bs{P_h}\vec{c}_0 + \int_{0}^{\theta} \bs{X_h}(\theta)\bs{P_h} [\bs{X_h}(s)]^{-1}
     \vec{N}(s+\tau^f_1, \y(s), \ep) ds \\
     &  - \int_{\theta}^{+\infty} \bs{X_h}(\theta)[\I-\bs{P_h}][\bs{X_h}(s)]^{-1} \vec{N}(s+\tau^f_1, \y(s), \ep) ds.
\end{split}
\end{equation}
Notice that $\y(\theta)$ is a fixed point of $\mathcal{T}$ if and only if it is a bounded  solution of \eqref{newnonlin},
\eqref{newnonlindata}.
We aim to show that $\mathcal{T}$ maps the ball of radius $R= 6k_1d^{1-\mu_1/2}$ of $\mathcal{H}$ in itself and that it is a contraction.
Then it follows that  $\vec{\upsilon}_s(\theta;d,\tau)$ is the unique fixed point of $\mathcal{H}$ in this ball, and that
$\|\vec{\upsilon}_s(\theta;d,\tau)\|_{\mathcal{H}} \le R$.

In fact from \eqref{hzero} and \eqref{dicesth} we find
\begin{equation}\label{estuno}
  \|\bs{X_h}(\theta) \bs{P_h}\vec{c}_0 \| \le 2k_1  d^{1-\mu_1/2}\eu^{\tfrac{\la_s^+}{1+\al/2}\theta}= \frac{R}{3}
  \eu^{\tfrac{\la_s^+}{1+\al/2} \theta}  .
\end{equation}
Further, possibly choosing a smaller $ \delta>d$, and consequently a smaller $R$, we can assume
$R \le \left(\frac{\al |\la_s^+|}{6 k_1 C_{\al} (1+\al/2)} \right)^{1/\alpha}$, so that
whenever $\| \vec{y} \|_{\mathcal{H}} \le R$ we find
\begin{equation}\label{estdue}
\begin{split}
    &   \left\|\int_{0}^{\theta} \bs{X_h}(\theta)\bs{P_h} [\bs{X_h}(s)]^{-1}
     \vec{N}(s+\tau^f_1, \y(s), \ep) ds  \right\| \\
     &\le 2k_1 C_\al R^{1+\al} \int_{0}^{\theta} \eu^{\tfrac{\la_s^+}{1+\al/2}(\theta-s)}
     \cdot \eu^{\tfrac{\la_s^+}{1+\al/2}(1+\al)s} ds\\
     & \le \frac{2k_1 C_\al (1+\al/2)}{\al |\la_s^+|}  R^{1+\al}\eu^{\tfrac{\la_s^+}{1+\al/2}\theta}
     \le \frac{R}{3} \eu^{\tfrac{\la_s^+}{1+\al/2} \theta}  ;
\end{split}
\end{equation}
and finally
\begin{equation}\label{esttre}
\begin{split}
&\left\|\int_{\theta}^{+\infty} \bs{X_h}(\theta)[\I-\bs{P_h}][\bs{X_h}(s)]^{-1} \vec{N}(s+\tau^f_1, \y(s), \ep) ds\right\| \\
&  \le 2k_1 C_\al R^{1+\al} \int_{\theta}^{+\infty}   \eu^{\tfrac{\la_u^+}{1+\al/2}(\theta-s)} \cdot
\eu^{\tfrac{\la_s^+}{1+\al/2}(1+\al) s} ds\\
  &   \le     \frac{2k_1 C_\al (1+\al/2)}{\la^{+}_u+ (1+\al) |\la_s^+|}  R^{1+\al}\eu^{\tfrac{\la_s^+}{1+\al/2}(1+\al)\theta}
  \le
     \frac{R}{3} \eu^{\tfrac{\la_s^+}{1+\al/2} \theta}   .
\end{split}
\end{equation}
So from \eqref{estuno}, \eqref{estdue} and \eqref{esttre} we conclude that $\mathcal{T}$ maps the ball of radius $R$ in itself.
With a similar argument we show that $\mathcal{T}$ is a contraction in $R$ and we conclude
that $\|\vec{\upsilon}_s(\theta;d,\tau)\|_{\mathcal{H}} \le R$ so
$$\|\vec{\upsilon}_s(\theta;d,\tau)\|_{\mathcal{H}} \le 6 k_1 d^{1-\mu_1/2} \le d^{1-\mu_1}$$
if $0<d \le \delta$ is small enough; thus
\eqref{estupsilon} is proved.
\end{proof}

Analogously, let us set
 $$\vec{\ell}_u(d,\tau):= \P_b^-(d,\tau) - \x(\tau-T^b_1(d,\tau), \tau ; \P_u(\tau)) ,$$
 $$\vec{\upsilon}_u(d,\tau) := \PPu(\tau-T^b_1(d,\tau))  - \x(\tau-T^b_1(d,\tau), \tau ; \P_u(\tau));$$
then, as in Lemma \ref{junction}, we have the following.
 \begin{lemma}\label{junction+}
   Let the assumptions of Proposition \ref{estloglog+} be satisfied,
  recalling
   that $0<d \le \de \le \varpi$ we find
   \begin{equation}\label{est.ell.bw}
   \begin{split}
  & \left\| \frac{\partial }{\partial d} \vec{\ell}_u(d,\tau)\lfloor_{d=0} \right\|   \le  |\ln(\varpi)|^{8 \ov{\la}/\und{\la}}  \, , \\
   &  \| \vec{\ell}_u(d,\tau)\| \le 2 |\ln(\varpi)|^{8 \ov{\la}/\und{\la}} \; d \le 2 |\ln(\de)|^{8 \ov{\la}/\und{\la}} \; d \le
   d\delta^{-\mu_1/2}.
    \end{split}
 \end{equation}
  Further
      \begin{equation}\label{est.needed.bw}
      \begin{split}
          & \left\|  \frac{\partial }{\partial d} \vec{\upsilon}_u(d,\tau)\lfloor_{d=0} \right\|
    \le   | \ln(\varpi)|^{8 \ov{\la}/{\und{\la}}}, \\
           & \|  \vec{\upsilon}_u(0,\tau)\|
    \le   4 k_2 \frac{\ov{\la}}{\und{\la}}   | \ln(\varpi)|^{8 \ov{\la}/{\und{\la}}} \; d
  \le  4 k_2 \frac{\ov{\la}}{\und{\la}}   | \ln(\de)|^{8 \ov{\la}/{\und{\la}}} \; d  \le   d\delta^{-\mu_1/2}.
   \end{split}
 \end{equation}
 \end{lemma}

We denote by  $\tau^b_1= \tau^b_1(d,\tau)= \tau - T^b_1(d,\tau)$  and
 $\vec{\pi}_{u,1}(\tau^b_1) =\x(\tau^b_1,\tau; \vec{P}_u(\tau))$.
 Again,  note that $\vec{\pi}_{u,1}(\tau^b_1) \in \tilde{W}^u(\tau^b_1)$
   but it is not necessarily in  $\tilde{S}^-$ while
$\{\vec{\pi}_u(\tau^b_1)\} = \tilde{W}^u(\tau^b_1)\cap \tilde{S}^-$ by construction, cf.~\eqref{S+}.
Let us set
$$\vec{\upsilon}_u(\theta;d,\tau) := \x(\theta+ \tau^b_1, \tau^b_1 ; \vec{\pi}_{u,1} (\tau^b_1))  - \x(\theta+ \tau^b_1, \tau^b_1 ;
\vec{\pi}_{u} (\tau^b_1)).$$

Then, as in Lemma \ref{bothstable} we have the following estimate.
\begin{lemma} \label{bothstable+}
Let $0<d \le \delta$ then
\begin{equation}\label{estupsilon+}
\|\vec{\upsilon}_u(\theta;d,\tau)\| \le d^{1-\mu_1} \eu^{\frac{\la_u^- \theta}{1+\al/2}}
\end{equation}
 for any $\theta \le 0$.
\end{lemma}

\begin{proof}[\textbf{Proof of Theorem \ref{keymissed}}]
Let us fix  $\mu_2=2 \mu_1$, $\mu= c_{\mu} \mu_2 \le \mu_0$;    then fix $\varpi$ and $\delta$ so that Lemmas \ref{0keymissed},
\ref{1keymissed} and  \ref{junction+}  hold
true,
and finally let $0<d \le \delta$ and $\tau \in \R$.
We start from \eqref{keymissed.es-}; we denote by $\tau_1^f= \tau_1^f(d,\tau)= \tau + T^f_1(d,\tau)$, $\vec{\pi}_{s,1} (\tau_1^f)
=\x(\tau_1^f,\tau; \vec{P}_s(\tau))$. Notice that    $\PPs(\tau_1^f)= \tilde{W}^s(\tau) \cap \tilde{S}^+$, see \eqref{S-S+}, while
$\vec{\pi}_{s,1}(\tau_1^f) \in \tilde{W}^s(\tau_1^f)$, but
it might not be in $\tilde{S}^+$.
Let us set, as in \S \ref{SLinFinal}, $\y_s(\theta):= \x(\theta+\tau_1^f, \tau_1^f; \vec{\pi}_s(\tau_1^f))$ and
$$\y_{s,1}(\theta):= \x(\theta+\tau_1^f, \tau_1^f; \vec{\pi}_{s,1}(\tau_1^f))  \equiv \x(\theta+ \tau_1^f,\tau; \vec{P}_s(\tau)).$$

 From Lemmas \ref{junction+} and \ref{1keymissed}, setting   $t=\theta+\tau_1^f$  we see that
\begin{equation}\label{end}
\begin{split}
& \|\x(t, \tau; \Q_s(d,\tau))-  \x(t, \tau; \P_s(\tau)) \| \le \|\x(\theta+\tau_1^f, \tau; \Q_s(d,\tau)) -\y_s(\theta)\| + \\
&+
\|\y_{s,1}(\theta)-\y_s(\theta)\| \le
  \frac{2 k_1}{c} d^{\sfwd_+ \left(1-\mu_2 \frac{\la_u^+}{|\la_s^+|}\right)} +
  d^{1-\mu_1}
\le  d^{\sfwd_+ -\mu}
\end{split}
\end{equation}
for any $\tau_1^f \le t \le \TTT_{\frac{1}{2}}(d,\tau)$,
since $\sfwd_+<1$ and   $\mu= c_{\mu} \mu_2$ (see the proof of Theorem \ref{key}).

Further, from Lemma \ref{0keymissed} we find
\begin{equation*}
\|\x(t, \tau; \Q_s(d,\tau))-  \x(t, \tau; \P_s(\tau)) \| \le d\delta^{-\mu_1} \le  d^{\sfwd_+ -\mu}
\end{equation*}
for any $\tau \le t \le \tau_1^f$. So \eqref{keymissed.es-} is proved.

The proofs of \eqref{keymissed.es+},
 \eqref{keymissed.eu+} and
 \eqref{keymissed.eu-} are similar and are omitted.
\end{proof}

\section{Some remarks and future developments}\label{s.open}
In this section we spend a few more words concerning some projects we wish to develop using the results of Theorems \ref{key} and \ref{keymissed}, as we said in the Introduction.

Firstly, it is well known that, if $\la_s+\la_u<0$ and $\ep=0$, in the framework of Scenario 1 the set $\bs{\Gamma}$ of \eqref{eq-smooth}  is asymptotically stable from inside
(stable from outside in Scenario 2). We think that the results of this paper could be of use in exploring the possibility to construct an
integral manifold which will be asymptotically stable from inside, when we have just one zero of the Melnikov function $\MM(\tau)$, see \eqref{melni-disc}. Notice that the presence of chaotic phenomena prevents the
possibility of asymptotic stability  if we have infinitely many zeros of $\MM(\tau)$.

Secondly,   we plan to use Theorems \ref{key} and \ref{keymissed} to study the possibility to establish a  sub-harmonic  Melnikov theory.
More precisely we think that if the function $\MM(\tau)$ has a unique zero at $\tau=\tau_0$ then we should find two different type of results:
if $\MM'(\tau_0)>0$ then we expect that the homoclinic trajectory $\x_b(t,\ep)$ bifurcating from $\ga(t)$ will be unique, while if
$\MM'(\tau_0)<0$  then we think there should be a monotone decreasing sequence of $\ep_k>0$ such that if $0<\ep \le \ep_k$ then there are
$k$ homoclinic trajectories, each of them performing exactly $j$ loops (i.e., passing $O(\ep)$ close to $\ga(0)$ exactly $j$ times), where $j =1, \ldots ,k$.
In fact we expect that this will be the case if $\MM(\tau)$ does not converge to $0$ when $\tau$ converges either to $+\infty$ or to $-\infty$,
while some extra conditions on the speed of convergence to $0$ is needed if $\MM(\tau) \to 0$.
This dynamical result is suggested also by   the bubble-tower phenomena appearing in quasi-linear elliptic equations, see e.g.\ \cite{CL1, DFS1, DFS2}.

 Thirdly,  we believe that  Theorems \ref{key} and \ref{keymissed} could be used to find \emph{safe regions} in systems characterized by Melnikov chaos, i.e., chaos generated by the
 non-autonomous perturbation of an autonomous homoclinic trajectory. More precisely we aim to find subsets
of initial conditions close to  $\bs{\Gamma}$,
which will remain far from the chaotic set for a long, possibly infinite time.
Roughly speaking we think we can use these ideas to give some conditions sufficient to  answer the question:\\ \indent
\emph{How far should I stay from $\bs{\Gamma}$ to be sure not to finish in a chaotic pattern?}
\\
 We think this kind of result
should be quite useful from an engineering point of view.

 Finally, as we said in the Introduction,
 Melnikov theory has been extended to the context  of piecewise smooth systems \eqref{eq-disc} assuming that $\vec{0} \in \Om^0$ in \cite{CaFr}, where we proved that the classical Melnikov condition is enough to ensure the persistence
to perturbation of the homoclinic trajectory. However quite unexpectedly, it does not guarantee the persistence of chaos if $\g$ is periodic in $t$, as it happens in
a smooth setting, see e.g.\ \cite{Pa84} but also in a non-smooth setting if the origin does not belong to the discontinuity surface $\Om^0$, see e.g.\ \cite{BF10, BF11, BF12}. In fact in \cite{FrPo} we have found a big class of counterexamples to the presence of chaos: roughly speaking a geometrical obstruction
forbids chaotic phenomena whenever we have sliding close to the origin.

In a forthcoming paper we plan to show that, if such a  geometrical obstruction is removed then the usual Melnikov conditions guarantee chaos as in the smooth setting,
and Theorems \ref{key} and \ref{keymissed} of the present paper are essential in this project.
We also think that these kinds of results should give us a better insight of what happens close to $\bs{\Gamma}$ in a smooth setting, possibly allowing us to give some
results concerning the size and the position of  the Cantor-like   a set giving rise to chaos.

\appendix
\section{Sketch of the proofs of Lemmas \ref{lemma0s} and  \ref{lemma0u}}\label{A.technic}

This section is devoted to the construction of the curves
$Z^{\textrm{fwd,in}}$, $Z^{\textrm{fwd,out}}$, $Z^{\textrm{bwd,in}}$, $Z^{\textrm{bwd,out}}$, i.e., to
 the proofs of Lemmas \ref{lemma0s} and  \ref{lemma0u},
which are adaption of the argument in \cite[\S 6.2]{FrPo}, see in particular
\cite[Lemmas 6.4, 6.7]{FrPo} and \cite[\S 6.2.2]{FrPo}.
We invite the reader to compare the references with Figures \ref{fig-Zui}, \ref{fig-Zuo}.

\mat{In the whole section we assume the hypotheses of Lemmas \ref{lemma0s} and  \ref{lemma0u} without further mentioning.}
Following \cite[\S 6.2.2]{FrPo} we denote by
$\f^{\perp,\pm}(\x)$ the $\R^2$--valued function such that $\langle\f^{\perp, \pm}(\x), \f^{\pm}(\x)\rangle=0$,
$\|\f^{\perp,\pm}(\x)\|=\| \f^{\pm}(\x)\|$,
i.e.~$\f^{\perp,\pm}(\x)$ is the $C^r$ function in $\Om^{\pm}$ obtained by
  rotating  $\f^{\pm}(\x)$ by $\pi/2$. We choose the orientation
 in such a way that $\ga(t)+c\f^{\perp,\pm}(\ga(t)) \in E^{\textrm{in}}$ if $c>0$ is small enough.

Then we set
\begin{align*}
	\mathcal{K}:= 2 \max\{ &\sup \{ \|\g(t,\x,0)\| /\| \f^+(\x)\| \mid \x \in B(\bs{\Gamma},1) \cap(\Om^+\cup\Om^0), \, t \in
\R\},\\
	&\sup \{ \|\g(t,\x,0)\| /\| \f^-(\x)\| \mid \x \in B(\bs{\Gamma},1) \cap(\Om^-\cup\Om^0), \, t \in \R\}\}.
\end{align*}
  Observe that $\mathcal{K}>0$ is bounded, see \cite[Page 1457]{FrPo}.
 Then we denote by
    \begin{equation}\label{finout}
    \begin{aligned}
      \f_{a}^\pm(\x) &= \f^\pm(\x)+  \ep  \mathcal{K} \f^{\perp,\pm}(\x), \\
          \f_{b}^\pm(\x) &= \f^\pm(\x) -  \ep   \mathcal{K} \f^{\perp,\pm}(\x).
    \end{aligned}
    \end{equation}
    The  curves  $Z^{\textrm{fwd,in}}$, $Z^{\textrm{fwd,out}}$, $Z^{\textrm{bwd,in}}$, $Z^{\textrm{bwd,out}}$ are in fact obtained as orbits
    of trajectories of the autonomous
    systems $\dot{\x}= \f_{a}^\pm(\x)$
and $\dot{\x}= \f_{b}^\pm(\x)$.
From now on we denote with the subscript ${a}$ quantities referred to the former system and with the
subscript ${b}$ quantities referred to the latter system.
In particular we denote by $\y_{a}(t,\P)$ and $\y_{b}(t,\P)$, resp.,  the
 trajectory of $\dot{\y}= \f_{a}^\pm(\y)$
and of  $\dot{\y}= \f_{b}^\pm(\y)$ leaving from $\P$ at $t=0$.
Notice that $\f_{a}^\pm(\x)$ on $ \bs{\Gamma}$ aims towards  $E^\inn$ while
$\f_{b}^\pm(\x)$ on $\bs{\Gamma}$ aims towards  $E^\outt$.

Note that the origin is still a saddle for both the systems so it admits unstable and stable leaves
$W^u_{a}$, $W^u_{b}$, $W^s_{a}$, $W^s_{b}$. Due to the discontinuity $W^u_{a}$ (and similarly for
 $W^u_{b}$, $W^s_{a}$, $W^s_{b}$)  has a corner in the origin
and it may be split on a component $W^{u,-}_{a}$ entering $\Om^-$, the one we are interested in, and another, say $W^{u,+}_{a}$ entering
$\Om^+$.

If we follow $W^u_{a}$, $W^s_{a}$, $W^u_{b}$, $W^s_{b}$ from the origin towards $L^0$, from a continuity argument,  we see that
they intersect $L^0$ transversely the first time in points denoted respectively by
$\vec{\zeta}^u_{a}$, $\vec{\zeta}^s_{a}$, $\vec{\zeta}^u_{b}$, $\vec{\zeta}^s_{b}$.

Further from \cite[Lemma 6.9]{FrPo} we get the following.
  \begin{remark}\label{zeasy}
        There   are     positive constants  $c^u_{a}, c^u_{b},c^s_{a}, c^s_{b} $ such that
     $$\begin{array}{cc}
         \vec{\zeta}^u_{a}= \ga(0) + \left(c^u_{a} \ep +o(\ep)\right)\vec{w} , & \vec{\zeta}^u_{b}= \ga(0) - \left(c^u_{b} \ep
         +o(\ep)\right)\vec{w},  \\
         \vec{\zeta}^s_{a}= \ga(0) - \left(c^s_{a} \ep +o(\ep)\right)\vec{w} , & \vec{\zeta}^s_{b}= \ga(0) +\left(c^s_{b} \ep
         +o(\ep)\right)\vec{w}
       \end{array}$$
       where $\vec{w}$ is a normalized vector, tangent to $\Om^0$ in $\ga(0)$ and oriented so that $\ga(0)+c \vec{w} \in E^\inn$ for $c>0$
       small enough.
\end{remark}

\smallskip

 \begin{figure}[t]
	\begin{center}
		\epsfig{file=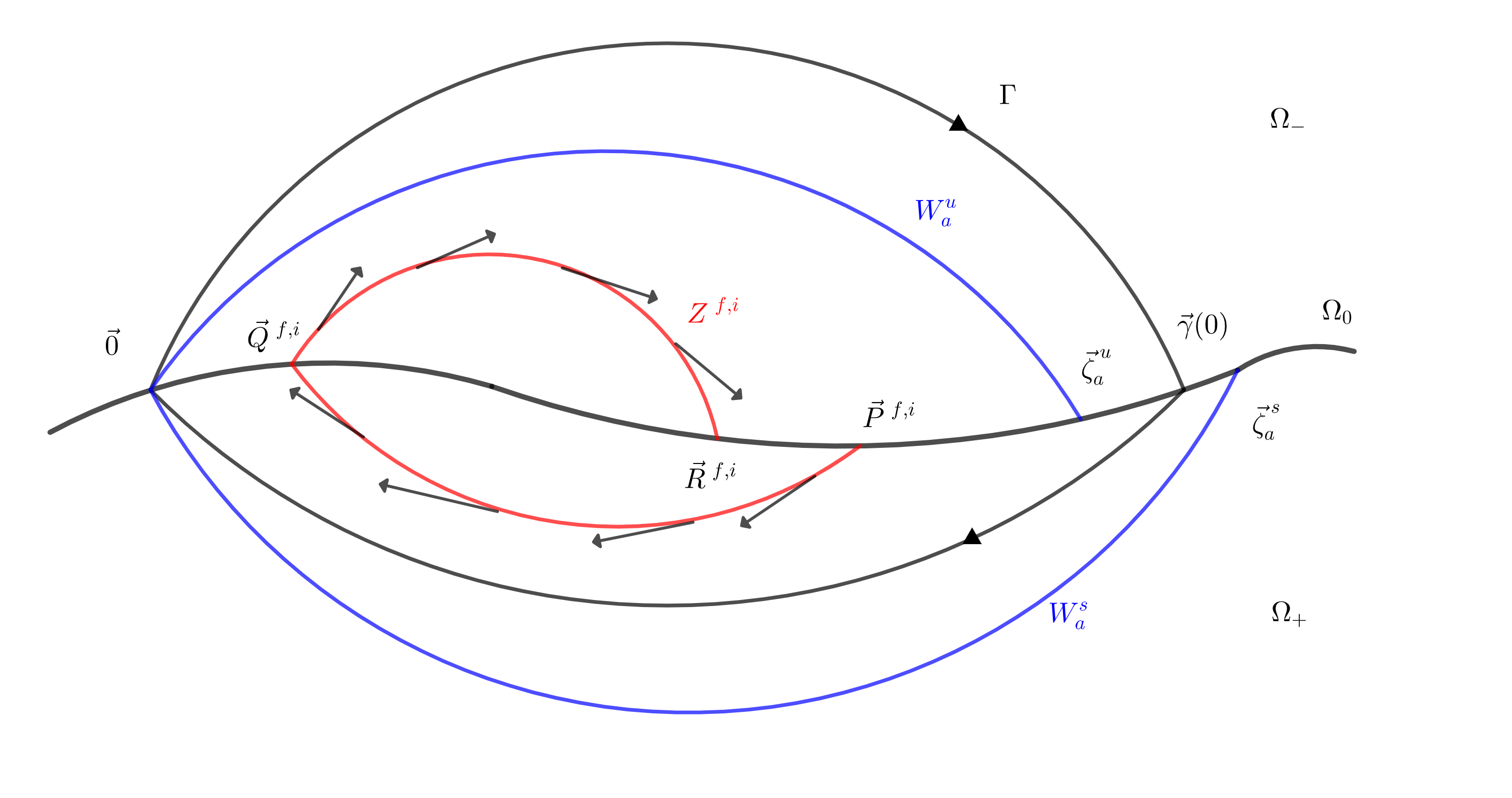, width=0.8\textwidth}
		\caption{$Z^{\textrm{fwd},\inn}$ constructed via Lemma \ref{forZui}.}
		\label{fig-Zui}
	\end{center}
\end{figure}

\noindent
 Let us start by a sketch of the proof of Lemma \ref{lemma0s}.

Let $\P^{\textrm{fwd,in}}(\Di)$ be the unique point such that
$$\P^{\textrm{fwd,in}}(\Di) \in L^0, \qquad \|\P^{\textrm{fwd,in}}(\Di)- \zz^s_{a} \|=\Di>0,
\qquad \langle\P^{\textrm{fwd,in}}(\Di)- \zz^s_{a}, \vec{w}\rangle  >0$$
 where $\Di>0$ is small enough.
The following result  is the key ingredient to construct the curve $Z^{\textrm{fwd,in}}$ (see Figure \ref{fig-Zui}).
\begin{lemma}\label{forZui}
  There is $\Delta>0$ such that for any $0<\Di \le \Delta$ there are $0<T^{u,1}_{a}(\Di)<T^{u,2}_{a}(\Di)$
  such that $\y_{a}(t,\P^{\textrm{fwd,in}}(\Di)) \in  (\Om^+ \cap E^\inn)$ for any $0<t<T^{u,1}_{a}(\Di)$,
  $\y_{a}(t,\P^{\textrm{fwd,in}}(\Di)) \in (\Om^- \cap E^\inn)$ for any $T^{u,1}_{a}(\Di)<t<T^{u,2}_{a}(\Di)$,
  and it crosses transversely $L^\inn$ at $t= T^{u,1}_{a}(\Di)$ in
  $\Q^{\textrm{fwd,in}}(\Di):=\y_{a}(T^{u,1}_{a}(\Di),\P^{\textrm{fwd,in}}(\Di))$
  and $L^0$ at $t= T^{u,2}_{a}(\Di)$ in $\vec{R}^{\textrm{fwd,in}}(\Di):=\y_{a}(T^{u,2}_{a}(\Di),\P^{\textrm{fwd,in}}(\Di))$.
  Further, for any $0< \mu \le \mu_0$, any $0< \Di \le \Delta$ and $0<\ep\le \ep_0$ we find
  \begin{equation}\label{useless1}
\begin{gathered}
         \Di^{\sfwd_++\mu}\le \|\Q^{\textrm{fwd,in}}(\Di)\| \le \Di^{\sfwd_+-\mu},\\
       \Di^{\sfwd +\mu} \le \|\vec{R}^{\textrm{fwd,in}}(\Di)- \zz^u_{a}\|
       \le \Di^{\sfwd-\mu} .
  \end{gathered}
  \end{equation}
\end{lemma}
The proof of this lemma is obtained by adapting the argument in \cite[\S 6.2.2]{FrPo}
 which derives from the estimates in \cite[Lemma 6.4]{FrPo}.

  \begin{remark}\label{R.forZui}
   Note that from Remark \ref{zeasy} and Lemma \ref{useless1} we find
     \begin{equation}\label{useless2}
       \Di^{\sfwd+\mu}-\bar{c}\ep \le \|\vec{R}^{\textrm{fwd,in}}(\Di)- \ga(0) \| \le \Di^{\sfwd-\mu}+ \bar{c} \ep
  \end{equation}
  where $\bar{c}:= 2 \max \{c^u_{a}, c^u_{b}, c^s_{a}, c^s_{b} \}$.
 \end{remark}

  Let  $\beta \ge \ep^{\sigma^{\mathrm{fb}}/2}$ be as in \eqref{def-sifb} and recall
  that
   $\sigma^{\mathrm{fb}} < 1$; it is easy to check that there is $\bar{\Di}^{\textrm{f,i}} \in [\beta- 2 c^s_a  \ep , \beta-
   \frac{c^s_a}{2} \ep]$ such that
   $\vec{P}^{\textrm{fwd},\inn}:=\vec{P}^{\textrm{fwd},\inn}(\bar{\Di}^{\textrm{f,i}})$ satisfies
   $\|\vec{P}^{\textrm{fwd},\inn}- \ga(0)\| = \beta$ and $\vec{P}^{\textrm{fwd},\inn}\in E^{\inn}$.

 Now the curve $Z^{\textrm{fwd},\inn}$ is defined as follows
  \begin{equation*}
Z^{\textrm{fwd},\inn}:= \{ \y_{a}(t,\P^{\textrm{fwd,in}}) \mid 0 \le t \le T^{u,2}_{a}(\bar{\Di}^{\textrm{f,i}}) \},
  \end{equation*}
and we set $\vec{Q}^{\textrm{fwd},\inn}:=\vec{Q}^{\textrm{fwd},\inn}(\bar{\Di}^{\textrm{f,i}})$,
$\vec{R}^{\textrm{fwd},\inn}:=\vec{R}^{\textrm{fwd},\inn}(\bar{\Di}^{\textrm{f,i}})$.

  In fact by construction the flow of \eqref{eq-disc} on $Z^{\textrm{fwd},\inn}$ points towards the  exterior
  of the bounded set enclosed by $Z^{\textrm{fwd},\inn}$ and the segment between $\P^{\textrm{fwd,in}}$ and
  $\vec{R}^{\textrm{fwd,in}}$, see Figure \ref{fig-Zui}.
  Now the estimates concerning  $\P^{\textrm{fwd,in}}$,
  $\Q^{\textrm{fwd,in}}$,
  $\vec{R}^{\textrm{fwd,in}}$
  in Lemma  \ref{lemma0s}   follow from Lemma \ref{forZui} and Remark \ref{R.forZui}.

 \smallskip

\noindent
Let $\P^{\textrm{fwd,out}}(\Di)$ be the unique point such that
$$\P^{\textrm{fwd,out}}(\Di) \in L^0, \qquad
\|\P^{\textrm{fwd,out}}(\Di)- \zz^u_{b} \|=\Di>0, \qquad  \langle\P^{\textrm{fwd,out}}(\Di)- \zz^u_{b}, \vec{w}\rangle \, <0$$
 where $\Di>0$ is small enough.
Rephrasing the argument of \cite[\S 6.2.2]{FrPo} which is based on \cite[Lemma 6.7]{FrPo}, we get the following result
needed to construct $Z^{\textrm{fwd},\outt}$ (see
Figure \ref{fig-Zuo}):
 \begin{lemma}\label{forZuo}
  There is $\Delta>0$ such that for any $0<\Di<\Delta$ there are $\tau^{u,-1}_{b}(\Di)<0<\tau^{u,1}_{b}(\Di)$
  such that $\y_{b}(t,\P^{\textrm{fwd,out}}(\Di)) \in (\Om^- \cap E^\outt)$ for any $\tau^{u,-1}_{b}(\Di) \le t<0$,
  $\y_{b}(t,\P^{\textrm{fwd,out}}(\Di)) \in (\Om^+ \cap E^\outt)$ for any $0<t \le\tau^{u,1}_{b}(\Di)$, and it crosses transversely
  $L^{-,\outt}$ at $t=\tau^{u,-1}_{b}(\Di)$
  in $\vec{O}^{\textrm{fwd,out}}(\Di):=\y_{b}(\tau^{u,-1}_{b}(\Di),\P^{\textrm{fwd,out}}(\Di))$ and $L^{+,\outt}$ at $t=\tau^{u,1}_{b}(\Di)$
  in $\vec{Q}^{\textrm{fwd,out}}(\Di):=\y_{b}(\tau^{u,1}_{b}(\Di),\P^{\textrm{fwd,out}}(\Di))$.
Further, for any $0< \mu \le \mu_0$, any $0< \Di \le \Delta$ and $0<\ep\le \ep_0$ we find
  \begin{equation}\label{useless3}
  \begin{gathered}
\Di^{\sbwd_-+\mu}
  	\le \|\vec{O}^{\textrm{fwd,out}}(\Di)\| \le \Di ^{\sbwd_--\mu} ,
   	\\
   	[\Di+(c_{b}^s+c_{b}^u)\ep]^{\sfwd_++\mu} \le \|\vec{Q}^{\textrm{fwd,out}}(\Di)\| \le
   		[\Di+(c_{b}^s+c_{b}^u)\ep]^{\sfwd_+ -\mu}.
  \end{gathered}
  \end{equation}
\end{lemma}

Again, there is $\bar{\Di}^{\textrm{f,o}} \in [\beta- 2 c^u_b  \ep , \beta- \frac{c^u_b}{2} \ep]$ such that
   $\vec{P}^{\textrm{fwd},\outt}:=\vec{P}^{\textrm{fwd},\outt}(\bar{\Di}^{\textrm{f,o}})$ satisfies
   $\|\vec{P}^{\textrm{fwd},\outt}- \ga(0)\| = \beta$ and $\vec{P}^{\textrm{fwd},\outt}\in E^{\outt}$.

Then the curve
  $Z^{\textrm{fwd},\outt}$ is defined as follows
    \begin{equation*}
Z^{\textrm{fwd},\outt}:= \{ \y_{b}(t,\P^{\textrm{fwd,out}})\mid \tau^{u,-1}_{b}(\bar{\Di}^{\textrm{f,o}})
\le t \le
  \tau^{u,1}_{b}(\bar{\Di}^{\textrm{f,o}}) \},
  \end{equation*}
and we set $\vec{O}^{\textrm{fwd},\outt}:=\vec{O}^{\textrm{fwd},\outt}(\bar{\Di}^{\textrm{f,o}})$,
$\vec{Q}^{\textrm{fwd},\outt}:=\vec{Q}^{\textrm{fwd},\outt}(\bar{\Di}^{\textrm{f,o}})$.

\begin{figure}[t]
 	\begin{center}
 		\epsfig{file=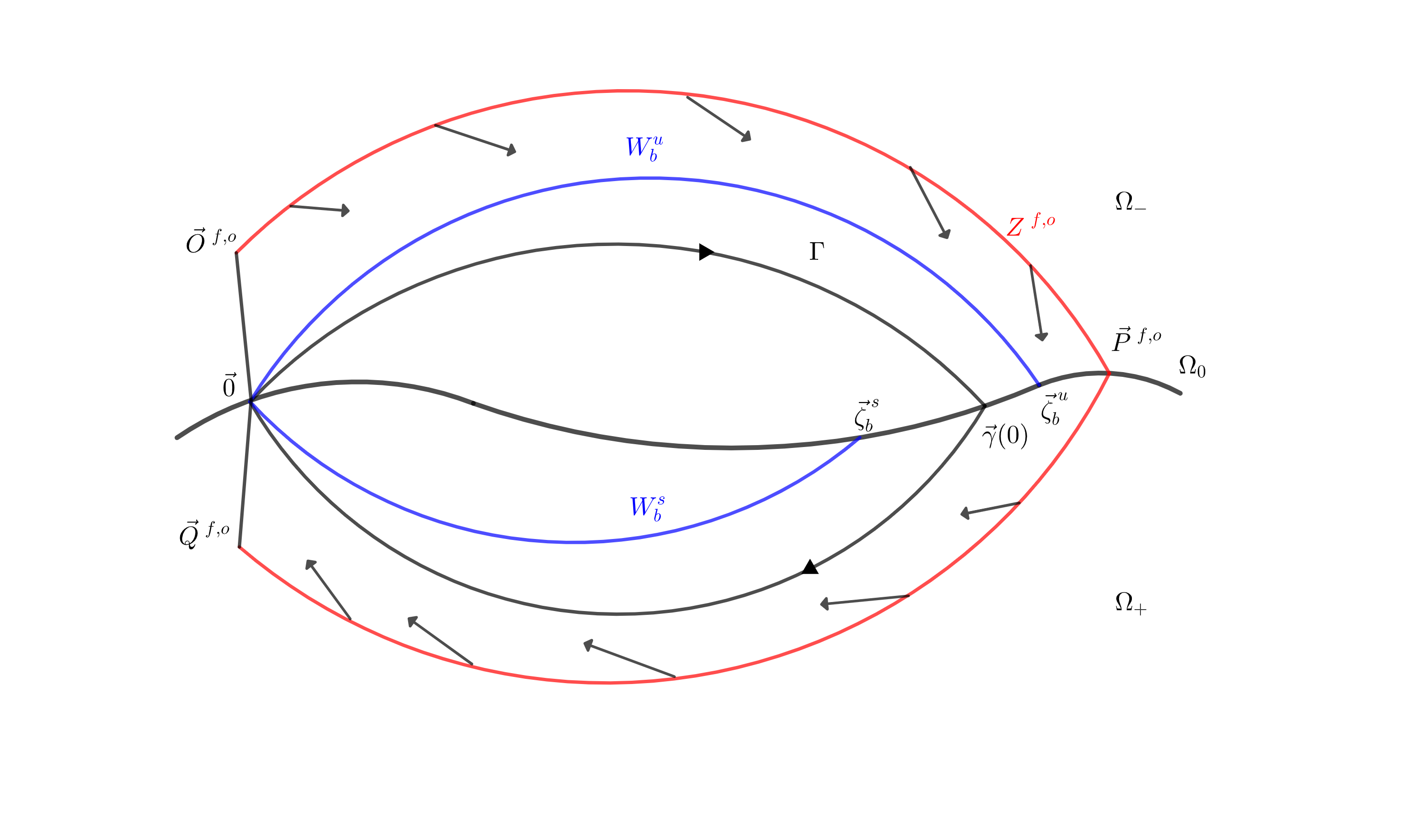, width=0.8\textwidth}
 		\caption{$Z^{\textrm{fwd},\outt}$ constructed in Lemma \ref{forZuo}.}
 		\label{fig-Zuo}
 	\end{center}
\end{figure}

  In fact by construction the flow of \eqref{eq-disc} on $Z^{\textrm{fwd},\outt}$ points towards the  interior
  of the bounded set enclosed by $Z^{\textrm{fwd},\outt}$ and the segments $L^{-,\outt}$ between $\vec{0}$ and
  $\vec{O}^{\textrm{fwd,out}}$, and
  $L^{+,\outt}$ between $\vec{0}$ and
  $\vec{Q}^{\textrm{fwd,out}}$.
  Further the estimates concerning $\vec{O}^{\textrm{fwd,out}}$,
  $\P^{\textrm{fwd,out}}$, $\vec{Q}^{\textrm{fwd,out}}$
  in Lemma \ref{lemma0s}  follow from Lemma \ref{forZuo}.
  So Lemma \ref{lemma0s} is proved.

\smallskip

\noindent

\begin{figure}[t]
	\begin{center}
		\epsfig{file=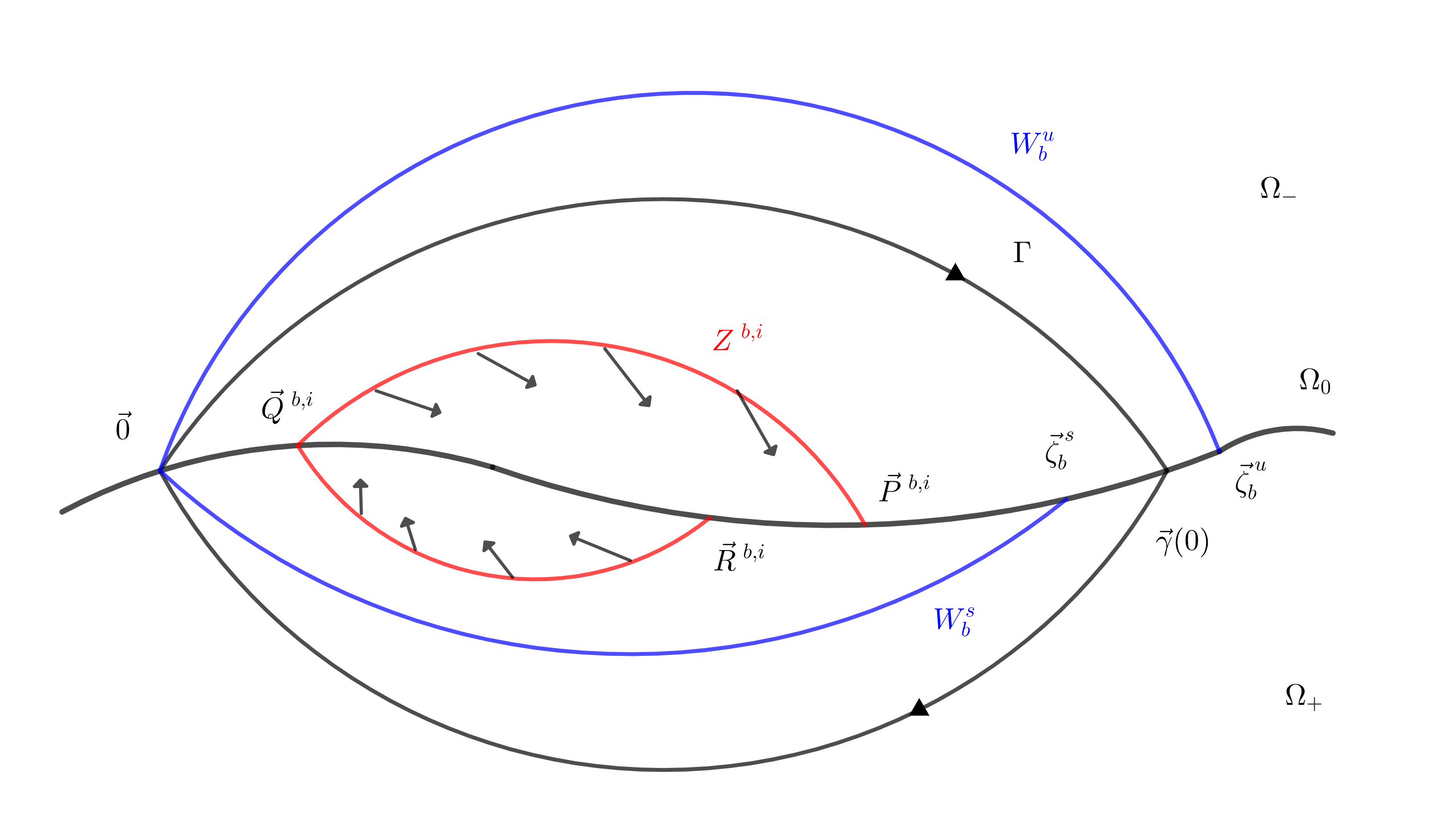, width=0.8\textwidth}
		\caption{$Z^{\textrm{bwd},\inn}$ constructed in Lemma \ref{forZsi}.}
		\label{fig-Zsi}
	\end{center}
\end{figure}

The construction of the curve $Z^{\textrm{bwd},\inn}$ (see Figure \ref{fig-Zsi}) is analogous to $Z^{\textrm{fwd},\inn}$.
Namely, let   $\vec{P}^{\textrm{bwd,in}}(\Di)$ be the unique point such that
$$ \vec{P}^{\textrm{bwd,in}}(\Di) \in L^0, \qquad
\| \vec{P}^{\textrm{bwd,in}}(\Di)- \zz^u_{b} \|=\Di>0, \qquad  \langle \vec{P}^{\textrm{bwd,in}}(\Di)- \zz^u_{b}, \vec{w}\rangle
>0.$$
 Reasoning as in Lemma \ref{forZui} we get the following.
 \begin{lemma}\label{forZsi}
  There is $\Delta>0$ such that for any $0<\Di \le \Delta$ there are $T^{s,-2}_{b}(\Di)<T^{s,-1}_{b}(\Di)<0$
  such that $\y_{b}(t,\vec{P}^{\textrm{bwd,in}}(\Di)) \in (\Om^-\cap E^\inn)$ for any $T^{s,-1}_{b}(\Di)< t<0$,
  $\y_{b}(t, \vec{P}^{\textrm{bwd,in}}(\Di)) \in (\Om^+ \cap E^\inn)$ for any $T^{s,-2}_{b}(\Di)<t<T^{s,-1}_{b}(\Di)$.
  Let us set $\vec{Q}^{\textrm{bwd,in}}(\Di):=\y_{b}(T^{s,-1}_{b}(\Di),\vec{P}^{\textrm{bwd,in}}(\Di))$
  and   $\vec{R}^{\textrm{bwd,in}}(\Di):=\y_{b}(T^{s,-2}_{b}(\Di),\vec{P}^{\textrm{bwd,in}}(\Di))$.
Then, for any $0< \mu \le \mu_0$, for any $0< \Di \le \Delta$ and $0<\ep\le \ep_0$ we find
  \begin{equation}\label{useless4}
  \begin{gathered}
     \Di^{\sbwd_- +\mu}  \le \|\vec{Q}^{\textrm{bwd,in}}(\Di)\| \le \Di^{\sbwd_--\mu} , \\
    \Di^{\sbwd+\mu}   \le \| \vec{R}^{\textrm{bwd,in}}(\Di)- \zz^s_{b}\| \le \Di^{\sbwd-\mu}.
  \end{gathered}
  \end{equation}
\end{lemma}
\begin{remark}\label{R.forZui2}
	Analogously to Remark \ref{R.forZui} we see that
	 $$ \Di^{\sbwd +\mu}  -\bar{c}\ep\leq \|\vec{R}^{\textrm{bwd,in}}(\Di)- \ga(0)\| \le \Di^{\sbwd -\mu}+\bar{c}\ep .$$
\end{remark}

Again we see that there is $\bar{\Di}^{\textrm{b,i}} \in [\beta+\frac{c^u_b}{2}\ep, \beta+2c^u_b\ep]$ such that
   $\vec{P}^{\textrm{bwd},\inn}:=\vec{P}^{\textrm{bwd},\inn}(\bar{\Di}^{\textrm{b,i}})$ satisfies
   $\|\vec{P}^{\textrm{bwd},\inn}- \ga(0)\| = \beta$ and $\vec{P}^{\textrm{bwd},\inn}\in E^{\inn}$.

So the curve $Z^{\textrm{bwd},\inn}$ is defined as follows
    \begin{equation*}
Z^{\textrm{bwd},\inn}:= \{ \y_{b}(t,\vec{P}^{\textrm{bwd,in}}) \mid T^{s,-2}_{b}(\bar{\Di}^{\textrm{b,i}}) \le t
\le 0
\},
  \end{equation*}
and we set  $\vec{Q}^{\textrm{bwd},\inn}:=\vec{Q}^{\textrm{bwd},\inn}(\bar{\Di}^{\textrm{b,i}})$ and
$\vec{R}^{\textrm{bwd},\inn}:=\vec{R}^{\textrm{bwd},\inn}(\bar{\Di}^{\textrm{b,i}})$.

 In fact by construction the flow of \eqref{eq-disc} on
$Z^{\textrm{bwd},\inn}$ aims toward the  interior of the bounded set enclosed by $Z^{\textrm{bwd},\inn}$ and the
segment
between
$\vec{P}^{\textrm{bwd,in}}$ and  $\vec{R}^{\textrm{bwd,in}}$.
Further the estimates concerning
  $\vec{P}^{\textrm{bwd,in}}$,
   $\Q^{\textrm{bwd,in}}$,
   $\vec{R}^{\textrm{bwd,in}}$
  in Lemma   \ref{lemma0u}    follow from Lemma \ref{forZsi} and Remark~\ref{R.forZui2}.

\begin{figure}[t]
	\begin{center}
		\epsfig{file=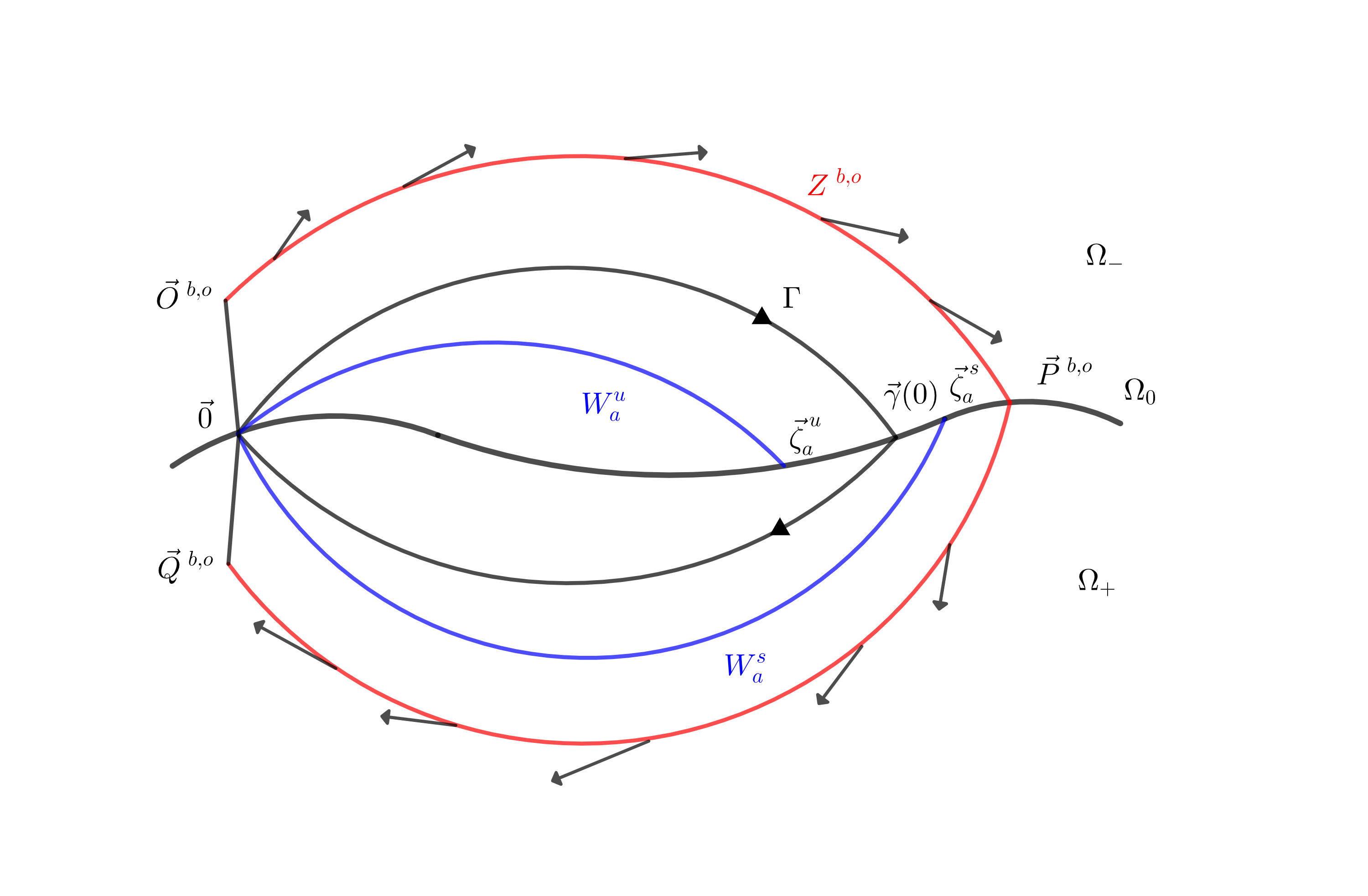, width=0.8\textwidth}
		\caption{$Z^{\textrm{bwd,out}}$ constructed in Lemma \ref{forZso}.}
		\label{fig-Zso}
	\end{center}
\end{figure}

\indent
Now, let $\P^{\textrm{bwd,out}}(\Di)$ be the unique point such that
$$\P^{\textrm{bwd,out}}(\Di) \in L^0, \qquad
\|\P^{\textrm{bwd,out}}(\Di)- \zz^{s}_{a} \|=\Di>0, \qquad  \langle\P^{\textrm{bwd,out}}(\Di)- \zz^{s}_{a}, \vec{w}\rangle <0$$
 where $\Di>0$ is small enough.
 Reasoning as in Lemma \ref{forZuo} we get $Z^{\textrm{bwd,out}}$ (see Figure \ref{fig-Zso}):
 \begin{lemma}\label{forZso}
  There is $\Delta>0$ such that for any $0<\Di \le \Delta$ there are $\tau^{s,-1}_{a}(\Di)<0<\tau^{s,1}_{a}(\Di)$
  such that $\y_{a}(t,\P^{\textrm{bwd,out}}(\Di)) \in (\Om^- \cap E^\outt)$ for any $\tau^{s,-1}_{a}(\Di) \le t<0$,
  $\y_{a}(t,\P^{\textrm{bwd,out}}(\Di)) \in (\Om^+ \cap E^\outt)$ for any $0<t \le\tau^{s,1}_{a}(\Di)$,
  and it crosses transversely $L^{-,\outt}$ at $t=\tau^{s,-1}_{a}(\Di)$ in
  $\vec{O}^{\textrm{bwd,out}}(\Di):=\y_{a}(\tau^{s,-1}_{a}(\Di),\P^{\textrm{bwd,out}}(\Di))$ and
  $L^{+,\outt}$ at $t=\tau^{s,1}_{a}(\Di)$
  in $\Q^{\textrm{bwd,out}}(\Di):=\y_{a}(\tau^{s,1}_{a}(\Di),\P^{\textrm{bwd,out}}(\Di))$.
Further, \michal{for any $0<\mu\leq\mu_0$,} any $0< \Di \le \Delta$ and $0<\ep\le \ep_0$ we find
  \begin{equation}\label{useless5}
  \begin{gathered}
    \left[\Di+(c_{a}^s+c_{a}^u)\ep\right]^{\sbwd_- +\mu}      \le \|\vec{O}^{\textrm{bwd,out}}(\Di)\| \le \left[\Di+(c_{a}^s+c_{a}^u)\ep\right]^{\sbwd_- -\mu}  , \\
   \Di^{\sfwd_+ +\mu} \le \|\Q^{\textrm{bwd,out}}(\Di)\| \le
     \Di^{\sfwd_+ -\mu}   .
  \end{gathered}
  \end{equation}
\end{lemma}

Again, there is $\bar{\Di}^{\textrm{b,o}} \in [\beta- 2 c^s_a  \ep , \beta- \frac{c^s_a}{2} \ep]$ such that
   $\vec{P}^{\textrm{bwd},\outt}:=\vec{P}^{\textrm{bwd},\outt}(\bar{\Di}^{\textrm{b,o}})$ satisfies
   $\|\vec{P}^{\textrm{bwd},\outt}- \ga(0)\| = \beta$ and $\vec{P}^{\textrm{bwd},\inn}\in E^{\outt}$.

Then the curve
  $Z^{\textrm{bwd},\outt}$ is defined as follows
    \begin{equation*}
Z^{\textrm{bwd},\outt}:= \{ \y_{a}(t,\P^{\textrm{bwd,out}}(\bar{\Di}^{\textrm{b,o}}) \mid \tau^{s,-1}_{a}(\bar{\Di}^{\textrm{b,o}})
\le t \le
  \tau^{s,1}_{a}(\bar{\Di}^{\textrm{b,o}}) \},
  \end{equation*}
and we set $\vec{O}^{\textrm{bwd},\outt}:=\vec{O}^{\textrm{bwd},\outt}(\bar{\Di}^{\textrm{b,o}})$,
$\vec{Q}^{\textrm{bwd},\outt}:=\vec{Q}^{\textrm{bwd},\outt}(\bar{\Di}^{\textrm{b,o}})$.

  The flow of \eqref{eq-disc} on $Z^{\textrm{bwd,out}}$ aims towards the  exterior
  of the bounded set enclosed by $Z^{\textrm{bwd,out}}$ and the segments $L^{-,\outt}$ between $\vec{0}$ and
  $\vec{O}^{\textrm{bwd,out}}$, and
   $L^{+,\outt}$ between $\vec{0}$ and
  $\Q^{\textrm{bwd,out}}$, see Figure \ref{fig-Zso}.
	Moreover, the estimates concerning $\vec{O}^{\textrm{bwd,out}}$,
	$\P^{\textrm{bwd,out}}$, $\Q^{\textrm{bwd,out}}$
	 follow from Lemma \ref{forZso}.
  So Lemma \ref{lemma0u} is proved.

\end{document}